\documentclass[11pt,letterpaper]{amsart}
\usepackage[foot]{amsaddr}
\usepackage{pstricks}

\usepackage[T1]{fontenc}
\usepackage[latin9]{inputenc}
\usepackage{geometry}
\usepackage{wrapfig}
\usepackage{multicol}
\usepackage{enumerate}   
\usepackage{graphicx}
\usepackage{soul}
\usepackage{xcolor}
\usepackage{amssymb}
\usepackage{relsize}

\newtheorem{theorem}{Theorem}[section]
\newtheorem{lemma}[theorem]{Lemma}
\newtheorem{proposition}[theorem]{Proposition}

{ \theoremstyle{definition}
\newtheorem{definition}[theorem]{Definition}}
{ \theoremstyle{remark}
\newtheorem{remark}[theorem]{Remark}}
\usepackage{placeins}
\usepackage{cite}
\usepackage{caption}
\usepackage{enumerate}
\usepackage{afterpage}
\usepackage{enumitem}
\usepackage{bmpsize}
\usepackage{hyperref}
\usepackage{tabu}
\usepackage{enumitem}   
\numberwithin{equation}{section}
\usepackage{stmaryrd}
\usepackage{tikz}
\usetikzlibrary{matrix,graphs,arrows,positioning,calc,decorations.markings,shapes.symbols}
\definecolor{dullmagenta}{rgb}{0.4,0,0.4}   % #660066
\definecolor{darkblue}{rgb}{0,0,0.4}

\def\i{ \mathsf{i}}

\def\P {{\mathbb{P}^{\theta,M}_{N,k}}}

\def\I {{I}}

\def\bX {{\bf X}}

\def\XX {{\mathfrak{X}^{\theta,M}_{N,k}}}

\def\R {{\textup{Res}}}

\def\XXX {{\mathfrak{X}^{\theta,M}_{N,1}}}

\def\K{\kappa}
\def\SS{\mathfrak S}

\def\tb {{\bf t}}
\def\vm {{\bf v}}

\def\Sp{A_{+}}
  \def\Sm{A_{-}}
\def\b{\mathfrak{b}^{i,s}_1}
\def\Tp{B_{+}}
  \def\Tm{B_{-}}
\def\bb{\mathfrak{b}^{i,s}_2}
\def\bn{\mathfrak{b}^{i,s}}

\psset{unit=1pt}
\psset{arrowsize=4pt 1}
\psset{linewidth=1pt}
\newrgbcolor{mygreen}{.2 .7 .2}
\newrgbcolor{myred}{.7 .3 .2}
\newgray{mygray}{.90}
\countdef\x=23
\countdef\y=24
\countdef\z=25
\countdef\t=26

\def\tbox(#1,#2)#3{
\x=#1 \y=#2 
\multiply\x by 12 
\multiply\y by 12 
\z=\x \t=\y
\advance\z by 12 
\advance\t by 12 
\psline(\x,\y)(\x,\t)(\z,\t)(\z,\y)(\x,\y)
\advance\x by 6
\advance\y by 6 
\rput(\x,\y){{\bf #3}}}

\def\gbox(#1,#2)#3{
\x=#1 \y=#2 
\multiply\x by 12 
\multiply\y by 12 
\z=\x \t=\y
\advance\z by 12 
\advance\t by 12 
\psline[linecolor=gray, linewidth=0.5pt](\x,\y)(\x,\t)(\z,\t)(\z,\y)(\x,\y)
\advance\x by 6
\advance\y by 6 
\rput(\x,\y){\gray{#3}}}

\def\ebox(#1,#2)#3{
\x=#1 \y=#2 
\multiply\x by 12 
\multiply\y by 12 
\advance\x by 6
\advance\y by 6 
\rput(\x,\y){#3}}

\def\ggbox(#1,#2){
\x=#1 \y=#2 
\multiply\x by 12 
\multiply\y by 12 
\z=\x \t=\y
\advance\z by 12 
\advance\t by 12 
\psframe[fillstyle=solid, fillcolor=mygray, linewidth=0pt](\x,\y)(\z,\t)
\psline[linecolor=gray, linewidth=0.5pt](\x,\y)(\x,\t)(\z,\t)(\z,\y)(\x,\y)}

\def\tline(#1,#2)(#3,#4){
\x=#1 \y=#2 \z=#3 \t=#4
\multiply\x by 12 
\multiply\y by 12 
\multiply\z by 12 
\multiply\t by 12 
\psline(\x,\y)(\z,\t)}

\title{Multi-level loop equations for $\beta$-corners processes}

\author[E. Dimitrov]{Evgeni Dimitrov}
\address{E. Dimitrov, Department of Mathematics, USC, Los Angeles, CA 90089} 
\email{edimitro52@gmail.com}
\author[A. Knizel]{Alisa Knizel}
\address{A. Knizel,  Department of Mathematics, Barnard College, New York, NY 10027}
\email{aknizel@barnard.edu}

\begin{document}

\maketitle 

\begin{abstract}
The goal of the paper is to introduce a new set of tools for the study of discrete and continuous $\beta$-corners processes. In the continuous setting, our work provides a multi-level extension of the loop equations (also called Schwinger-Dyson equations) for $\beta$-log gases obtained by Borot and Guionnet in (Commun. Math. Phys. 317, 447-483, 2013). In the discrete setting, our work provides a multi-level extension of the loop equations (also called Nekrasov equations) for discrete $\beta$-ensembles obtained by Borodin, Gorin and Guionnet in (Publications math{\' e}matiques de l'IH{\' E}S 125, 1-78, 2017).
\end{abstract}

\tableofcontents

%-------------------------------------------------------------------------------------------------------------------------------------------------------------------------------------------------
% Section 1
%
%-------------------------------------------------------------------------------------------------------------------------------------------------------------------------------------------------
\section{Introduction and main result}\label{Section1}

%-------------------------------------------------------------------------------------------------------------------------------------------------------------------------------------------------
% Section 1.1
%
%-------------------------------------------------------------------------------------------------------------------------------------------------------------------------------------------------
\subsection{Preface}\label{Section1.1}
 Fix $N \in \mathbb{N}$, $a_- \in [-\infty, \infty)$, $a_+ \in (-\infty, \infty]$ with $a_- < a_+$, $\beta > 0$ and a continuous function $V$ on $\mathbb{R}$. A {\em continuous $\beta$-corners process} is a probability distribution with density
\begin{equation}\label{S1Corners1}
\begin{split}
&f(y) =  (Z_N^c)^{-1} \cdot {\bf 1}_{\mathcal{G}_N}  {\bf 1}\{ a_- < y^N_1, y^N_N < a_+\} \cdot  g_N(y) \mbox{, where } y = (y^N, \dots, y^1) \in \mathbb{R}^N \times \cdots \times \mathbb{R},\\
&g_N(y) = \prod_{1 \leq i<j \leq N} \hspace{-2mm} (y_j^N-y_i^N)  \hspace{-1mm}  \prod_{j = 1}^{N-1}\left(\prod_{1\leq a<b \leq j}
(y_b^j-y_a^j)^{2-\beta} \prod_{a=1}^j\prod_{b=1}^{j+1}|y_a^j-y_b^{j+1}|^{\frac{\beta}{2}-1}\right)  \prod_{i = 1}^N e^{-\frac{N \beta}{2} V(y^N_i)}, \\
& \mathcal{G}_N = \{ y = (y^N, \dots, y^1) \in \mathbb{R}^N \times \cdots \times \mathbb{R}:  y_{i}^{j}<y_i^{j-1}<y_{i+1}^{j}, \mbox{$j = 1,\dots, N$, $i = 1, \dots, j$} \},
\end{split}
\end{equation}
and $Z_N^c$ is a normalization constant. Measures of the form (\ref{S1Corners1}) are of special interest to random matrix theory as they describe the joint distribution of eigenvalues of classical random matrix ensembles. For example, when $a_- = -\infty$, $a_+ = \infty$ and $V(y) = y^2/2$ the measure with density (\ref{S1Corners1}) is called the {\em Hermite $\beta$-corners process}, introduced in \cite{GS}. For $\beta = 1,2,4$ the Hermite $\beta$-corners process is the distribution of the set of eigenvalues of the top-left $k \times k$ corner $[M_{ij} ]^k_{i,j=1}$ for $k  = 1, \dots, N$, where $M = [M_{ij}]^N_{i,j=1}$ is a random matrix from the Gaussian Orthogonal ($\beta = 1$)/ Unitary ($\beta = 2$)/ Symplectic ($\beta = 4$) Ensemble, see \cite[Section 9]{NajVir} for a proof.

The measures in (\ref{S1Corners1}) enjoy many remarkable properties. For example, if we fix $k \in \{1, \dots, N\}$ and consider the pushforward measure on $(y^N, \dots, y^k) \in \mathbb{R}^N \times \cdots \times \mathbb{R}^k$, then it has density
\begin{equation}\label{S1Corners2}
\begin{split}
&f_{N,k}(y^N, \dots, y^k) =  (Z_N^c)^{-1} \cdot {\bf 1}_{\mathcal{G}_{N,k}} \cdot  {\bf 1}\{ a_- < y^N_1, y^N_N < a_+\} \cdot  g_{N,k}(y^N, \dots, y^k) \mbox{, where }\\
&g_{N,k}(y^N, \dots, y^k) = \prod_{j = 1}^{k} \frac{\Gamma(\beta/2)^j}{\Gamma(j\beta/2)} \cdot \prod_{1 \leq i<j \leq N}  (y_j^N-y_i^N) \cdot  \prod_{1 \leq i<j \leq k}  (y_j^k-y_i^k)^{\beta - 1}  \\
& \times \prod_{j = k}^{N-1}\left(\prod_{1\leq a<b \leq j}(y_b^j-y_a^j)^{2-\beta} \prod_{a=1}^j\prod_{b=1}^{j+1}|y_a^j-y_b^{j+1}|^{\frac{\beta}{2}-1}\right)  \prod_{i = 1}^N e^{-\frac{N\beta}{2} V(y^N_i)}, \mbox{ and }\\
&\mathcal{G}_{N,k} = \{ y = (y^N, \dots, y^k) \in \mathbb{R}^N \times \cdots \times \mathbb{R}^k: y_{i}^{j}<y_i^{j-1}<y_{i+1}^{j}, \mbox{$j = k,\dots, N$, $i = 1, \dots, j$} \}.
\end{split}
\end{equation}
The latter statement can be deduced from a version of the Dixon-Anderson identity \cite{Anderson, Dixon}, see Section \ref{Section4.1} for the details. When $k = N$ we have that (\ref{S1Corners2}) becomes
\begin{equation}\label{S1CornersProjTop}
\begin{split}
f_{N, N}(y^N) =  \hspace{2mm} & (Z_{N}^c)^{-1} \cdot \prod_{j = 1}^{N} \frac{\Gamma(\beta/2)^j}{\Gamma(j\beta/2)}  \cdot  {\bf 1}\{ a_- < y^N_1 < y^N_2 < \cdots < y^N_N < a_+\}  \\
&  \times \prod_{1 \leq i<j \leq N} (y_j^N-y_i^N)^{\beta}  \prod_{i = 1}^N e^{-\frac{N\beta}{2} V(y^N_i)},
\end{split}
\end{equation}
which is precisely the distribution of a {\em $\beta$-log gas} \cite{Forr}. In this way, one can view (\ref{S1Corners1}) and (more generally) (\ref{S1Corners2}) as natural multi-level extensions of the $\beta$-log gases on $\mathbb{R}$.\\

In \cite{DK2020} the authors of the present paper proposed a certain {\em integrable discretization} of the measures in (\ref{S1Corners2}), which we presently describe. We begin by introducing some useful notation and defining the state space of the model.

Throughout the paper we write $\llbracket a,b \rrbracket = \{a, a+1, \dots, b\}$ for two integers $b \geq a$, with the convention that $\llbracket a,b \rrbracket = \emptyset$ if $a > b$. For $N \in \mathbb{N}$, $M \in \mathbb{Z}_{\geq 0}$, $k \in \llbracket 1, N \rrbracket$ and $\theta = \beta/2  > 0$  we define
\begin{equation}\label{S1GenState}
\begin{split}
&\Lambda^M_N = \{  (\lambda_1, \dots, \lambda_N) \in \mathbb{Z}^N : \lambda_1\geq  \lambda_2 \geq \cdots \geq \lambda_N,\mbox{ and }0  \leq \lambda_i  \leq M \}, \\
&\mathbb{W}^{\theta,M}_{k} = \{ (\ell_1, \dots, \ell_k):  \ell_i = \lambda_i - i \cdot \theta, \mbox{ with } (\lambda_1, \dots, \lambda_k) \in \Lambda^M_k\}, \\
&\XX= \{ (\ell^N, \ell^{N-1}, \dots, \ell^{k}) \in \mathbb{W}^{\theta, M}_{N} \times \mathbb{W}^{\theta,M}_{N-1} \times \cdots \times \mathbb{W}^{\theta,M}_{k} : \ell^N \succeq \ell^{N-1} \succeq \cdots \succeq \ell^{k} \},
\end{split}
\end{equation}
where $\ell \succeq m$ means that if $ \ell_i = \lambda_i - i \cdot\theta$ and $m_i = \mu_i - i \cdot \theta$ then $\lambda_1 \geq \mu_1 \geq \lambda_2 \geq \mu_2 \geq \cdots \geq \mu_{N-1} \geq \lambda_N$. The set $\XX$ is the state space of our point configuration $(\ell^N, \ell^{N-1}, \dots, \ell^{k})$.

We consider measures on $\XX$ of the form
\begin{equation}\label{S1PDef}
\P(\ell^N, \dots, \ell^{k}) = (Z_{N,k}^d)^{-1} \cdot H^t(\ell^N) \cdot H^b(\ell^k)
\cdot \prod_{j = k}^{N-1}  \I(\ell^{j+1}, \ell^j), \mbox{ where }
\end{equation}
\begin{equation}\label{S1PDef2}
\begin{split}
\I(\ell^{j+1}, \ell^{j}) = & \prod_{1 \leq p < q \leq j+1}\frac{\Gamma(\ell^{j+1}_p - \ell^{j+1}_q + 1 - \theta)}{\Gamma(\ell^{j+1}_p - \ell^{j+1}_q) } \cdot \prod_{1 \leq p < q \leq j} \frac{\Gamma(\ell^{j}_p - \ell^{j}_q + 1)}{\Gamma(\ell^{j}_p - \ell^{j}_q + \theta)} \\
&\times\prod_{1 \leq p < q \leq j+1} \frac{\Gamma(\ell^{j}_p -
  \ell^{j+1}_q)}{ \Gamma(\ell^{j}_p - \ell^{j+1}_q + 1 - \theta)}  \cdot
\prod_{1 \leq p \leq q \leq j} \frac{\Gamma(\ell^{j+1}_p -
  \ell^{j}_q + \theta)}{\Gamma(\ell^{j+1}_p - \ell^{j}_q + 1)} ,
\end{split}
\end{equation}
\begin{equation}\label{S1PDef3}
\begin{split}
H^t(\ell^N) = \prod_{1 \leq p < q \leq N} &\frac{\Gamma(\ell^N_p - \ell^N_q + 1)}{\Gamma(\ell^N_p - \ell^N_q + 1 - \theta)}  \prod_{p = 1}^N w_N(\ell^N_p) , \hspace{2mm} H^b(\ell^k) =   \prod_{1 \leq p < q \leq k} \frac{\Gamma(\ell^k_p - \ell^k_q + \theta)}{\Gamma(\ell^k_p - \ell^k_q)}.
\end{split}
\end{equation}
In (\ref{S1PDef}), (\ref{S1PDef2}) and (\ref{S1PDef3}) we have that $\Gamma$ is the usual Euler gamma function, $w_N$ is a positive function on $[-N\theta, M - \theta]$ and $Z_{N,k}^d$ is a normalization constant. When $k = 1$ we refer to the measures in (\ref{S1PDef}) as {\em discrete $\beta$-corners processes}.
\begin{remark}\label{S1ProjR} We mention that the same way (\ref{S1Corners2}) for $k = m$ is the projection of (\ref{S1Corners2}) for any $k \in \llbracket 1, m \rrbracket$ to $(y^N, \dots, y^m)$ we have that (\ref{S1PDef}) for $k = m \in \llbracket 1, N \rrbracket$ is a projection of (\ref{S1PDef}) for any $k \in \llbracket 1, m \rrbracket$ to $(\ell^N, \dots, \ell^m)$. This stability with respect to projections is proved in Lemma \ref{ProjLemma} and is a consequence of the branching relations for Jack symmetric functions.
\end{remark}
\begin{remark}In Proposition \ref{prop_cont_limit} we prove that the measures in (\ref{S1Corners2}) can be obtained as a diffuse scaling limit of those in (\ref{S1PDef}) -- this is why we refer to $\P$ as a discretization of (\ref{S1Corners2}). We mention that the fact that (\ref{S1Corners2}) is a diffuse scaling limit of (\ref{S1PDef}) is somewhat well-known, and can be deduced for example using the arguments in \cite[Section 2]{GS} for $\theta$-Gibbs measures. The proof we give for Proposition \ref{prop_cont_limit} is more direct than \cite[Section 2]{GS} and contains more details.

The reason we further call $\P$ an {\em integrable} discretization comes from the enhanced algebraic structure contained in $\P$, coming from the connection of these measures to Jack symmetric functions -- see Section \ref{Section4.2} and \cite[Section 6]{DK2020} for more details. We mention that the symmetric function origin of the measure $\P$ is the reason we reversed the indices of the particle locations compared to (\ref{S1Corners2}) so that $\ell_1^j > \cdots > \ell_j^j$, while in (\ref{S1Corners2}) we have $y^j_j > \cdots > y_1^j$ (as is common in random matrix theory).
\end{remark}
\begin{remark}Our primary interest in the measures $\P$ in (\ref{S1PDef}) comes from their connection to integrable probability. For example, $\P$ arise in a class of integrable models, called {\em ascending Jack processes} (these are special cases of the {\em Macdonald processes} of \cite{BorCor} in the parameter specialization that takes Macdonald to Jack symmetric functions). 

When $\theta = 1$ the measures $\P$ also naturally appear in probability distributions on the irreducible representations of $U(N)$, see \cite[Section 3.2]{BufGor} for the details. This connection is based on the structure of $\P$, which mimics the branching relations for Schur symmetric functions (these are Jack symmetric functions when $\theta = 1$). When $\theta = 1/2$ or $\theta = 2$ one should be able to make a similar construction to \cite[Section 3.2]{BufGor} for the real and quaternionic case, respectively, since when $\theta = 1/2$ or $2$, Jack symmetric functions become {\em zonal polynomials}, see \cite[Chapter VII]{Mac}. 
\end{remark}

When $k = N$ the measures in (\ref{S1PDef}) take the form
\begin{equation}\label{S1SigleLevel}
\mathbb{P}(\ell^N) = (Z_{N,N}^d)^{-1} \prod_{1\leq i<j \leq N} \frac{\Gamma(\ell^N_i - \ell^N_j + 1)\Gamma(\ell^N_i - \ell^N_j+ \theta)}{\Gamma(\ell^N_i - \ell^N_j)\Gamma(\ell^N_i - \ell^N_j +1-\theta)}  \prod_{i=1}^{N}w_N(\ell_i^N).
\end{equation}
The measures in (\ref{S1SigleLevel}) are called {\em discrete $\beta$-ensembles}, a term that was coined in \cite{BGG}, where these measures were introduced and extensively studied. In this way, one can view (\ref{S1PDef}) as a natural multi-level extension of the discrete $\beta$-ensembles from \cite{BGG}.\\

The goal of this paper is to derive {\em loop equations} for the (projections) of continuous and discrete $\beta$-corners processes, i.e. the measures in (\ref{S1Corners2}) and (\ref{S1PDef2}). Our loop equations can be thought of as functional equations that relate the joint cumulants of the Stieltjes transforms of the empirical measures on all the levels of a $\beta$-corners processes. The equations we derive for continuous $\beta$-corners processes generalize the loop equations (also called {\em Schwinger-Dyson equations}) for $\beta$-log gases, obtained in \cite[Theorem 3.1 and Theorem 3.2]{BoGu}, from a single-level to a multi-level setting. Similarly, the equations we derive for discrete $\beta$-corners processes generalize the loop equations (also called {\em Nekrasov equations} \cite{N}) for discrete $\beta$-ensembles, obtained in \cite{BGG}, from a single-level to a multi-level setting. 

For $\beta$-log gases loop equations have proved to be a very efficient tool in the study of global fluctuations,  see \cite{BoGu2, BoGu,JL, KS, S} and the references therein. Since their introduction loop equations have also been used to prove {\em local universality} for random matrices \cite{BEY, BFG}. In the discrete setup loop equations have been used to study global fluctuations in \cite{BGG} and edge fluctuations in \cite{huang} for discrete $\beta$-ensembles. It is our strong hope that the multi-level loop equations we derive in the present paper can also be used to study the global fluctuations of continuous and discrete $\beta$-corners processes. 

As mentioned earlier, our primary interest is on the discrete side. Specifically, we are interested in the study of ascending Jack processes, which are special cases of the Macdonald processes from \cite{BorCor}, and which fit into the setup of (\ref{S1PDef}). We have recently finished a project that utilizes the multi-level loop equations of the present paper to study the global fluctuations of a class of ascending Jack processes, which we call {\em $\beta$-Krawtchouk corners processes}, see \cite{DK24}. We have demonstrated that the asymptotic fluctuations of these models are governed by a 2d Gaussian field, which agrees with the one for Wigner matrices.

%-------------------------------------------------------------------------------------------------------------------------------------------------------------------------------------------------
% Section 1.2
%
%-------------------------------------------------------------------------------------------------------------------------------------------------------------------------------------------------
\subsection{Main results}\label{Section1.2} In this section we state our loop equations for discrete and continuous $\beta$-corners processes. In the discrete setting we refer to our equations as the {\em multi-level Nekrasov equations}, and in the continuous one we call them the {\em continuous multi-level loop equations}.

%-------------------------------------------------------------------------------------------------------------------------------------------------------------------------------------------------
% Section 1.2.1
%
%-------------------------------------------------------------------------------------------------------------------------------------------------------------------------------------------------
\subsubsection{Discrete setting}\label{Section1.2.1} In this section we state our multi-level Nekrasov equations -- this is Theorem \ref{MainNek1_intro} below. The equations take a different form depending on whether $\theta = 1$ or $\theta \neq 1$ and we forgo stating the equations for the case $\theta = 1$ until the main text, see Theorem \ref{MainNek2}. We also mention that Theorem \ref{MainNek1_intro} below is not the most general formulation of our equations, which can be found as Theorem \ref{MainNek1} in the main text.

\begin{theorem}\label{MainNek1_intro} Let $\P$ be a measure as in (\ref{S1PDef}) for $\theta > 0$, $\theta \neq 1$, $ N \in \mathbb{N}$, $k \in \llbracket 1, N \rrbracket$, $M \in \mathbb{Z}_{\geq 0}$. Let $\mathcal{M} \subseteq \mathbb{C}$ be an open set and $[- N \cdot \theta, M + 1 - \theta] \subseteq \mathcal{M}$. Suppose that there exist holomorphic functions $\Phi^+_N, \Phi^-_N$ on $\mathcal{M}$ such that 
\begin{equation}\label{S1eqPhiN}
\begin{split}
&\frac{w_N(x)}{w_N(x-1)}=\frac{\Phi_N^+(x)}{\Phi_N^-(x)} \mbox{ for $x \in [ -N\theta +1, M -\theta]$ and  $\Phi_N^-(-N\theta) = \Phi_N^+(M+1 - \theta) = 0$.}
\end{split}
\end{equation} 

Then the following functions $R_1(z)$, $R_2(z)$ are analytic in $\mathcal{M}$:
\begin{equation}\label{eq:mN1_gen_intro}
\begin{split}
R_1 (z):= \hspace{2mm}&\Phi_N^-(z) \cdot \mathbb{E} \left[ \prod_{p = 1}^N\frac{z- \ell^N_p -\theta}{z - \ell^N_p} \right] + \Phi_N^+(z)\cdot \mathbb{E} \left[ \prod_{p = 1}^{k}\frac{z- \ell^{k}_p + \theta - 1}{z - \ell^{k}_p - 1} \right]   \\
&+ \frac{\theta \Phi_N^+(z)}{1-\theta}  \sum\limits_{j=k+1}^{N}  \mathbb{E} \left[  \prod_{p = 1}^{j} \frac{z- \ell^{j}_p  -\theta}{z - \ell_p^{j} - 1}  \prod_{p = 1}^{j-1} \frac{z-  \ell_p^{j-1} + \theta - 1}{z -  \ell_p^{j-1}} \right];
\end{split}
\end{equation}

\begin{equation}\label{eq:mN2_gen_intro}
\begin{split}
&R_2 (z):= \Phi_N^+(z)\cdot \mathbb{E}\hspace{-1mm} \left[ \prod_{p = 1}^N\frac{z- \ell^N_p +\theta-1}{z - \ell^N_p-1} \right] \hspace{-1mm} + \Phi_N^-(z)  \cdot \mathbb{E}\hspace{-1mm} \left[ \prod_{p = 1}^{k}\frac{z- \ell^{k}_p+(N-k-1)\theta }{z - \ell^{k}_p +(N-k)\theta} \right] \hspace{-1mm}  \\
& + \frac{\theta  \Phi_N^-(z) }{1-\theta} \sum\limits_{j=k+1}^{N}  \mathbb{E}\left[  \prod_{p = 1}^{j} \frac{z- \ell^{j}_p+(N-j+1)\theta -1}{z - \ell_p^{j} +(N-j)\theta}    \prod_{p = 1}^{j-1}  \frac{z-\ell_p^{j-1} +(N-j)\theta}{z -  \ell_p^{j-1}+(N-j+1) \theta-1} \right].
\end{split}
\end{equation}
\end{theorem}
\begin{proof} Theorem \ref{MainNek1_intro} is a special case of Theorem \ref{MainNek1} when $w_j(x) \equiv 1$ for $j \in \llbracket k, N - 1\rrbracket$, in which case we remark that the conditions of Theorem \ref{MainNek1} are satisfied with $\phi^{N+1}_1(z) = \phi^{j}_2(z) = \Phi_N^-(z)$ and $\phi^{j}_1(z) = \phi^{N+1}_2(z) = \Phi^+_N(z)$ for $j \in \llbracket k, N \rrbracket$. We also mention that $\R_1(z) = \R_2(z) = 0$ in Theorem \ref{MainNek1} in view of the second equality in (\ref{S1eqPhiN}), cf. Remark \ref{RemBoundary}.
\end{proof}

We call the expressions in (\ref{eq:mN1_gen_intro}) and (\ref{eq:mN2_gen_intro}) equations because once we multiply both sides by an analytic function and integrate around a closed contour the integrals involving $R_{1}$ and $R_2$ will give zero due to analyticity. In Section \ref{Section3} we explain how to use Theorem \ref{MainNek1_intro}, or rather its generalizations Theorems \ref{MainNek1} and \ref{MainNek2} in the main text, to obtain integral equations that relate the joint cumulants of the Stieltjes transforms of the empirical measures on the different levels. Specifically, in Lemma \ref{CrudeCumExp} we obtain integral equations that relate the joint cumulants of the random analytic functions
$$G_L^j(z) = \sum_{i =1}^j \frac{1}{z - \ell_i^j/L} \mbox{ for $j \in \llbracket k, N \rrbracket$ },$$
evaluated at different complex points. Here, $L > 0$ is an arbitrary positive scaling parameter. 

Theorem \ref{MainNek1_intro} is a multi-level analogue of the discrete loop equations obtained in \cite{BGG, DK2020}, which in turn go back to the work of Nekrasov \cite{N}. Our proof is similar in spirit to the ones in \cite{BGG, DK2020, N}, essentially performing a careful residue cancellation. This cancellation is a bit trickier in the multi-level setting and the closest analogue we could find is \cite[Appendix A]{NT22}. The main difficulty lies in the construction of  $R_{1}$ and $R_2$. In \cite{BGG} the construction comes from the work \cite{N}; however, we are not aware of a proper analogue for $R_1$ and $R_2$ in the multi-level settings in the physics literature. We also mention that recently a different generalization of the loop equations was proposed in \cite{GH22}, called {\em dynamical loop equations}, see also \cite{JH20}. At this time it is unclear whether and how our multi-level equations relate to the dynamical loop equations algebraically. In \cite[Section 1.3]{DK24} we explain that from a probabilistic point of view our loop equations provide access to the ``level'' integrals of the Stieltjes transforms, while the dynamical loop equations describe how one level changes if the previous is fixed, i.e. they provide access to the ``level'' derivatives. 

We find it remarkable that out multi-level equations in Theorems \ref{MainNek1} and \ref{MainNek2} have a somewhat different form depending on whether $\theta \neq 1$ or $\theta = 1$. This is not the case for the single-level Nekrasov equations from \cite{BGG}. We believe that it was crucial for us to find the correct observables in the Theorem \ref{MainNek1} for general $\theta \neq 1$ first and then specialize to $\theta=1$, which made the derivation of the equations notably harder.

%-------------------------------------------------------------------------------------------------------------------------------------------------------------------------------------------------
% Section 1.2.2
%
%-------------------------------------------------------------------------------------------------------------------------------------------------------------------------------------------------
\subsubsection{Continuous setting}\label{Section1.2.2}  In this section we state our continuous multi-level loop equations.

\begin{theorem}\label{ManyLevelLoopThm_intro}
  Fix $N \in \mathbb{N}$, $a_-,a_+ \in \mathbb{R}$ with $a_- < a_+$ and $\theta = \beta/2> 0$. In addition, let $V(z)$ be a holomorphic function in a complex neighborhood $U$ of $[a_-,a_+]$, which is real-valued on $[a_-,a_+]$. Let $k \in \llbracket 1, N \rrbracket$ and let $(Y^N, \dots, Y^k)$ be a random vector taking value in $\mathbb{R}^N \times \cdots \times \mathbb{R}^{k}$, whose distribution has density $f_{N,k}$ as in (\ref{S1Corners2}).

We think of the vector $(Y^N, \dots, Y^k)$ as encoding the locations of $N-k+1$ sets of particles, indexed by $j \in \llbracket k, N \rrbracket$, and for $z \in \mathbb{C} \setminus [a_-, a_+]$ and $j \in \llbracket k, N \rrbracket$ we denote the Stieltjes transform of the $j$-th particle configuration by
\begin{equation}\label{S6GcontDef_intro}
\mathcal{G}^j(z) = \sum_{i = 1}^j \frac{1}{z - Y^j_i}.
\end{equation}
We let $m_k, \dots, m_N \in \mathbb{Z}_{\geq 0}$ and set $\mathfrak{M} = \{ (r,f): r \in \llbracket k, N \rrbracket \mbox{ and } f \in \llbracket 1, m_r \rrbracket\}$. For any $ v_r^f \in \mathbb{C} \setminus [a_-, a_+]$ for $(r,f) \in J \subseteq \mathfrak{M}$ and bounded complex random variable $\xi$ we write $\K(\xi; J)$ for the joint cumulant of the random variables $\xi$ and $\{ \mathcal{G}^j(v_f^r) : (r,f) \in J \}$, see (\ref{CumDef}) for the definition.

\vskip 2pt
Then for $v \in \mathbb{C} \setminus [a_-, a_+]$
\begin{equation}\label{TLRankmn_intro}
\begin{split} 
  & \oint_{\Gamma}   \frac{(z-a_-)(z- a_+)  }{2 \pi \i  (z - v)}  \cdot \left[ \frac{\K \left( \SS(z); \mathfrak{M} \right)}{2} +\sum_{(r,f) \in \mathfrak{M}}  \frac{\K \left( \mathcal{G}^r(z) ; \mathfrak{M} \setminus \{ (r,f)\} \right)}{(z - v^r_f )^2} \right] = 0,
\end{split}
\end{equation}
where
\begin{equation}
  \begin{split}
    \SS(z)= \hspace{2mm} &-2N\theta \partial_z V(z) \cdot \mathcal{G}^N(z) +  \theta \partial_z \mathcal{G}^N(z) + \theta [\mathcal{G}^N(z) ]^2 +  (\theta - 2 ) \partial_z \mathcal{G}^k(z) +  \theta[ \mathcal{G}^k(z) ]^2 \\
    &+ \sum\limits_{j=k+1}^{N}\left(-(\theta + 1)  \partial_z \mathcal{G}^j(z) +    (1-\theta)[ \mathcal{G}^j(z) - \mathcal{G}^{j-1}(z) ]^2   +  (1 - \theta)  \partial_z \mathcal{G}^{j-1}(z)\right).
    \end{split}
\end{equation}
and $\Gamma$ is a positively oriented contour, which encloses the segment $[a_-, a_+]$, is contained in $U$ and excludes the points $\{ v_f^r: (r,f) \in \mathfrak{M} \}$ and $v$.
\end{theorem}
\begin{remark} In simple words, equation (\ref{TLRankmn_intro}) provides an integral equation that relates the joint cumulants (of different orders) of the Stieltjes transforms $\mathcal{G}^j(z)$ and their first order derivatives $\partial_z \mathcal{G}^j(z)$ for $j \in \llbracket k, N \rrbracket$, evaluated at different complex points. 
\end{remark}

\begin{remark}
Theorem \ref{ManyLevelLoopThm_intro} should be thought of as a multi-level generalization of \cite[Theorem 3.1 and Theorem 3.2]{BoGu} for the measures in (\ref{S1Corners2}). Specifically, when $N = k$ equation (\ref{TLRankmn_intro}) implies \cite[Theorem 3.1 and Theorem 3.2]{BoGu} as we explain here.

When $N = k$ equation (\ref{TLRankmn_intro}) simplifies to
\begin{equation*}
\begin{split} 
&0 =  \oint_{\Gamma} dz \frac{ N \theta (z-a_-)(z- a_+) [- \partial_z V(z)] }{2 \pi \i (z - v)} \cdot \K \left( \mathcal{G}^N(z);\mathfrak{M} \right) +\sum_{(r,f) \in \mathfrak{M}}  \oint_{\Gamma} dz \frac{  (z-a_-)(z- a_+)}{2 \pi \i (z - v) (v^r_f - z)^2} \\
& \times \K \left( \mathcal{G}^r(z) ; \mathfrak{M} \setminus \{ (r,f)\} \right) +\oint_{\Gamma}   \frac{dz(z-a_-)(z- a_+)  }{2 \pi \i  (z - v)}  \K\left(  \theta [\mathcal{G}^N(z) ]^2 + (\theta - 1) \partial_z \mathcal{G}^N(z)  ; \mathfrak{M} \right).
\end{split}
\end{equation*}
We apply Malyshev's formula (see (\ref{Mal3}) in the main text) to the term $\K\left(  \theta [\mathcal{G}^N(z) ]^2  ; \mathfrak{M} \right)$ and compute the last integral as the sum of the (minus) residues at $z = v$  and infinity. Further computing the sum of integrals on the first line as the sum of the (minus) residues at $z= v$ and $z = v_f^r$ brings us to \cite[Theorem 3.1 and Theorem 3.2]{BoGu}.
\end{remark}

The proof of Theorem \ref{ManyLevelLoopThm_intro} is given in Section \ref{Section4.1} and relies on two inputs -- Lemma \ref{CrudeCumExp} and Proposition \ref{prop_cont_limit}. As mentioned in Section \ref{Section1.2.1}, Lemma \ref{CrudeCumExp} is a direct consequence of our multi-level Nekrasov equations and provides integral equations, relating the joint cumulants of 
$$G_L^j(z) = \sum_{i =1}^j \frac{1}{z - \ell_i^j/L} \mbox{ for $j \in \llbracket k, N \rrbracket$ }$$
at various complex points. On the other hand, Proposition \ref{prop_cont_limit} shows that the random vector $(Y^N, \dots, Y^k)$ in Theorem \ref{ManyLevelLoopThm_intro} is the weak diffuse scaling limit for a suitable sequence, indexed by $L$, of random vectors $(\ell^{N}, \dots, \ell^k)$ with law (\ref{S1PDef}). The integral equations in (\ref{TLRankmn_intro}) are then obtained as a limit of the integral equations in Lemma \ref{CrudeCumExp} using the weak convergence of Proposition \ref{prop_cont_limit}.

\subsection*{Outline}
Section \ref{Section2} presents the proof of our Nekrasov equations. The equations take a different form depending on whether $\theta \neq 1$, see Theorem \ref{MainNek1}, and when $\theta = 1$, see Theorem \ref{MainNek2}. Section \ref{Section3} introduces a general method of deriving integral equations that relate a rich class of multi-level observables for discrete $\beta$-corners processes. One such set of integral equations is derived in Lemma \ref{CrudeCumExp}. In Section \ref{Section4.1} we summarize some basic statements about continuous $\beta$-corners processes and in Section \ref{Section4.2} we explain the relationship between discrete $\beta$-corners processes and Jack symmetric functions. In Section \ref{Section4.3} we prove that continuous $\beta$-corners processes are diffuse scaling limits of discrete ones, see Proposition \ref{prop_cont_limit}. Finally, in Section \ref{Section4.4} we use Lemma \ref{CrudeCumExp} and Proposition \ref{prop_cont_limit} to obtain the continuous loop equations from Theorem \ref{ManyLevelLoopThm_intro}.

\subsection*{Acknowledgments} The authors would like to thank Alexei Borodin and Vadim Gorin for many useful comments on earlier drafts of the paper. AK would also like to thank Nikita Nekrasov and Oleksandr Tsymbaliuk for helpful discussions. ED is partially supported by the Minerva Foundation Fellowship and the NSF grant DMS:2054703. AK is partially supported by NSF grant DMS:2153958.

%-------------------------------------------------------------------------------------------------------------------------------------------------------------------------------------------------
%    Section 2
%
%-------------------------------------------------------------------------------------------------------------------------------------------------------------------------------------------------
\section{General setup and the equations}\label{Section2} In this section we present the main result we establish for discrete $\beta$-corners processes, which we call the {\em multi-level Nekrasov equations}. The equations have different formulations when $\theta \neq 1$ -- see Theorem \ref{MainNek1}, and when $\theta = 1$ -- see Theorem \ref{MainNek2}.

%-------------------------------------------------------------------------------------------------------------------------------------------------------------------------------------------------
%    Section 2.1
%
%-------------------------------------------------------------------------------------------------------------------------------------------------------------------------------------------------
\subsection{Preliminary computations}\label{Section2.1} In this section we introduce some notation, which is required to formulate our multi-level Nekrasov equations in the next section and perform some preliminary computations, which will be required in their proof. In the remainder of this section we fix $N \in \mathbb{N}$, $M \in \mathbb{Z}_{\geq 0}$, $k \in \llbracket 1, N \rrbracket$ and $\theta  > 0$. We recall from Section \ref{Section1} that we write $\llbracket a,b \rrbracket = \{a, a+1, \dots, b\}$ for two integers $b \geq a$, with the convention that $\llbracket a,b \rrbracket = \emptyset$ if $a > b$.

Let $\XX$ be as in (\ref{S1GenState}). We consider complex measures on $\XX$ of the form
\begin{equation}\label{S2PDef}
\P(\ell^N, \dots, \ell^{k}) = Z_N^{-1} \cdot H^t(\ell^N) \cdot H^b(\ell^k)
\cdot \prod_{j = k}^{N-1}  \I(\ell^{j+1}, \ell^j), \mbox{ where }
\end{equation}
\begin{equation}\label{S2PDef2}
\begin{split}
\I(\ell^{j+1}, \ell^{j}) = & \prod_{1 \leq p < q \leq j+1}\frac{\Gamma(\ell^{j+1}_p - \ell^{j+1}_q + 1 - \theta)}{\Gamma(\ell^{j+1}_p - \ell^{j+1}_q) } \cdot \prod_{1 \leq p < q \leq j} \frac{\Gamma(\ell^{j}_p - \ell^{j}_q + 1)}{\Gamma(\ell^{j}_p - \ell^{j}_q + \theta)} \\
&\times\prod_{1 \leq p < q \leq j+1} \frac{\Gamma(\ell^{j}_p -
  \ell^{j+1}_q)}{ \Gamma(\ell^{j}_p - \ell^{j+1}_q + 1 - \theta)}  \cdot
\prod_{1 \leq p \leq q \leq j} \frac{\Gamma(\ell^{j+1}_p -
  \ell^{j}_q + \theta)}{\Gamma(\ell^{j+1}_p - \ell^{j}_q + 1)} \prod\limits_{p=1}^{j}w_j(\ell_p^{j}),
\end{split}
\end{equation}
\begin{equation}\label{S2PDef3}
\begin{split}
H^t(\ell^N) = \prod_{1 \leq p < q \leq N} &\frac{\Gamma(\ell^N_p - \ell^N_q + 1)}{\Gamma(\ell^N_p - \ell^N_q + 1 - \theta)}  \prod_{p = 1}^N w_N(\ell^N_p) , \hspace{2mm} H^b(\ell^k) =   \prod_{1 \leq p < q \leq k} \frac{\Gamma(\ell^k_p - \ell^k_q + \theta)}{\Gamma(\ell^k_p - \ell^k_q)}.
\end{split}
\end{equation}
In equations (\ref{S2PDef2}) and (\ref{S2PDef3}) we have that $w_j(\cdot)$ are non-vanishing complex-valued functions on the interval $[-j \cdot  \theta, M - \theta]$ for $j \in \llbracket k, N \rrbracket $ and we have assumed that the sum of $H^t(\ell^N) \cdot H^b(\ell^k) \cdot \prod_{j = k}^{N-1}  \I(\ell^{j+1}, \ell^j)$ over $\XX$ is a non-zero complex number $Z_N$ so that the division by this quantity in (\ref{S2PDef}) is well-defined. In this way $\P$ defines a complex measure on $\XX$ of total mass $1$. If $w_j(\cdot)$ are positive functions on $[-j \cdot  \theta, M - \theta]$ for $j \in \llbracket k, N \rrbracket $ then $Z_N > 0$ and $\P$ defines a probability measure on $\XX$. 

We mention that the measure in (\ref{S2PDef}) is a generalization of the one in (\ref{S1PDef}), the difference being that we now have a weight function $w_j(\cdot)$ attached to each level $j \in \llbracket k, N \rrbracket$, while in (\ref{S1PDef}) we only had the weight $w_N(\cdot)$ attached to the $N$-th level. 
We will formulate our multi-level Nekrasov equations for the more general measures in (\ref{S2PDef}). \\

We next proceed to derive formulas for ratios of the form $\P(\tilde{\ell})/\P({\ell}),$ where $\tilde{\ell}$ is obtained from $\ell$ by shifting a few of the particle locations by $1$. The formulas we will require are presented in two quadruples of equations, namely (\ref{TR1}), (\ref{TR2}), (\ref{TR3}) and (\ref{TR4}), and (\ref{UR1}), (\ref{UR2}), (\ref{UR3}) and (\ref{UR4}). In order to illustrate how they are derived we first explain how the functions $H^t, H^b$ and $I$ from (\ref{S2PDef2}) and (\ref{S2PDef3}) behave under shifts.

 In what follows we fix $\ell \in \XX$. If $s \in \mathbb{R}$, $j \in \llbracket k, N \rrbracket$, $i \in \llbracket 1, j \rrbracket $, $\ell_i^j = s$ we let $\tilde{\ell}^j(i)$ be 
$$\tilde{\ell}^j_p(i) = \ell^j_p \mbox{ for $p \in \llbracket 1, j \rrbracket \setminus \{i \}$ and } \tilde{\ell}_i^j(i)= \ell_i^j - 1 = s-1.$$
Assuming that $\tilde{\ell}^N(i) \in \mathbb{W}^{\theta,M}_{N} $ we have 
\begin{equation}\label{HTRat}
\begin{split}
&\frac{H^t(\tilde \ell^N(i))}{H^t(\ell^N)} =  \prod_{p = 1}^{i-1} \frac{ s -\ell^N_p -1}{s - \ell^N_p + \theta - 1} \cdot \prod_{p = i+1}^N\frac{s-\ell^N_p - \theta}{s-\ell^N_p} \cdot \dfrac{w_N(s-1)}{w_N(s)},
\end{split}
\end{equation}
where we used the functional equation for the gamma function $\Gamma(z+1) = z \Gamma(z)$. Analogously, assuming that $1\leq i \leq k$ and $\tilde{\ell}^k(i) \in \mathbb{W}^{\theta,M}_{k}$ we have
\begin{equation}\label{HBRat}
\begin{split}
&\frac{H^b(\tilde \ell^k(i))}{H^b(\ell^k)} = \prod_{p = 1}^{i-1} \frac{s-\ell^k_p -\theta}{s- \ell^k_p } \cdot \prod_{p = i+1}^{k}\frac{s-\ell^k_p- 1}{s-\ell^k_p+ \theta - 1}.
\end{split}
\end{equation}
Assuming that $j \in \{\max(k, i-1), \dots, N-1\}$ and $\tilde{\ell}^{j+1}(i) \in \mathbb{W}^{\theta,M}_{j+1}$ we have
\begin{equation}\label{ILeftRat}
\begin{split}
&\frac{\I(\tilde \ell^{j+1}(i), \ell^{j})}{\I(\ell^{j+1}, \ell^{j})} =  \prod_{p= 1}^{i-1} \frac{s-\ell^{j+1}_p +\theta - 1}{s - \ell_p^{j+1} }  \prod_{p = i+1}^{j+1} \frac{s-\ell_p^{j+1} - 1}{s-\ell_p^{j+1}- \theta}  \cdot \prod_{p = 1}^{j} \frac{s - \ell_p^j}{s - \ell_p^j + \theta - 1}.
\end{split}
\end{equation}
Assuming that $j \in \{\max(k, i), \dots, N-1\}$ and $\tilde{\ell}^{j}(i) \in \mathbb{W}^{\theta,M}_{j}$ we have
\begin{equation}\label{IRightRat}
\begin{split}
&\frac{\I(\ell^{j+1}, \tilde \ell^{j}(i))}{\I(\ell^{j+1}, \ell^{j})} =\prod_{p = 1}^{i-1} \frac{s -\ell_p^j - 1 }{s-\ell_p^j - \theta}  \prod_{p= i+1}^{j} \frac{s-\ell_p^j + \theta - 1}{s-\ell_p^j } \cdot \prod_{p= 1}^{j+1}\frac{s- \ell_p^{j+1}- \theta}{s  - \ell_p^{j+1} - 1}  \cdot \dfrac{w_{j}(s-1)}{w_{j}(s)}.
\end{split}
\end{equation}
Combining (\ref{ILeftRat}) and (\ref{IRightRat}) we get for $j \in \{\max(k, i), \dots, N-1\}$, provided that $\ell_i^j = \ell_i^{j+1} = s$, $\tilde{\ell}^{j+1}(i) \in \mathbb{W}^{\theta,M}_{j+1}$ and $\tilde{\ell}^{j}(i) \in \mathbb{W}^{\theta,M}_{j}$ that
\begin{equation}\label{IBothRat}
\begin{split}
&\frac{I(\tilde{\ell}^{j+1}(i), \tilde{\ell}^j(i))}{I({\ell}^{j+1}, {\ell}^j)} =  \frac{I(\tilde{\ell}^{j+1}(i), \tilde{\ell}^j(i))}{I({\ell}^{j+1}, \tilde{\ell}^j(i))} \cdot \frac{I({\ell}^{j+1}, \tilde{\ell}^j(i))}{I({\ell}^{j+1}, {\ell}^j)} \\
& = \prod_{p= 1}^{i-1} \frac{s-\tilde{\ell}^{j+1}_p(i) +\theta - 1}{s - \tilde{\ell}_p^{j+1}(i) }  \prod_{p = i+1}^{j+1} \frac{s-\tilde{\ell}_p^{j+1}(i) - 1}{s-\tilde{\ell}_p^{j+1}(i)- \theta}  \cdot \prod_{p = 1, p \neq i}^{j} \frac{s - \ell_p^j}{s - \ell_p^j + \theta - 1} \cdot \frac{1}{\theta} \\
&\times \prod_{p = 1}^{i-1} \frac{s -\ell_p^j - 1 }{s-\ell_p^j - \theta}  \prod_{p= i+1}^{j} \frac{s-\ell_p^j + \theta - 1}{s-\ell_p^j } \cdot \prod_{p= 1}^{j+1}\frac{s- \ell_p^{j+1}- \theta}{s  - \ell_p^{j+1} - 1}  \cdot \dfrac{w_{j}(s-1)}{w_{j}(s)} \\
& =  \prod_{p = 1}^{i-1} \frac{(s - \ell_p^j)(s -\ell_p^j - 1)}{(s - \ell_p^j + \theta - 1)(s-\ell_p^j - \theta)}    \cdot \prod_{p= 1}^{i-1} \frac{(s-{\ell}^{j+1}_p +\theta - 1)(s- \ell_p^{j+1}- \theta)}{(s - {\ell}_p^{j+1})(s  - \ell_p^{j+1} - 1) }   \cdot \dfrac{w_{j}(s-1)}{w_{j}(s)}.
\end{split}
\end{equation}
Combining (\ref{ILeftRat}) and (\ref{IRightRat}) we also get for $j \in \{\max(k, i), \dots, N-1\}$, provided that $\ell_i^j  = s$ and $\ell_{i+1}^{j+1} = s - \theta$, $\tilde{\ell}^{j+1}(i+1) \in \mathbb{W}^{\theta,M}_{j+1}$ and $\tilde{\ell}^{j}(i) \in \mathbb{W}^{\theta,M}_{j}$ that
\begin{equation}\label{IBothRatDiag}
\begin{split}
&\frac{I(\tilde{\ell}^{j+1}(i+1), \tilde{\ell}^j(i))}{I({\ell}^{j+1}, {\ell}^j)} =  \frac{I(\tilde{\ell}^{j+1}(i+1), \tilde{\ell}^j(i))}{I(\tilde{\ell}^{j+1}(i+1), {\ell}^j)} \cdot \frac{I(\tilde{\ell}^{j+1}(i+1), {\ell}^j)}{I({\ell}^{j+1}, {\ell}^j)} \\
& = \prod_{p = 1}^{i-1} \frac{s -\ell_p^j - 1 }{s-\ell_p^j - \theta}  \prod_{p= i+1}^{j} \frac{s-\ell_p^j + \theta - 1}{s-\ell_p^j } \cdot \prod_{p= 1, p \neq i+1}^{j+1}\frac{s- \tilde\ell_p^{j+1}(i+1)- \theta}{s  - \tilde\ell_p^{j+1}(i+1) - 1} \cdot \frac{1}{\theta} \cdot \dfrac{w_{j}(s-1)}{w_{j}(s)} \\
&\times \prod_{p= 1}^{i} \frac{s-\ell^{j+1}_p - 1}{s - \ell_p^{j+1} - \theta }  \prod_{p = i+2}^{j+1} \frac{s-\ell_p^{j+1} - 1 - \theta}{s-\ell_p^{j+1}- 2\theta}  \cdot \prod_{p = 1}^{j} \frac{s - \ell_p^j - \theta}{s - \ell_p^j  - 1} \\
& = \prod_{p= i+1}^{j} \frac{(s-\ell_p^j + \theta - 1)(s - \ell_p^j - \theta)}{(s-\ell_p^j)(s - \ell_p^j  - 1) }    \cdot \prod_{p=  i+2}^{j+1}\frac{(s- \ell_p^{j+1}- \theta)(s-\ell_p^{j+1} - 1 - \theta)}{(s  - \ell_p^{j+1}- 1)(s-\ell_p^{j+1}- 2\theta)}  \cdot \dfrac{w_{j}(s-1)}{w_{j}(s)}.
\end{split}
\end{equation}

We now derive our first quadruple of equations. Suppose that we have fixed $i \in \llbracket 1, N \rrbracket$, $N  \geq j_1 \geq j_2 \geq \max(k, i)$ and $\ell \in \XX$ such that $\ell_i^{j_2} = \cdots = \ell_i^{j_1} = s$. Let $\tilde{\ell}$ be defined by 
\begin{equation*}
\tilde{\ell}_q^j = \begin{cases}  s - 1 &\mbox{ for $q = i$ and $j_1 \geq j \geq j_2$, } \\  \ell_i^j & \mbox{ for all other $(q,j)$ with $N \geq j \geq k$ and $j \geq q  \geq 1$}, \end{cases}
\end{equation*}
and suppose $\tilde{\ell} \in \XX$. Then from (\ref{ILeftRat}), (\ref{IRightRat}) and (\ref{IBothRat}) we have the following formula when $N-1 \geq j_1 \geq j_2 \geq \max(k + 1 , i)$
\begin{equation}\label{TR1}
\begin{split}
&\frac{\P(\tilde{\ell})}{\P(\ell)} =  \frac{I(\ell^{j_1+1}, \tilde{\ell^{j_1}})}{I(\ell^{j_1+1}, {\ell^{j_1}})} \cdot \frac{I(\tilde{\ell}^{j_2}, {\ell^{j_2-1}})}{I(\ell^{j_2}, \ell^{j_2-1})} \cdot  \prod_{j = j_2 }^{j_1-1} \frac{I(\tilde{\ell}^{j+1}, \tilde{\ell^j})}{I(\ell^{j+1}, \ell^j)} =  \prod_{j = j_2}^{j_1} \frac{w_j(s-1)}{w_j(s)}  \\
&   \times  \prod_{p= 1, p \neq i}^{j_1} \frac{s-\ell_p^{j_1} + \theta - 1}{s-\ell_p^{j_1} } \cdot    \prod_{p = 1, p \neq i}^{j_2} \frac{s-\ell_p^{j_2} - 1}{s-\ell_p^{j_2}- \theta}  \cdot \prod_{p = 1}^{j_2 - 1} \frac{s - \ell_p^{j_2-1}}{s - \ell_p^{j_2 - 1} + \theta - 1} \cdot \prod_{p= 1}^{j_1+1}\frac{s- \ell_p^{j_1+1}- \theta}{s  - \ell_p^{j_1+1} - 1}.
\end{split}
\end{equation}
When $N = j_1 \geq j_2 \geq \max(k + 1 , i)$ we have from (\ref{HTRat}), (\ref{ILeftRat}) and (\ref{IBothRat}) that 
\begin{equation}\label{TR2}
\begin{split}
&\frac{\P(\tilde{\ell})}{\P(\ell)} = \frac{H^t(\tilde \ell^N)}{H^t(\ell^N)} \cdot \frac{I(\tilde{\ell}^{j_2}, {\ell^{j_2-1}})}{I(\ell^{j_2}, \ell^{j_2-1})} \cdot  \prod_{j = j_2 }^{N-1} \frac{I(\tilde{\ell}^{j+1}, \tilde{\ell^j})}{I(\ell^{j+1}, \ell^j)}  \\
& = \prod_{p = 1, p \neq i}^{N} \frac{ s -\ell^N_p -\theta}{s - \ell^N_p } \cdot    \prod_{p= 1, p \neq i}^{j_2} \frac{s-\ell_p^{j_2} - 1}{s-\ell_p^{j_2}- \theta}  \cdot \prod_{p = 1}^{j_2 - 1} \frac{s - \ell_p^{j_2-1}}{s - \ell_p^{j_2 - 1} + \theta - 1} \cdot  \prod_{j = j_2}^{N} \frac{w_j(s-1)}{w_j(s)} . 
\end{split}
\end{equation}
When $N -1 \geq j_1 \geq j_2 = k$ and $1 \leq i \leq k$ we have from (\ref{HBRat}), (\ref{IRightRat}) and (\ref{IBothRat}) that 
\begin{equation}\label{TR3}
\begin{split}
&\frac{\P(\tilde{\ell})}{\P(\ell)} =   \frac{I(\ell^{j_1+1}, \tilde{\ell^{j_1}})}{I(\ell^{j_1+1}, {\ell^{j_1}})} \cdot \frac{H^b(\tilde{\ell}^k)}{H^b({\ell}^k)} \cdot  \prod_{j = k}^{j_1-1} \frac{I(\tilde{\ell}^{j+1}, \tilde{\ell^j})}{I(\ell^{j+1}, \ell^j)}  \\
&  =  \prod_{p=1, p \neq i}^{j_1} \frac{s-\ell_p^{j_1} + \theta - 1}{s-\ell_p^{j_1} } \cdot \prod_{p= 1}^{j_1+1}\frac{s- \ell_p^{j_1+1}- \theta}{s  - \ell_p^{j_1+1} - 1} \cdot  \prod_{p = 1, p \neq i}^{k}\frac{s-\ell^k_p- 1}{s-\ell^k_p+ \theta - 1} \cdot \prod_{j = k}^{j_1} \frac{w_j(s-1)}{w_j(s)}.
\end{split}
\end{equation}
When $ j_1 = N$, $j_2 = k$ and $1 \leq i \leq k$ we have from (\ref{HTRat}), (\ref{HBRat}) and (\ref{IBothRat}) that 
\begin{equation}\label{TR4}
\begin{split}
&\frac{\P(\tilde{\ell})}{\P(\ell)} =   \frac{H^t(\tilde \ell^N)}{H^t(\ell^N)}  \cdot \frac{H^b(\tilde{\ell}^k)}{H^b({\ell}^k)} \cdot  \prod_{j = k}^{N-1} \frac{I(\tilde{\ell}^{j+1}, \tilde{\ell^j})}{I(\ell^{j+1}, \ell^j)}   \\
& =  \prod_{p = 1, p \neq i}^N\frac{s-\ell^N_p - \theta}{s-\ell^N_p}  \prod_{p =1, p \neq i}^{k}\frac{s-\ell^k_p- 1}{s-\ell^k_p+ \theta - 1}  \cdot  \prod_{j = k}^{N} \frac{w_j(s-1)}{w_j(s)}.
\end{split}
\end{equation}

We finally derive our second quadruple of equations. Suppose that we have fixed $i \in \llbracket 1, N \rrbracket $, $N  \geq j_1 \geq j_2 \geq \max(k, i)$ and $\ell \in \XX$ such that $\ell_{i + (m-j_2)}^m = s -(m - j_2) \theta$ for $m \in \llbracket  j_2, j_1 \rrbracket$. Let $\tilde{\ell}$ be defined by 
\begin{equation*}
\tilde{\ell}_q^j = \begin{cases}  s - (j-j_2) \theta - 1 &\mbox{ for  $j_1 \geq j \geq j_2$ and $q = i + (j - j_2)$, } \\  \ell_i^j & \mbox{ for all other $(q,j)$ with $N \geq j \geq k$ and $j \geq q  \geq 1$} \end{cases}
\end{equation*}
and suppose $\tilde{\ell} \in \XX$. Then from (\ref{ILeftRat}), (\ref{IRightRat}) and (\ref{IBothRatDiag}) we have the following formula when $N-1 \geq j_1 \geq j_2 \geq \max(k + 1 , i)$
\begin{equation}\label{UR1}
\begin{split}
&\frac{\P(\tilde{\ell})}{\P(\ell)} =  \frac{I(\ell^{j_1+1}, \tilde{\ell^{j_1}})}{I(\ell^{j_1+1}, {\ell^{j_1}})} \cdot \frac{I(\tilde{\ell}^{j_2}, {\ell^{j_2-1}})}{I(\ell^{j_2}, \ell^{j_2-1})} \cdot  \prod_{j = j_2 }^{j_1-1} \frac{I(\tilde{\ell}^{j+1}, \tilde{\ell^j})}{I(\ell^{j+1}, \ell^j)}  \\
& = \prod_{p = 1, p \neq i + j_1 -j_2}^{j_1} \frac{s -\ell_p^{j_1}  - \theta(j_1 - j_2)- 1 }{s-\ell_p^{j_1} - \theta (j_1 - j_2) - \theta }  \prod_{p= 1, p \neq i}^{j_2} \frac{s-\ell^{j_2}_p +\theta - 1}{s - \ell_p^{j_2} } \\
& \times \prod_{p = 1}^{j_2-1} \frac{s - \ell_p^{j_2-1}}{s - \ell_p^{j_2 - 1} + \theta - 1}  \prod_{p= 1}^{j_1+1}\frac{s- \ell_p^{j_1+1}- \theta( j_1 - j_2) -\theta}{s  - \ell_p^{j_1+1} - \theta(j_1 - j_2) -1}\ \prod_{j = j_2 }^{j_1} \dfrac{w_{j}(s-1- \theta(j-j_2))}{w_{j}(s - \theta(j-j_2))}.
\end{split}
\end{equation}
When $N = j_1 \geq j_2 \geq \max(k + 1 , i)$ we have from (\ref{HTRat}), (\ref{ILeftRat}) and (\ref{IBothRatDiag}) that 
\begin{equation}\label{UR2}
\begin{split}
&\frac{\P(\tilde{\ell})}{\P(\ell)} = \frac{H^t(\tilde \ell^N)}{H^t(\ell^N)} \cdot \frac{I(\tilde{\ell}^{j_2}, {\ell^{j_2-1}})}{I(\ell^{j_2}, \ell^{j_2-1})} \cdot  \prod_{j = j_2 }^{N-1} \frac{I(\tilde{\ell}^{j+1}, \tilde{\ell^j})}{I(\ell^{j+1}, \ell^j)} =  \prod_{p= 1, p \neq i}^{j_2} \frac{s-\ell^{j_2}_p +\theta - 1}{s - \ell_p^{j_2} }  \\
& \times \hspace{-3mm} \prod_{p = 1, p \neq i + N - j_2}^{N} \hspace{-2mm} \frac{ s -\ell^N_p -  \theta (N-j_2)  -1}{s - \ell^N_p - \theta (N-j_2-1) - 1} \prod_{p = 1}^{j_2-1} \frac{s - \ell_p^{j_2-1}}{s - \ell_p^{j_2 - 1} + \theta - 1} \prod_{j = j_2 }^{N} \dfrac{w_{j}(s-1- \theta(j-j_2))}{w_{j}(s - \theta(j-j_2))}.
\end{split}
\end{equation}
When $N -1 \geq j_1 \geq j_2 = k$ and $1 \leq i \leq k$ we have from (\ref{HBRat}), (\ref{IRightRat}) and (\ref{IBothRatDiag}) that 
\begin{equation}\label{UR3}
\begin{split}
&\frac{\P(\tilde{\ell})}{\P(\ell)} =   \frac{I(\ell^{j_1+1}, \tilde{\ell^{j_1}})}{I(\ell^{j_1+1}, {\ell^{j_1}})} \cdot \frac{H^b(\tilde{\ell}^k)}{H^b({\ell}^k)} \cdot  \prod_{j = k}^{j_1-1} \frac{I(\tilde{\ell}^{j+1}, \tilde{\ell^j})}{I(\ell^{j+1}, \ell^j)} = \prod_{p= 1}^{j_1+1}\frac{s- \ell_p^{j_1+1}- \theta( j_1 -k) -\theta}{s  - \ell_p^{j_1+1} - \theta(j_1 - k) -1}  \\
& \times \prod_{p = 1, p \neq i + j_1 -k}^{j_1} \frac{s -\ell_p^{j_1}  - \theta(j_1 -k)- 1 }{s-\ell_p^{j_1} - \theta (j_1 - k) - \theta } \prod_{p = 1, p \neq i}^{k} \frac{s-\ell^k_p -\theta}{s- \ell^k_p }  \prod_{j = k }^{j_1 } \dfrac{w_{j}(s-1- \theta(j-k))}{w_{j}(s - \theta(j-k))}.
\end{split}
\end{equation}
When $ j_1 = N$, $j_2 = k$ and $1 \leq i \leq k$ we have from (\ref{HTRat}), (\ref{HBRat}) and (\ref{IBothRatDiag}) that 
\begin{equation}\label{UR4}
\begin{split}
&\frac{\P(\tilde{\ell})}{\P(\ell)} =   \frac{H^t(\tilde \ell^N)}{H^t(\ell^N)}  \cdot \frac{H^b(\tilde{\ell}^k)}{H^b({\ell}^k)} \cdot  \prod_{j = k}^{N-1} \frac{I(\tilde{\ell}^{j+1}, \tilde{\ell^j})}{I(\ell^{j+1}, \ell^j)}   \\
& = \prod_{p = 1, p \neq i + N -k}^{N} \frac{ s -\ell^N_p -  \theta (N-k)  -1}{s - \ell^N_p - \theta (N-k-1) - 1} \prod_{p = 1, p \neq i}^{k} \frac{s-\ell^k_p -\theta}{s- \ell^k_p } \prod_{j = k }^{N} \dfrac{w_{j}(s-1- \theta(j-k))}{w_{j}(s - \theta(j-k))}.
\end{split}
\end{equation}

%-------------------------------------------------------------------------------------------------------------------------------------------------------------------------------------------------
%    Section 2.2
%
%-------------------------------------------------------------------------------------------------------------------------------------------------------------------------------------------------
\subsection{Multi-level Nekrasov equations: $\theta \neq 1$}\label{Section2.2} The goal of this section is to state and prove the multi-level Nekrasov equations for the case $\theta \neq 1$. 

\begin{theorem}\label{MainNek1} Let $\P$ be a complex measure as in (\ref{S2PDef}) for $\theta > 0$, $\theta \neq 1$, $ N \in \mathbb{N}$, $k \in \llbracket 1, N \rrbracket$, $M \in \mathbb{Z}_{\geq 0}$. Let $\mathcal{M} \subseteq \mathbb{C}$ be an open set and $[- N \cdot \theta, M + 1 - \theta] \subseteq \mathcal{M}$. Suppose that there exist functions $\phi_r^{ j}$ for $j=k,\dots, N + 1$, $r=1,2$ that are analytic in $\mathcal{M}$ and such that for any $j  \in \llbracket k, N \rrbracket$
\begin{equation}\label{eq:phi}
\begin{split}
\frac{\phi^{j+1}_1(z)}{ \phi^{j}_1(z)} &=  \frac{w_j(z-1)}{w_j(z)} \mbox{ for $z \in [1 -j \cdot \theta, M - \theta]$};\\
 \frac{\phi_2^{j}(z)}{ \phi_2^{j+1}(z)} &=\dfrac{w_{j}(z + (N -j) \cdot \theta-1)}{w_{j}(z + (N- j) \cdot \theta)} \mbox{ for $z \in [1 - N \cdot \theta, M - (N-j+1) \cdot \theta]$}.
\end{split}
\end{equation}
Then the following functions $R_1(z)$, $R_2(z)$ are analytic in $\mathcal{M}$:
\begin{equation}\label{eq:mN1_gen}
\begin{split}
R_1 (z):= \hspace{2mm}&\phi_1^{N+1}(z) \cdot \mathbb{E} \left[ \prod_{p = 1}^N\frac{z- \ell^N_p -\theta}{z - \ell^N_p} \right] + \phi_1^k(z) \cdot \mathbb{E} \left[ \prod_{p = 1}^{k}\frac{z- \ell^{k}_p + \theta - 1}{z - \ell^{k}_p - 1} \right]   \\
&+ \frac{\theta}{1-\theta}  \sum\limits_{j=k+1}^{N}\phi^{j}_1(z) \cdot \mathbb{E} \left[  \prod_{p = 1}^{j} \frac{z- \ell^{j}_p  -\theta}{z - \ell_p^{j} - 1}  \prod_{p = 1}^{j-1} \frac{z-  \ell_p^{j-1} + \theta - 1}{z -  \ell_p^{j-1}} \right]-\R_1(z);
\end{split}
\end{equation}

\begin{equation}\label{eq:mN2_gen}
\begin{split}
&R_2 (z):= \phi_2^{N+1}(z)  \mathbb{E}\hspace{-1mm} \left[ \prod_{p = 1}^N\frac{z- \ell^N_p +\theta-1}{z - \ell^N_p-1} \right] \hspace{-1mm} + \phi_2^k(z)  \mathbb{E}\hspace{-1mm} \left[ \prod_{p = 1}^{k}\frac{z- \ell^{k}_p+(N-k-1)\theta }{z - \ell^{k}_p +(N-k)\theta} \right] \hspace{-1mm} + \hspace{-1mm} \frac{\theta}{1-\theta}  \\
&  \times \hspace{-2mm}  \sum\limits_{j=k+1}^{N} \hspace{-2mm} \phi^{ j}_2(z)\mathbb{E}\left[  \prod_{p = 1}^{j} \frac{z- \ell^{j}_p+(N-j+1)\theta -1}{z - \ell_p^{j} +(N-j)\theta}    \prod_{p = 1}^{j-1}  \frac{z-\ell_p^{j-1} +(N-j)\theta}{z -  \ell_p^{j-1}+(N-j+1) \theta-1} \right]-\R_2(z),
\end{split}
\end{equation}
where $s_M=M+1 - \theta$ and the functions $\R_1(z)$, $\R_2(z)$ are given by
\begin{equation}\label{NERem1}
\begin{split}
\R_1(z)= \hspace{2mm}& \dfrac{(-  \theta)  \phi^{N+1}_1(-N \cdot \theta) }{z + N \cdot \theta }  \cdot \mathbb{E}\left[\prod\limits_{p=1}^{N-1}\dfrac{\ell^N_p+(N+1) \cdot \theta}{\ell^N_p +N \cdot \theta} \cdot {\bf 1}\{  \ell^N_N= -N \cdot \theta  \}\right] \\
&+\dfrac{\theta \phi_1^k(s_M)  }{z-s_M} \cdot \mathbb{E}\left[\prod\limits_{p=2}^{k}\dfrac{s_M-\ell^k_p+\theta-1}{s_M-\ell^k_p-1}\cdot {\bf 1}\{ \ell^k_1=s_M-1 \} \right]  \\
   & + \hspace{-1mm} \sum\limits_{j=k+1}^{N} \hspace{-1mm}\frac{\theta  \phi^{j}_1(s_M)  }{z - s_M} \cdot \mathbb{E}\left[\prod\limits_{p=2}^{j} \dfrac{s_M-\ell^{j}_p-\theta}{s_M-\ell^{j}_p-1}
   \prod\limits_{p=1}^{j-1}   \dfrac{s_M-\ell^{j-1}_p+\theta-1}{s_M-\ell^{j-1}_p} \cdot {\bf 1 } \{ \ell^j_1=s_M-1 \} \right] \hspace{-1mm},
\end{split}
\end{equation}
\begin{equation}\label{NERem2}
\begin{split}
& \R_2(z)=  \frac{\theta \phi_2^{N+1}(s_M)}{z - s_M} \cdot  \mathbb{E}\left[\prod\limits_{p=2}^{N}\dfrac{s_M-\ell^N_p+\theta-1}{s_M-\ell^N_p-1}\cdot {\bf 1 }\{ \ell^N_1=s_M-1 \}\right]  \\
& + \frac{(-\theta)\phi_2^k(-N \cdot \theta ) }{z + N \cdot \theta} \cdot \mathbb{E}\left[\prod\limits_{p=1}^{k-1}\dfrac{\ell^k_p + (k-1)\cdot \theta}{\ell^k_p + k \cdot \theta} \cdot 
      {\bf 1 }\{\ell^k_{k}= -k \cdot \theta \} \right] \hspace{-1mm}+ \hspace{-1mm} \sum \limits_{j=k+1}^{N} \frac{(-\theta)\phi^{j}_2(-N \cdot \theta)}{z + N \cdot \theta}  \\
&\times \mathbb{E}\left[\prod\limits_{p=1}^{j-1} \dfrac{\ell^{j}_p+(j-1)\cdot \theta+1}{\ell^{j}_p+ j \cdot \theta}
     \prod\limits_{p=1}^{j-1} \dfrac{\ell^{j-1}_p+ j \cdot \theta}{\ell^{j-1}_p+ (j-1) \cdot \theta+1} \cdot {\bf 1}\{ \ell^{j}_{j}= - j \cdot \theta \} \right].
    \end{split}
\end{equation}
  \end{theorem}

\begin{remark}\label{RemBoundary}
If we have that $\phi^{N+1}_1(-N \cdot \theta)=0$ and $ \phi^{j}_1(s_M) = 0$ for $j = 1, \dots, N$ then $\R_1(z) = 0$. Analogously, if $\phi_2^{N+1}(s_M ) = 0$ and $\phi^j_2(-N \cdot \theta) = 0$ for $j = 1, \dots, N$ then $\R_2(z) = 0$.
\end{remark}

\begin{proof} Before we go into the details let us briefly explain the structure of the argument. We will establish the analyticity of $R_1$ and $R_2$ separately from each other by completing six steps (for each function). In the first step of the argument we prove that $R_1, R_2$ can only have possible poles at $s =  a- b \cdot \theta$, where $b \in \llbracket 1, N \rrbracket$ and $a \in \llbracket 0, M+1 \rrbracket$ and that all of these poles are simple. In the second step we fix one of the possible poles and compute the residue at this pole, for which we obtain an expression of the form
\begin{equation}\label{S2S2X1}
\mathsf{Res} = \sum_{i = 1}^N \sum_{j = k}^N \sum_{(\ell,j) \in {X}^{i,j}} F^{i}(s; \ell, j ) +  \sum_{i = 1}^N \sum_{j = k}^N \sum_{(\ell,j) \in {Y}^{i,j}} G^i(s; \ell,j),
\end{equation}
where ${X}^{i,j}$ and ${Y}^{i,j}$ are certain sets of pairs of $\ell \in \XX$ and $j \in \llbracket k, N \rrbracket$, and the functions $F^i,G^i$ are certain rational functions of $s$. In the proof of analyticity of $R_1$ the above residue formula is (\ref{eq:expand2}) and in the one for $R_2$ it is (\ref{eq:expand2V2}). To prove the analyticity of $R_1$ and $R_2$ we seek to show that the residue in (\ref{S2S2X1}) is in fact zero, i.e. the above two sums perfectly cancel with each other. 

In order to show that the sums in (\ref{S2S2X1}) cancel we construct for each $i \in \llbracket 1, N \rrbracket$ bijections
$$\bn: \sqcup_{j = k}^N {X}^{i,j} \rightarrow  \sqcup_{j = k}^N {Y}^{i,j} , \mbox{ such that } F^{i}(s; \ell, j ) = - G^i(s; \bn(\ell,j)).$$
The existence of such maps shows that for each $i$ the sums in (\ref{S2S2X1}) cancel with each other, which of course implies that the residue is zero. The definition of the map $\bn$ is done in Step 4 of the two proofs and in Step 5 it is shown that the map is a bijection. We mention that the map $\bn$ is {\em different } in the proof of analyticity for $R_1$ and $R_2$, denoted by $\b$ and $\bb$, respectively.

In Step 6 we verify that $F^{i}(s; \ell, j ) = - G^i(s; \bn(\ell,j))$, for which we use the quadruple of equations (\ref{TR1}), (\ref{TR2}), (\ref{TR3}) and (\ref{TR4}) from Section \ref{Section2.1} when working with $R_1$, and the quadruple (\ref{UR1}), (\ref{UR2}), (\ref{UR3}) and (\ref{UR4}) from Section \ref{Section2.1} when working with $R_2$.

We now turn to providing the details behind the above outline. Throughout the proof we will switch between $\ell$ and $\lambda$ using the formulas $\ell_i^j = \lambda_i^j - j \cdot \theta$ for $j \in \llbracket k , N \rrbracket$ and $i \in \llbracket  1, j \rrbracket$ -- we will not mention this further. We also adopt the convention $\ell_{j+1}^j = \lambda_{j+1}^j = -\infty$ and $\lambda_0^j = \ell_0^j = \infty$.\\

{\bf \raggedleft Analyticity of $R_1.$} For clarity we split the proof into six steps.\\

{\bf \raggedleft Step 1.} The function $R_1$ has possible poles at $s =  a- b \cdot \theta$, where $b \in \llbracket 1,  N\rrbracket $ and $a \in \llbracket 0, M+1 \rrbracket$. In this step we show that all of these poles are simple. We also show that the residues at $z = - N \cdot \theta$ and $z= s_M$ are equal to $0$, i.e. $R_1$ is analytic near these points. 

Firstly, note that the function $\R_1(z)$ only has simple poles at $z = -N \cdot \theta$ and $z = s_M$ by definition. Next, we have for each $j \in  \llbracket k, N \rrbracket$ that $\ell_1^j > \ell_2^j > \cdots > \ell_j^j$ and so the expectations in the first line of (\ref{eq:mN1_gen}) can produce only simple poles. Similarly, each of the two products inside the expectation in the second line of (\ref{eq:mN1_gen}) can only have simple poles. Let us check that it can not happen that these two products share the same pole. Assuming the contrary, there must exist $j \in \llbracket k+1,N \rrbracket $ and $\ell \in \XX$ such that $s = \ell_q^j + 1$ for some $q \in \llbracket 1, j \rrbracket$ and $s = \ell_{r}^{j-1}$ for some $r \in \llbracket 1, j - 1\rrbracket$. The last statement implies
$$ \ell^j_{q}-\ell^{j-1}_{r}=\lambda^j_q-\lambda^{j-1}_r +(r - q)\theta=-1.$$
Since $\lambda^{j-1}_r\geq \lambda^{j}_q$ for $r < q$ and $\lambda^{j}_q \geq \lambda^{j-1}_r$ when $q \leq  r$ the above equality can only hold if $r < q$, $\lambda^j_q=\lambda^{j-1}_r$  and $\theta = (q - r)^{-1}$ (here we used that $\ell \in \XX$ and so $\ell^j \succeq \ell^{j-1}$ as in (\ref{S1GenState})). From the interlacing $\ell^j \succeq \ell^{j-1}$ we conclude that necessarily $\lambda^{j-1}_{q-1}=\lambda^j_q$, which implies $\ell^{j-1}_{q-1}=\ell^j_q+\theta = s- 1 + \theta.$ This produces an extra factor $z - s$ from the $z - \ell_{q-1}^{j-1} + \theta -1$ in the second product in the second line of (\ref{eq:mN1_gen}), which means that the order of the pole at $s$ is reduced from $2$ to $1$. Thus indeed, all the poles are simple.\\

We next check that the residues at $z = -N \cdot \theta$ and $z= s_M$ are equal to $0$. Note that the residue at $z = -N \cdot \theta$ has two contributions, coming from the first expectation in the first line of (\ref{eq:mN1_gen}), when $\ell_N^N = -N \cdot \theta$, and from the first line of (\ref{NERem1}), and they exactly cancel. Similarly, the residue at $z = s_M$ gets contributions from the second expectation in the first line of (\ref{eq:mN1_gen}), when $\ell^k_1 = s_M - 1$, from the expectations on the second line of (\ref{eq:mN1_gen}), when $\ell_1^j = s_M - 1$, and from the second and third lines of (\ref{NERem1}), and they exactly cancel. This proves the analyticity of $R_1$ near $-N \cdot \theta$ and $s_M$.\\

{\bf \raggedleft Step 2.} In the next five steps we will compute the residue at any $s =  a- b \cdot \theta$ such that $s\not \in \{- N \cdot \theta, s_M\}$ and verify that it is zero. In this step we obtain a useful formula for this residue, which will be used later to show that it is zero.

We first expand all the expectations to get
\begin{equation}\label{eq:expand}
\begin{split}
&R_1(z)=\phi_1^{N+1}(z) \hspace{-1mm} \sum_{\ell \in \XX}  \P(\ell)\cdot  \prod_{p = 1}^N\frac{z- \ell^N_p -\theta}{z - \ell^N_p}  + \phi^k_1(z) \hspace{-1mm} \sum_{\ell \in \XX} 
\P(\ell) \prod_{p = 1}^{k}\frac{z- \ell^{k}_p + \theta - 1}{z - \ell^{k}_p - 1}  \\
& + \frac{\theta}{1-\theta}  \sum\limits_{j=k+1}^{N}\phi_1^{j}(z)
\sum_{\ell \in \XX } \P(\ell)  \prod_{p = 1}^{j} \frac{z- \ell^{j}_p -\theta}{z - \ell_p^{j} - 1}\cdot \prod_{p = 1}^{j-1}  \frac{z-  \ell_p^{j-1} + \theta - 1}{z -  \ell_p^{j-1}} -\R_1(z).
\end{split}
\end{equation}

For each $j \in \llbracket k+1, N \rrbracket$ and $i \in \llbracket 1, j \rrbracket$ we let $\Sm^{i,j}$ denote the set of pairs $(\ell,j)$, such that $\ell \in \XX$, $\ell^j_i = s - 1$ and
\begin{equation}\label{PR1}
\prod_{p = 1}^{j} \frac{1}{z - \ell_p^{j} - 1} \cdot \prod_{p = 1}^{j-1} (z-  \ell_p^{j-1} + \theta - 1)
\end{equation}
has a simple pole at $z= s$. If $i \geq j+1$ we let $\Sm^{i,j}$ denote the empty set.

For $i \in \llbracket 1, k \rrbracket$ we also let $\Sm^{i,k}$ denote the sets of pairs $(\ell, k)$ such that $\ell \in \XX$, $\ell^k_i = s - 1$ and 
\begin{equation}\label{PR2}
  \prod_{p = 1}^{k}\frac{z- \ell^{k}_p + \theta - 1}{z - \ell^{k}_p - 1} 
\end{equation}
has a simple pole at $z = s$. If $i \geq k+1$ we let $\Sm^{i,k}$ denote the empty set.

For each $j \in \llbracket  k, N-1 \rrbracket$ and $i \in \llbracket 1,  j \rrbracket$ we let $\Sp^{i,j}$ denote the set of pairs $(\ell,j)$, such that $\ell \in \XX$, $\ell^{j}_i = s$ and
\begin{equation}\label{PR3}
 \prod_{p = 1}^{j}  \frac{1}{z -  \ell_p^{j}} \cdot \prod_{p = 1}^{j+1} (z- \ell^{j+1}_p -\theta)
\end{equation}
has a simple pole at $z = s$. If $i \geq j+1$ we let $\Sp^{i,j}$ denote the empty set. 

For $i \in \llbracket 1, N \rrbracket$ we let $\Sp^{i,N}$ denote the  sets of pairs $(\ell, N)$ such that $\ell \in \XX$, $\ell_i^N = s$ and 
\begin{equation}\label{PR4}
\prod_{p = 1}^N\frac{z- \ell^N_p -\theta}{z - \ell^N_p}
\end{equation}
has a simple pole at $z = s$.\\

In view of our work in Step 1, and (\ref{eq:expand}) we have the following formula for the residue at $z = s$
\begin{equation}\label{eq:expand2}
\begin{split}
&\sum_{i = 1}^N \phi_1^{N+1}(s)  \sum_{(\ell,N) \in \Sp^{i,N}}  \P(\ell)\cdot (- \theta) \cdot   \prod_{p = 1, p \neq i}^N\frac{s- \ell^N_p -\theta}{s - \ell^N_p} \hspace{2mm} \\
& + \sum_{i = 1}^N  \phi^k_1(s)  \sum_{(\ell, k) \in \Sm^{i,k}}  \P(\ell) \cdot \theta \cdot \prod_{p = 1, p \neq i}^{k}\frac{s- \ell^{k}_p + \theta - 1}{s - \ell^{k}_p - 1}  \hspace{2mm}  \\
& + \frac{\theta}{1-\theta} \sum_{ i = 1}^N  \sum_{j=k+1}^{N}\phi_1^{j}(s) \sum_{(\ell,j) \in \Sm^{i,j}} \P(\ell)  (1 - \theta)  \prod_{p = 1, p \neq i}^{j} \frac{s- \ell^{j}_p -\theta}{s - \ell_p^{j} - 1}\cdot \prod_{p = 1}^{j-1}  \frac{s-  \ell_p^{j-1} + \theta - 1}{s -  \ell_p^{j-1}}  \\
& + \frac{\theta}{1-\theta} \sum_{ i = 1}^N  \sum_{j=k}^{N-1}\phi_1^{j + 1}(s) \sum_{(\ell,j ) \in \Sp^{i,j} } \hspace{-3mm} \P(\ell)  (\theta - 1)  \prod_{p = 1}^{j + 1} \frac{s- \ell^{j+1}_p -\theta}{s - \ell_p^{j+1} - 1}\cdot \prod_{p = 1, p \neq i}^{j}  \frac{s-  \ell_p^{j} + \theta - 1}{s -  \ell_p^{j}},
\end{split}
\end{equation}  
where we used that $\R_1(z)$ is analytic away from the points $-N \cdot \theta$ and $s_M$ and that $s\not \in \{-N \cdot \theta, s_M\}$ by assumption (hence $\R_1(z)$ does not contribute to the residue at $z = s$). We also mention that our convention is that the sum over an empty set is zero.

Equation (\ref{eq:expand2}) is the main output of this step and below we show that the total sum is zero.\\

{\bf \raggedleft Step 3.} The goal of this step is to show that the sum over all terms in (\ref{eq:expand2}) for a fixed $i \in \llbracket 1, N\rrbracket$ vanishes, which would of course imply that the full sum vanishes. To accomplish this we construct a bijection $\b$ 
\begin{equation}\label{ER0}
\b: \sqcup_{j = k}^N \Sp^{i,j}\rightarrow  \sqcup_{j = k}^N \Sm^{i,j} , \mbox{ such that }\b \left(\Sp^{i,j_1} \right)\subseteq \sqcup_{j =k}^{j_1} \Sm^{i,j} \mbox{ for $j_1  \in \llbracket k, N \rrbracket$ and }
\end{equation}
 if $(\tilde{\ell}, \tilde{n}) = \b(\ell, n)$ the following four equations all hold. We first have that
\begin{equation}\label{ER1}
0 = \phi_1^{N+1}(s) \P(\ell)\cdot (- 1) \cdot   \prod_{p = 1, p \neq i}^N\frac{s- \ell^N_p -\theta}{s - \ell^N_p} +  \phi_1^k(s)  \P(\tilde{\ell}) \cdot \prod_{p = 1, p \neq i}^{k}\frac{s- \tilde{\ell}^{k}_p + \theta - 1}{s - \tilde{\ell}^{k}_p - 1},
\end{equation}
if $(\ell, n) \in \Sp^{i, N}$ and $(\tilde{\ell}, \tilde{n}) \in \Sm^{i,k}$. Secondly, we have
\begin{equation}\label{ER2}
\begin{split}
0 = \hspace{2mm}&\phi_1^{N+1}(s) \P(\ell)\cdot (- 1) \cdot   \prod_{p = 1, p \neq i}^N\frac{s- \ell^N_p -\theta}{s - \ell^N_p}  \\
&+ \phi_1^{\tilde{n}}(s) \P(\tilde{\ell}) \cdot  \prod_{p = 1, p \neq i}^{\tilde{n}} \frac{s- \tilde{\ell}^{\tilde{n}}_p -\theta}{s - \tilde{\ell}_p^{\tilde{n}} - 1}\cdot \prod_{p = 1}^{\tilde{n}-1}  \frac{s-  \tilde{\ell}_p^{\tilde{n}-1} + \theta - 1}{s -  \tilde{\ell}_p^{\tilde{n}-1}},
\end{split}
\end{equation}
if $(\ell, n) \in \Sp^{i, N}$ and $(\tilde{\ell}, \tilde{n}) \in \Sm^{i, \tilde{n}}$ for $N \geq \tilde{n} \geq k+1$. Thirdly, we have
\begin{equation}\label{ER3}
\begin{split}
0 = \hspace{2mm}&\phi_1^{n+1}(s) \P(\ell) \cdot(-1) \cdot   \prod_{p = 1}^{n+1} \frac{s- \ell^{n+1}_p -\theta}{s - \ell_p^{n+1} - 1}\cdot \prod_{p = 1, p \neq i}^{n}  \frac{s-  \ell_p^{n} + \theta - 1}{s -  \ell_p^{n}}  \\
&+ \phi_1^{\tilde{n}}(s) \P(\tilde{\ell})  \cdot  \prod_{p = 1, p \neq i}^{\tilde{n}} \frac{s- \tilde{\ell}^{\tilde{n}}_p -\theta}{s - \tilde{\ell}_p^{\tilde{n}} - 1}\cdot \prod_{p = 1}^{\tilde{n}-1}  \frac{s-  \tilde{\ell}_p^{\tilde{n}-1} + \theta - 1}{s -  \tilde{\ell}_p^{\tilde{n}-1}},
\end{split}
\end{equation}
if $(\ell, n) \in \Sp^{i,n}$ and $(\tilde{\ell}, \tilde{n}) \in \Sm^{i, \tilde{n}}$ with $N - 1 \geq n \geq \tilde{n} \geq k + 1$. Finally, we have
\begin{equation}\label{ER4}
\begin{split}
0 = \hspace{2mm}&\phi_1^{n+1}(s) \P(\ell) \cdot (-1) \cdot  \prod_{p = 1}^{n+1} \frac{s- \ell^{n+1}_p -\theta}{s - \ell_p^{n+1} - 1}\cdot \prod_{p = 1, p \neq i}^{n}  \frac{s-  \ell_p^{n} + \theta - 1}{s -  \ell_p^{n}}  \\
& + \phi_1^k(s)  \P(\tilde{\ell}) \cdot \prod_{p = 1, p \neq i}^{k}\frac{s- \tilde{\ell}^{k}_p + \theta - 1}{s - \tilde{\ell}^{k}_p - 1} ,
\end{split}
\end{equation}
if $(\ell, n) \in \Sp^{i,n}$ for $N-1 \geq n \geq k$ and $(\tilde{\ell}, \tilde{n}) \in \Sm^{i,k}$.

If we can find a bijection $\b$ as in (\ref{ER0}) that satisfies (\ref{ER1}), (\ref{ER2}), (\ref{ER3}) and (\ref{ER4}) we would conclude that the sum in (\ref{eq:expand2}) is zero since the contributions from the first and fourth line in (\ref{eq:expand2}) exactly cancel with those on the second and third line. \\

We next focus on constructing the map $\b$, showing that it is a bijection of the sets in (\ref{ER0}) and that it satisfies equations (\ref{ER1}), (\ref{ER2}), (\ref{ER3}) and (\ref{ER4}).\\

{\bf \raggedleft Step 4.} In this step we construct the map $\b$ on $\sqcup_{j = k}^N \Sp^{i,j}$, and show that 
$$\b\left(\Sp^{i,j_1} \right)\subseteq \sqcup_{j =k}^{j_1} \Sm^{i,j} \mbox{ for $j_1 \in \llbracket k, N \rrbracket $}.$$

Suppose that $(\ell, n) \in \Sp^{i, n}$ for some $n \in \llbracket k, N\rrbracket$. We let $\tilde{n}$ be the largest index in $\llbracket \max(i,k+1), n\rrbracket$ such that $\lambda_i^n > \lambda_i^{\tilde{n} - 1}$ (recall that we use the convention $\lambda_{q+1}^q = - \infty$). By the interlacing condition in the definition of $\XX$ we have that $\lambda_i^n \geq \lambda_i^{m-1}$ for all $m \in \llbracket \max(i, k+1), n \rrbracket$ and if we have equality for all $m \in \llbracket \max(i, k+1), n \rrbracket$ we set $\tilde{n} = k$.

With the above choice of $\tilde{n}$ we define the configuration $\tilde{\ell} =(\tilde{\ell}^N, \dots, \tilde{\ell}^k)$ by 
$$\tilde{\ell}^m = \ell^m \mbox{ for $m \not \in \llbracket \tilde{n}, n \rrbracket$, and } \tilde{\ell}^m= (\ell^m_1, \dots, \ell^m_{i-1}, \ell^m_{i} - 1, \ell^m_{i+1} , \dots, \ell^m_{m}) \mbox{ for $m \in \llbracket \tilde{n}, n \rrbracket $}.$$

The above two  paragraphs define $(\tilde{\ell}, \tilde{n})$ and we let $\b(\ell, n) = (\tilde{\ell}, \tilde{n})$ for $(\ell, n) \in  \Sp^{i, n}$, see Figure \ref{S2_1}. We check that $(\tilde{\ell}, \tilde{n}) \in \Sm^{i, \tilde{n}}$, which would complete our work in this step as $\tilde{n} \leq n$ by construction. 

\begin{figure}[h]
\centering
  \scalebox{0.7}{\includegraphics[width=0.9\linewidth]{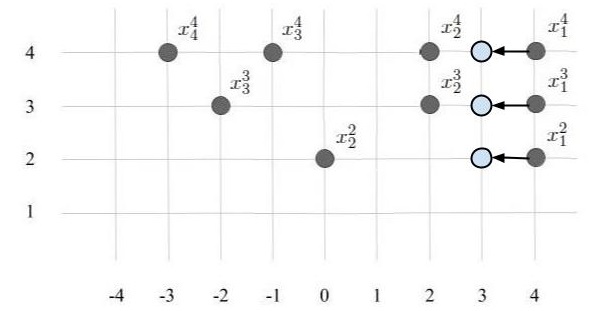}}
\captionsetup{width=.95\linewidth}
  \caption{The figure depicts the action of the map $\b$. To make the action comprehensible we explain how it works in the changed coordinates $x_i^j = \ell_i^j + i \cdot (\theta -1) $ -- this way the $x$'s all lie on the integer lattice and $x_1^j, \dots, x_j^j$ are distinct. The function $\b$ takes as input $(\ell,n)$ and in the changed coordinates looks at the longest string of particles $x^j_i$ for $k \leq j \leq n$ that have the same horizontal coordinate as $x_i^n$. This string is denoted by $x^{\tilde{n}}_i, \dots,x^{n}_i$ and $\b$ acts by shifting all these particles to the left by one. In the picture $N = n = 4$, $k =  \tilde{n} = 2$ and $i = 1$. }
  \label{S2_1}
\end{figure}

We need to show that $\tilde{\ell} \in \XX$ and that (\ref{PR2}) if $\tilde{n} = k$ or (\ref{PR1})  if $n \geq \tilde{n} \geq k+1$ have simple poles at $z = s$. 

To prove that $\tilde{\ell} \in \XX$ we need to show that $\tilde{\ell}$ satisfies the interlacing condition in the definition of $\XX$ and that $M \geq \tilde{\ell}_q^j + q \cdot \theta \geq 0$ for $k \leq j \leq N$ and $1 \leq q \leq j$. From the definition of $\tilde{\ell}$ we see that the only way that the interlacing condition can be violated is if one of the following holds:
\begin{enumerate}
\item for some $m \in \llbracket \tilde{n}, n \rrbracket$ we have $\ell_{i+1}^m \geq s - \theta$;
\item $n \leq N-1$ and $\ell_{i+1}^{n+1} \geq s-\theta$;
\item $\tilde{n} \geq k+1$ and $\ell_{i}^{\tilde{n}-1} \geq s$.
\end{enumerate}
Note that if either condition (1) or (2) held we would have by the interlacing condition in the definition of $\XX$ that either $n = N$ and $\ell_{i+1}^N = s- \theta$, in which case (\ref{PR4})  has no pole at $z =s $ due to the factor $z - \ell_{i+1}^N - \theta$ in the numerator, or $N-1 \geq n$ and $\ell_{i+1}^{n+1} = s-\theta$, in which case (\ref{PR3})  has no pole at $z = s$ due to the factor $z - \ell_{i+1}^{n+1} - \theta$ in the numerator. We conclude that if either condition (1) or (2) held we would obtain a contradiction with the assumption that $(\ell, n) \in \Sp^{i, n}$. We also see that condition (3) cannot hold by the minimality assumption in the definition of $\tilde{n}$. Overall, we conclude that $\tilde{\ell}$ satisfies the interlacing conditions.

We next show that $M \geq \tilde{\ell}_q^j + q \cdot \theta \geq 0$ for $k \leq j \leq N$ and $1 \leq q \leq j$. Since the latter inequalities hold for $\ell$, we see that it suffices to show that $s + i \cdot \theta \geq 1$. If $N -1 \geq n \geq k$ the latter is clear since from our earlier discussion we have $  s + i \cdot \theta - 1 \geq \ell_{i+1}^{n+1}  + (i+1)\cdot \theta \geq 0$. If $N = n$ and $n \geq i+1$ we have $  s + i \cdot \theta - 1 \geq \ell_{i+1}^{N}  + (i+1) \cdot \theta \geq 0$. Finally, if $N = n = i$ we have $s + i \cdot \theta \geq 1$ since in this case $s + N \cdot \theta \in \mathbb{Z}_{\geq 0}$ and we assumed $s \neq -N \cdot \theta$.\\

Our work in the last two paragraphs shows that $\tilde{\ell} \in \XX$. We next show that (\ref{PR2}) if $\tilde{n} = k$ or (\ref{PR1})  if $n \geq \tilde{n} \geq k+1$ have simple poles at $z = s$. 

Suppose first that $\tilde{n} = k$. In this case, we have that $1\leq i \leq k$ and $\ell_i^m = s$ for $m \in \llbracket k, n \rrbracket$ (this is by the definition of $\b(\ell, n) $). In particular, we see that in (\ref{PR2}) the denominator vanishes at $z = s$ due to the factor $z -\tilde{\ell}_i^k - 1$. In addition, the numerator in (\ref{PR2})  does not vanish at $z = s$. To see the latter note that if $1 \leq p \leq i-1$ we have
$$s - \tilde{\ell}_p^k + \theta - 1 =s - {\ell}_p^k + \theta - 1 \leq s - \ell_{i-1}^k + \theta - 1 \leq s - \ell_{i}^k -1 = -1.$$
Also for $k \geq p \geq i$ we have $s - \tilde{\ell}_p^k + \theta - 1 \geq s - \tilde{\ell}_i^k + \theta - 1 = \theta$.
So none of the factors in the numerator in (\ref{PR2}) vanish. In particular, we see that (\ref{PR2}) has a simple pole at $z = s$.

Suppose next that $n  \geq \tilde{n} \geq k+1$. In this case we have that $\ell^m_i = s$ for $m \in \llbracket \tilde{n}, n \rrbracket$ and $\ell^{\tilde{n}-1}_i \leq s - 1$. The latter implies that the denominator in (\ref{PR1}) has a simple zero at $z = s$ due to the factor $z - \tilde{\ell}_i^{\tilde{n}} - 1$. One also observes that the numerator in (\ref{PR1}) does not vanish. Indeed, if $1\leq p \leq i-1$ we have by the definition of $\b$ that
$$s - \tilde{\ell}_p^{\tilde{n} -1} + \theta - 1  =  s - {\ell}_p^{\tilde{n}-1} + \theta - 1 \leq s - {\ell}_{i-1}^{\tilde{n}-1} + \theta - 1 \leq s - {\ell}_{i}^{\tilde{n}}- 1 = -1.$$
If $\tilde{n}-1 \geq p \geq i$ we have by the definition of $\b$ that
$$s - \tilde{\ell}_p^{\tilde{n}-1} + \theta - 1 = s - {\ell}_p^{\tilde{n}-1} + \theta - 1 \geq s - \tilde{\ell}_{i}^{\tilde{n}-1} + \theta - 1 \geq \theta,$$
where we used $\tilde{\ell}_{i}^{\tilde{n}-1} \leq s - 1$. The latter observations show that (\ref{PR1}) has a simple pole at $z = s$. \\

{\bf \raggedleft Step 5.} In this step we show that the map $\b$ from Step 4 is a bijection from $\sqcup_{j = k}^N \Sp^{i,j}$ to $\sqcup_{j = k}^N \Sm^{i,j}$.

Suppose first that $\b(\ell, n) = \b(\rho, r) = (\tilde{\ell}, \tilde{n})$ for some $(\ell, n) , (\rho, r)  \in \sqcup_{j = k}^N \Sp^{i,j}$  and assume without loss of generality that $n \geq r$. We wish to show that $\ell = \rho$ and $n = r$. 

Suppose for the sake of contradiction that $n > r$. We have that $\ell_i^j = s$ for $j = \tilde{n}, \dots, n$ and so $\tilde{\ell}_i^{r+1} = s-1$ (here we used $n > r \geq \tilde{n}$). On the other hand, $\rho_{i}^{r+1} = \tilde{\ell}^{r+1}_i = s - 1$ (from $\b(\rho, r) = (\tilde{\ell}, \tilde{n})$ ) while $\rho_i^r = s$ (from the assumption $(\rho, r) \in \Sp^{i,r}$), which contradicts the fact that $\rho_{i}^{r+1} \geq \rho_i^r$. The contradiction arose from our assumption $n > r$ and so we conclude that $n = r$.

By definition, we have that $n \geq \tilde{n} \geq i $, and $\ell_p^j = \tilde{\ell}_p^j = \rho_p^j$ provided that $k \leq j \leq N$, $1 \leq p \leq j$ and $p \neq i$. The latter equality still holds if $p = i$ and $j \not \in \llbracket \tilde{n}, n\rrbracket$ and for $j \in \llbracket \tilde{n}, n\rrbracket$ we have $\ell_i^j = \tilde{\ell}_i^j +1 = \rho_i^j$. We conclude that $\rho = \ell$. This proves $(\ell, n) = (\rho, r) $ and so $\b$ is injective.\\

Suppose now that $ (\tilde{\ell}, \tilde{n}) \in \sqcup_{j = k}^N \Sm^{i,j}$. We wish to find $(\ell, n) \in  \sqcup_{j = k}^N\Sp^{i,j}$ with $\b(\ell, n) =  (\tilde{\ell}, \tilde{n})$.

Since $ (\tilde{\ell}, \tilde{n}) \in \sqcup_{j = k}^N \Sm^{i,j}$ we must have $(\tilde{\ell}, \tilde{n})  \in \Sm^{i,\tilde{n}}$, $1 \leq i \leq \tilde{n}$ and $\tilde{\ell}^{\tilde{n}}_i = s - 1$. Let $n$ be the largest index in $\llbracket \tilde{n}, N\rrbracket$ such that $s -  1 = \tilde{\ell}^{n}_i = \tilde{\ell}^n_i$. We define the configuration ${\ell} =({\ell}^N, \dots, {\ell}^k)$ through
$${\ell}^m = \tilde{\ell}^m \mbox{ for $m \not \in \llbracket \tilde{n}, n \rrbracket$, and } {\ell}^m= (\tilde\ell^m_1, \dots, \tilde\ell^m_{i-1}, \tilde\ell^m_{i} +1, \tilde\ell^m_{i+1} , \dots, \tilde\ell^m_{m}) \mbox{ for $m \in\llbracket \tilde{n}, n \rrbracket $}.$$
We next proceed to check that $(\ell,n) \in \Sp^{i, n}$ and $\b(\ell, n) =  (\tilde{\ell}, \tilde{n})$.

We need to show that ${\ell} \in \XX$ and that (\ref{PR4}) if $n = N$ or (\ref{PR3})  if $N-1 \geq n $ have simple poles at $z = s$. In order to prove that ${\ell} \in \XX$ we need to show that ${\ell}$ satisfies the interlacing condition in the definition of $\XX$ and that $M \geq {\ell}_q^j + q \cdot \theta \geq 0$ for $k \leq j \leq N$ and $1 \leq q \leq j$. From the definition of ${\ell}$ we see that the only way that the interlacing condition can be violated is if one of the following holds:
\begin{enumerate}
\item for some $m \in \llbracket \tilde{n}, n \rrbracket$ we have $\tilde{\ell}_{i-1}^m \leq s - 1 + \theta$;
\item $n \leq N-1$ and $\tilde{\ell}_{i}^{n+1} \leq s-1$;
\item $\tilde{n} \geq k+1$ and $\tilde{\ell}_{i-1}^{\tilde{n}-1} \leq s -1 + \theta$.
\end{enumerate}

Note that if either condition (1) or (3) held we would have by the interlacing condition in the definition of $\XX$ that either $\tilde{n} = k$ and $\ell_{i-1}^k = s-  1 + \theta$, in which case (\ref{PR2})  has no pole at $z =s $ due to the factor $z - \ell_{i-1}^k + \theta - 1$ in the numerator, or $\tilde{n} \geq k+1$ and $\tilde{\ell}_{i-1}^{\tilde{n}-1} = s - 1 + \theta$, in which case (\ref{PR1})  has no pole at $z = s$ due to the factor $z - \tilde{\ell}_{i-1}^{\tilde{n}-1} + \theta - 1$ in the numerator. We conclude that if either condition (1) or (3) held we would obtain a contradiction with the assumption that $(\tilde{\ell}, \tilde{n}) \in \Sm^{i, n}$. We also see that condition (2) cannot hold by the maximality assumption in the definition of $n$. Overall, we conclude that ${\ell}$ satisfies the interlacing conditions.

We next show that $M \geq {\ell}_q^j + q \cdot  \theta \geq 0$ for $k \leq j \leq N$ and $1 \leq q \leq j$. Since the latter inequalities hold for $\tilde{\ell}$, we see that it suffices to show that $s + i \cdot \theta \leq M$. If $i \geq 2$ the latter is clear since from our earlier discussion we have $  s + i \cdot \theta  \leq \tilde{\ell}_{i-1}^{\tilde{n}}  + (i-1) \cdot \theta  \leq M$. If $i = 1$ and $n \leq N-1$ we have $  s + \theta  \leq \tilde{\ell}_{1}^{n+1}  + \theta \leq M$. Finally, if $i = 1$ and $N = n$ we have $s + \theta \leq M$ since in this case $s + \theta \in \llbracket 1, M + 1\rrbracket$ and we assumed $s \neq s_M = M+1 - \theta$.\\

Our work in the last two paragraphs shows that ${\ell} \in \XX$. We next show that that (\ref{PR4}) if $n = N$ or (\ref{PR3})  if $N-1 \geq n $ have simple poles at $z = s$. Suppose that $N = n$. Then we have for $1 \leq p \leq i$ that 
$$s - \ell_p^N - \theta \leq s - \ell_i^N - \theta = -\theta,$$
while for $N \geq p \geq i + 1$ we have
$$s - \ell_p^N - \theta \geq s - \ell_{i+1}^N - \theta = s - \tilde\ell_{i+1}^N - \theta \geq s - \tilde{\ell}_{i}^N = 1.$$
The latter implies that none of the factors in the numerator in (\ref{PR4}) vanish and so the expression in (\ref{PR4}) has a simple pole at $z = s$.

Suppose next that $N -1 \geq n \geq \tilde{n}$. Then we have for $1 \leq p \leq i$ that 
$$s - \ell_{p}^{n+1} - \theta = s - \tilde{\ell}^{n+1}_p - \theta \leq s - \tilde{\ell}^{n+1}_i - \theta \leq - \theta,$$
where we used that $\tilde{\ell}^{n+1}_i \geq s$ by the definition of $n$. For $n + 1 \geq p \geq i + 1$ we have
$$s - \ell_{p}^{n+1} - \theta = s - \tilde{\ell}^{n+1}_p - \theta \geq s - \tilde{\ell}^{n+1}_{i+1} - \theta \geq  s - \tilde{\ell}^{n}_{i}  = 1.$$
The latter implies that none of the factors in the numerator in (\ref{PR3}) vanish and so the expression in (\ref{PR3}) a simple pole at $z = s$.\\

Our work so far shows that $(\ell, n) \in  \sqcup_{j = k}^N\Sp^{i,j}$. What remains is to show that $\b(\ell, n) =  (\tilde{\ell}, \tilde{n})$. The last statement is clear by the definitions of $(\ell, n)$ and $\b$ once we use $\tilde{\ell}^{\tilde{n}-1}_i \leq \tilde{\ell}^{\tilde{n}}_i = s - 1$. 

Our work in this step shows that $\b$ is both injective and surjective, hence bijective.\\

{\bf \raggedleft Step 6.} In this final step we show that (\ref{ER1}), (\ref{ER2}), (\ref{ER3}) and (\ref{ER4}) all hold.

We have that the right side of (\ref{ER1}) is equal to 
\begin{equation*}
\begin{split}
& \left[ \frac{ \phi_1^k(s)}{\phi_1^{N+1}(s)} \cdot  \frac{\P(\tilde{\ell})}{\P({\ell})} \cdot \prod_{p = 1, p \neq i}^{k}\frac{s- {\ell}^{k}_p + \theta - 1}{s - {\ell}^{k}_p - 1} \cdot \prod_{p = 1, p \neq i}^N\frac{s - \ell^N_p}{s- \ell^N_p -\theta}  - 1 \right] \cdot \phi_1^{N+1}(s) \P(\ell)    \\
& \times  \prod_{p = 1, p \neq i}^N\frac{s- \ell^N_p -\theta}{s - \ell^N_p}  = \phi_1^{N+1}(s) \P(\ell) \cdot   \prod_{p = 1, p \neq i}^N\frac{s- \ell^N_p -\theta}{s - \ell^N_p}  \left[ \frac{ \phi_1^k(s)}{\phi_1^{N+1}(s)} \cdot  \prod_{j = k}^{N} \frac{w_j(s-1)}{w_j(s)}  - 1 \right] = 0,
\end{split}
\end{equation*}
where in the first equality we used (\ref{TR4}) and in the second we used (\ref{eq:phi}). This proves (\ref{ER1}).

We have that the right side of (\ref{ER2}) is equal to 
\begin{equation*}
\begin{split}
&  \left[\frac{\phi_1^{\tilde{n}}(s)}{\phi_1^{N+1}(s)} \cdot \frac{\P(\tilde{\ell})}{\P({\ell})}  \cdot \prod_{p = 1, p \neq i}^{\tilde{n}} \frac{s- \tilde{\ell}^{\tilde{n}}_p -\theta}{s - \tilde{\ell}_p^{\tilde{n}} - 1}\prod_{p = 1}^{\tilde{n}-1}  \frac{s-  \tilde{\ell}_p^{\tilde{n}-1} + \theta - 1}{s -  \tilde{\ell}_p^{\tilde{n}-1}}\prod_{p = 1, p \neq i}^N\frac{s - \ell^N_p}{s- \ell^N_p -\theta}   - 1 \right] \cdot \phi_1^{N+1}(s)  \\
& \times \P(\ell) \hspace{-3mm} \prod_{p = 1, p \neq i}^N \hspace{-2mm} \frac{s- \ell^N_p -\theta}{s - \ell^N_p} = \left[\frac{\phi_1^{\tilde{n}}(s)}{\phi_1^{N+1}(s)} \cdot \prod_{j = \tilde{n}}^{N} \frac{w_j(s-1)}{w_j(s)} - 1   \right]\phi_1^{N+1}(s) \P(\ell) \hspace{-2mm}  \prod_{p = 1, p \neq i}^N\hspace{-2mm} \frac{s- \ell^N_p -\theta}{s - \ell^N_p} = 0 , 
\end{split}
\end{equation*}
where in the first equality we used (\ref{TR2}) and in the second we used (\ref{eq:phi}). This proves (\ref{ER2}).

We have that the right side of (\ref{ER3}) is equal to 
\begin{equation*}
\begin{split}
& \phi_1^{n+1}(s) \P(\ell)    \prod_{p = 1}^{n+1} \frac{s- \ell^{n+1}_p -\theta}{s - \ell_p^{n+1} - 1}\cdot \prod_{p = 1, p \neq i}^{n}  \frac{s-  \ell_p^{n} + \theta - 1}{s -  \ell_p^{n}} \times   \\
&  \left[ \frac{\phi_1^{\tilde{n}}(s)}{\phi_1^{n + 1}(s) }  \frac{\P(\tilde{\ell})}{ \P(\ell)  } \hspace{-1mm}  \prod_{p = 1, p \neq i}^{\tilde{n}} \hspace{-2mm} \frac{s- {\ell}^{\tilde{n}}_p -\theta}{s - {\ell}_p^{\tilde{n}} - 1} \prod_{p = 1}^{\tilde{n}-1}  \frac{s-  {\ell}_p^{\tilde{n}-1} + \theta - 1}{s -  {\ell}_p^{\tilde{n}-1}} \prod_{p = 1}^{n+1} \frac{s - \ell_p^{n+1} - 1}{s- \ell^{n+1}_p -\theta} \hspace{-1mm} \prod_{p = 1, p \neq i}^{n}\hspace{-2mm}  \frac{s -  \ell_p^{n}}{s-  \ell_p^{n} + \theta - 1} - 1\right]   \\
& =  \phi_1^{n+1}(s) \P(\ell)    \prod_{p = 1}^{n+1} \frac{s- \ell^{n+1}_p -\theta}{s - \ell_p^{n+1} - 1}\cdot \prod_{p = 1, p \neq i}^{n}  \frac{s-  \ell_p^{n} + \theta - 1}{s -  \ell_p^{n}} \left[\frac{\phi_1^{\tilde{n}}(s)}{\phi_1^{n+1}(s) }\prod_{j = \tilde{n}}^{n} \frac{w_j(s-1)}{w_j(s)} - 1  \right] = 0,\\
\end{split}
\end{equation*}
where in the first equality we used (\ref{TR1}) and in the second we used (\ref{eq:phi}). This proves (\ref{ER3}).

We have that the right side of (\ref{ER4}) is equal to 
\begin{equation*}
\begin{split}
& \phi_1^{n+1}(s) \P(\ell)    \prod_{p = 1}^{n+1} \frac{s- \ell^{n+1}_p -\theta}{s - \ell_p^{n+1} - 1}\cdot \prod_{p = 1, p \neq i}^{n}  \frac{s-  \ell_p^{n} + \theta - 1}{s -  \ell_p^{n}}  \\
& \times \left[ \frac{\phi_1^k(s)}{\phi_1^{n+1}(s)} \cdot  \frac{\P(\tilde{\ell})}{\P(\ell)  } \cdot \prod_{p = 1, p \neq i}^{k}\frac{s- \tilde{\ell}^{k}_p + \theta - 1}{s - \tilde{\ell}^{k}_p - 1}  \prod_{p = 1}^{n+1} \frac{s - \ell_p^{n+1} - 1}{s- \ell^{n+1}_p -\theta} \cdot \prod_{p = 1, p \neq i}^{n}  \frac{s -  \ell_p^{n}}{s-  \ell_p^{n} + \theta - 1} - 1\right]   \\
& = \phi_1^{n+1}(s) \P(\ell)    \prod_{p = 1}^{n+1} \frac{s- \ell^{n+1}_p -\theta}{s - \ell_p^{n+1} - 1} \prod_{p = 1, p \neq i}^{n}  \frac{s-  \ell_p^{n} + \theta - 1}{s -  \ell_p^{n}}  \left[  \frac{\phi_1^k(s)}{\phi_1^{n+1}(s)}  \cdot \prod_{j = k}^{n} \frac{w_j(s-1)}{w_j(s)} - 1 \right] = 0, \\
\end{split}
\end{equation*}
where in the first equality we used (\ref{TR3}) and in the second we used (\ref{eq:phi}). This proves (\ref{ER4}).\\

{\bf \raggedleft Analyticity of $R_2.$} The proof of analyticity of $R_2$ is similar to that for $R_1$. For clarity we split the proof into six steps.\\

{\bf \raggedleft Step 1.} The function $R_2$ has possible poles at $s =  a- b \cdot \theta$, where $b \in \llbracket 1, N \rrbracket$ and $a \in \llbracket 0, M + 1\rrbracket$. In this step we show that all of these poles are simple. We also show that the residues at $z = - N \cdot \theta$ and $z= s_M$ are equal to $0$, i.e. $R_2$ is analytic near these points. 

Firstly, note that the function $\R_2(z)$ only has simple poles at $z = - N \cdot \theta$ and $z = s_M$ by definition. Next, we have for each $j \in \llbracket k, N \rrbracket$ that $\ell_1^j > \ell_2^j > \cdots > \ell_j^j$ and so the expectations in the first line of (\ref{eq:mN2_gen}) can produce only simple poles. Similarly, each of the two products inside the expectation in the second line of (\ref{eq:mN2_gen}) can only have simple poles. Let us check that it can not happen that these two products share the same pole. Assuming the contrary, there must exist $j \in \llbracket k+1, N\rrbracket$ and $\ell \in \XX$ such that $s  + (N-j) \cdot \theta = \ell_q^j$ for some $q \in \llbracket 1, j\rrbracket$ and $s + (N-j+1) \cdot \theta = \ell_{r}^{j-1} + 1$ for some $r \in \llbracket 1, j - 1\rrbracket$. The last statement implies
$$ \ell^j_{q}-\ell^{j-1}_{r}=\lambda^j_q-\lambda^{j-1}_r +(r - q)\theta=1 - \theta.$$
Since $\lambda^{j-1}_r\geq \lambda^{j}_q$ for $r < q$ and $\lambda^{j}_q \geq \lambda^{j-1}_r$ when $q \leq  r$ the above equality can only hold if $r \geq q$, $\lambda^j_q=\lambda^{j-1}_r$  and $\theta = (r - q + 1)^{-1}$ (here we used that $\ell \in \XX$ and so $\ell^j \succeq \ell^{j-1}$ as in (\ref{S1GenState})). From the interlacing $\ell^j \succeq \ell^{j-1}$ we conclude that necessarily $\lambda^{j}_{r}=\lambda^{j-1}_r$, which implies $\ell^j_r = \ell^{j-1}_{r}.$ This produces an extra factor $z - s$ from the $z - \ell_{r}^{j} + (N-j+1) \theta -1$ in the first product in the second line of (\ref{eq:mN2_gen}), which means that the order of the pole at $s$ is reduced from $2$ to $1$. Thus indeed, all the poles are simple.\\

We next check that the residues at $z = - N \cdot \theta$ and $z= s_M$ are equal to $0$. Note that the residue at $z = s_M$ has two contributions, coming from the first expectation in the first line of (\ref{eq:mN2_gen}), when $\ell_1^N = s_M-1$, and from the first line of (\ref{NERem2}), and they cancel. Similarly, the residue at $z = -N \cdot \theta$ gets contributions from the second expectation in the first line of (\ref{eq:mN2_gen}), when $\ell^k_k = -k \cdot \theta$, from the expectations on the second line of (\ref{eq:mN2_gen}), when $\ell_j^j = - j \cdot \theta$, and from the second and third lines of (\ref{NERem2}), and they exactly cancel. This proves the analyticity of $R_2$ near $-N \cdot \theta$ and $s_M$.\\

{\bf \raggedleft Step 2.} In the next five steps we will compute the residue at any $s =  a - b \cdot \theta$ such that $s\not \in \{-N \cdot \theta, s_M\}$ and verify that it is zero. In this step we obtain a useful formula for this residue, which will be used later to show that it is zero.

We first expand all the expectations to get
\begin{equation}\label{eq:expandV2}
\begin{split}
&R_2(z)=\phi_2^{N+1}(z)  \sum_{\ell \in \XX}  \P(\ell)\cdot  \prod_{p = 1}^N\frac{z- \ell^N_p + \theta-1}{z - \ell^N_p - 1}   \\
& + \phi^k_2(z)  \sum_{\ell \in \XX} \P(\ell) \prod_{p = 1}^{k}\frac{z- \ell^{k}_p + (N-k-1)\theta}{z - \ell^{k}_p + (N-k) \theta}  +  \frac{\theta}{1-\theta}  \sum\limits_{j=k+1}^{N}\phi_2^{j}(z)  \\
& \times \sum_{\ell \in \XX } \P(\ell)  \prod_{p = 1}^{j} \frac{z- \ell^{j}_p+(N-j+1)\theta -1}{z - \ell_p^{j} +(N-j)\theta}    \prod_{p = 1}^{j-1}  \frac{z-\ell_p^{j-1} +(N-j)\theta}{z -  \ell_p^{j-1}+(N-j+1) \theta-1} -\R_2(z).
\end{split}
\end{equation}

For each $j \in \llbracket k, N-1\rrbracket$ and $i \in \llbracket 1, j \rrbracket$ we let $\Tm^{i,j}$ denote the set of pairs $(\ell,j)$, such that $\ell \in \XX$, $\ell^j_{j -i + 1} = s + (N-j) \cdot \theta- 1$ and
\begin{equation}\label{QR1}
\prod_{p = 1}^{j+1}(z- \ell^{j+1}_p+(N-j) \cdot \theta -1)   \prod_{p = 1}^{j}  \frac{1}{z -  \ell_p^{j}+(N-j) \cdot \theta-1}
\end{equation}
has a simple pole at $z= s$. If $i \geq j+1$ we let $\Tm^{i,j}$ denote the empty set.

For $i \in  \llbracket 1, N \rrbracket$ we let $\Tm^{i,N}$ be the sets of pairs $(\ell, N)$ such that $\ell \in \XX$, $\ell^N_{N - i + 1} = s - 1$ and 
\begin{equation}\label{QR2}
 \prod_{p = 1}^N\frac{z- \ell^N_p +\theta-1}{z - \ell^N_p-1} 
\end{equation}
has a simple pole at $z = s$. 

For each $j \in \llbracket k + 1, N \rrbracket$ and $i \in \llbracket 1, j \rrbracket$ we let $\Tp^{i,j}$ denote the set of pairs $(\ell,j)$, such that $\ell \in \XX$, $\ell^{j}_{j - i + 1} = s + (N-j) \cdot \theta$ and
\begin{equation}\label{QR3}
\prod_{p = 1}^{j} \frac{1}{z - \ell_p^{j} +(N-j) \cdot \theta}    \prod_{p = 1}^{j-1} (z-\ell_p^{j-1} +(N-j) \cdot\theta)
\end{equation}
has a simple pole at $z = s$. If $i \geq j + 1$ we let $\Tp^{i,j}$ denote the empty set. 

For $i  \in \llbracket 1, k \rrbracket $ we let $\Tp^{i,k}$ be the  sets of pairs $(\ell, k)$ with $\ell \in \XX$, $\ell_{k - i + 1}^k = s + (N-k) \cdot \theta$ and 
\begin{equation}\label{QR4}
\prod_{p = 1}^{k}\frac{z- \ell^{k}_p+(N-k-1)\cdot \theta }{z - \ell^{k}_p +(N-k) \cdot \theta}
\end{equation}
has a simple pole at $z = s$. If $i \geq k+1$ we let $\Tp^{i,k}$ denote the empty set. \\

In view of our work in Step 1, and (\ref{eq:expandV2}) we have the following formula for the residue at $z = s$
\begin{equation}\label{eq:expand2V2}
\begin{split}
&\sum_{i = 1}^N \phi_2^{N+1}(s)  \sum_{(\ell,N) \in \Tm^{i,N}}  \P(\ell)\cdot  \theta \cdot   \prod_{p = 1, p \neq N - i + 1}^N\frac{s- \ell^N_p +\theta-1}{s - \ell^N_p-1} \hspace{2mm} \\
& + \sum_{i = 1}^N  \phi^k_2(s)  \sum_{(\ell, k) \in \Tp^{i,k}}  \P(\ell) \cdot (-\theta) \cdot \prod_{p = 1, p \neq k - i + 1}^{k}\frac{s- \ell^{k}_p+(N-k-1) \cdot \theta }{s - \ell^{k}_p +(N-k) \cdot \theta}     \\
&+ \frac{\theta}{1-\theta} \sum_{ i = 1}^N  \sum_{j=k+1}^{N}\phi_2^{j}(s) \sum_{(\ell,j) \in \Tp^{i,j} } \hspace{-2mm} \P(\ell)  (\theta - 1)  \\
&  \times \prod_{p = 1, p \neq j -  i + 1}^{j}\hspace{-1mm} \frac{s- \ell^{j}_p+(N-j+1) \cdot \theta -1}{s - \ell_p^{j} +(N-j) \cdot \theta}\cdot \prod_{p = 1}^{j-1}   \frac{s-\ell_p^{j-1} +(N-j) \cdot \theta}{s -  \ell_p^{j-1}+(N-j+1)\cdot  \theta-1}  \\
&+ \frac{\theta}{1-\theta} \sum_{ i = 1}^N  \sum_{j=k}^{N-1}\phi_2^{j + 1}(s) \sum_{(\ell,j ) \in \Tm^{i,j} } \P(\ell)  (1 - \theta)   \\
& \times \prod_{p = 1}^{j + 1}  \frac{s- \ell^{j+1}_p+(N-j) \cdot \theta -1}{s - \ell_p^{j+1} +(N-j-1) \cdot \theta}\cdot \prod_{p = 1, p \neq j - i + 1}^{j}   \frac{s-\ell_p^{j} +(N-j-1) \cdot \theta}{s -  \ell_p^{j}+(N-j) \cdot \theta-1},
\end{split}
\end{equation}  
where we used that $\R_2(z)$ is analytic away from the points $-N \cdot \theta$ and $s_M$ and that $s\not \in \{- N \cdot \theta, s_M\}$ by assumption (hence $\R_2(z)$ does not contribute to the residue at $z = s$). We also mention that our convention is that the sum over an empty set is zero.

Equation (\ref{eq:expand2V2}) is the main output of this step and below we show that the total sum is zero.\\

{\bf \raggedleft Step 3.} The goal of this step is to show that the sum over all terms in (\ref{eq:expand2V2}) for a fixed $i \in \llbracket 1, N \rrbracket$ vanishes, which would of course imply that the full sum vanishes. To accomplish this we construct a bijection $\bb$ 
\begin{equation}\label{FR0}
\bb: \sqcup_{j = k}^N \Tp^{i,j} \rightarrow  \sqcup_{j = k}^N \Tm^{i,j}, \mbox{ such that }\bb \left(\Tp^{i,j_1} \right)\subseteq \sqcup_{j =j_1}^N\Tm^{i,j}\mbox{ for $j_1 \in \llbracket k, N \rrbracket$ and }
\end{equation}
 if $(\tilde{\ell}, \tilde{n}) = \bb(\ell, n)$ the following four equations all hold. We first have that

\begin{equation}\label{FR1}
0 = \phi_2^{N+1}(s) \P(\tilde\ell)     \hspace{-6mm} \prod_{p = 1, p \neq N - i + 1}^N \hspace{-2mm} \frac{s- \tilde{\ell}^N_p +\theta-1}{s - \tilde{\ell}^N_p-1} -  \phi_2^k(s)  \P({\ell})   \hspace{-6mm} \prod_{p = 1, p \neq k - i + 1}^{k}\frac{s- \ell^{k}_p+(N-k-1)\theta }{s - \ell^{k}_p +(N-k)\theta} ,
\end{equation}
if $(\ell, n) \in \Tp^{i, k}$ and $(\tilde{\ell}, \tilde{n}) \in \Tm^{i,N}$. Secondly, we have
\begin{equation}\label{FR2}
\begin{split}
0 = \hspace{2mm}&\phi_2^{N+1}(s) \P(\tilde{\ell})   \prod_{p = 1, p \neq N - i + 1}^N\frac{s- \tilde{\ell}^N_p +\theta-1}{s - \tilde{\ell}^N_p-1} - \phi_2^{n}(s) \P({\ell})    \\
& \times  \prod_{p = 1, p \neq n - i+ 1}^{n} \frac{s- \ell^{n}_p+(N-n+1)\theta -1}{s - \ell_p^{n} +(N-n)\theta}    \prod_{p = 1}^{n-1}  \frac{s-\ell_p^{n-1} +(N-n)\theta}{s -  \ell_p^{n-1}+(N-n+1) \theta-1},
\end{split}
\end{equation}
if $(\ell, n) \in \Tp^{i, n}$ and $(\tilde{\ell}, \tilde{n}) \in \Tm^{i, N}$ for $N \geq n \geq k+1$. Thirdly, we have
\begin{equation}\label{FR3}
\begin{split}
&0 =\phi_2^{\tilde{n} + 1}(s) \P(\tilde\ell)  \cdot  \prod_{p = 1}^{\tilde{n}+1} \frac{s- \ell^{\tilde{n}+1}_p+(N-\tilde{n})\theta -1}{s - \ell_p^{\tilde{n}+1} +(N-\tilde{n} - 1)\theta}    \prod_{p = 1, p \neq \tilde{n} -  i + 1}^{\tilde{n}}  \frac{s-\ell_p^{\tilde{n}} +(N-\tilde{n} - 1)\theta}{s -  \ell_p^{\tilde{n}}+(N-\tilde{n}) \theta-1}   \\
& - \phi_2^{n}(s) \P(\ell)  \cdot  \prod_{p = 1, p \neq n - i+ 1}^{n} \frac{s- \ell^{n}_p+(N-n+1)\theta -1}{s - \ell_p^{n} +(N-n)\theta}    \prod_{p = 1}^{n-1}  \frac{s-\ell_p^{n-1} +(N-n)\theta}{s -  \ell_p^{n-1}+(N-n+1) \theta-1},
\end{split}
\end{equation}
if $(\ell, n) \in \Tp^{i,n}$ and $(\tilde{\ell}, \tilde{n}) \in \Tm^{i, \tilde{n}}$ with $N - 1 \geq \tilde{n} \geq n \geq k + 1$. Finally, we have
\begin{equation}\label{FR4}
\begin{split}
0 = \hspace{2mm}&\phi_2^{\tilde{n} + 1}(s) \P(\tilde\ell)  \cdot  \prod_{p = 1}^{\tilde{n}+1} \frac{s- \ell^{\tilde{n}+1}_p+(N-\tilde{n})\theta -1}{s - \ell_p^{\tilde{n}+1} +(N-\tilde{n} - 1)\theta}    \prod_{p = 1, p \neq \tilde{n} - i + 1}^{\tilde{n}}  \frac{s-\ell_p^{\tilde{n}} +(N-\tilde{n} - 1)\theta}{s -  \ell_p^{\tilde{n}}+(N-\tilde{n}) \theta-1}  \\
& - \phi_2^k(s)  \P({\ell})  \cdot \prod_{p = 1, p \neq k -  i + 1}^{k}\frac{s- \ell^{k}_p+(N-k-1)\theta }{s - \ell^{k}_p +(N-k)\theta} ,
\end{split}
\end{equation}
if $(\ell, n) \in \Tp^{i,k}$ and $(\tilde{\ell}, \tilde{n}) \in \Tm^{i,\tilde{n}}$ for $N-1 \geq \tilde{n} \geq k$.

If we can find a bijection $\bb$ as in (\ref{FR0}) that satisfies (\ref{FR1}), (\ref{FR2}), (\ref{FR3}) and (\ref{FR4}) we would conclude that the sum in (\ref{eq:expand2V2}) is zero since the contributions from the first, fifth and sixth line in (\ref{eq:expand2V2}) exactly cancel with those from the second, third and fourth line. \\

We next focus on constructing the map $\bb$, showing that it is a bijection of the sets in (\ref{FR0}) and that it satisfies equations (\ref{FR1}), (\ref{FR2}), (\ref{FR3}) and (\ref{FR4}).\\

{\bf \raggedleft Step 4.} In this step we construct the map $\bb$ on $\sqcup_{j = k}^N \Tp^{i,j}$, and show that 
$$\bb\left(\Tp^{i,j_1} \right)\subseteq \sqcup_{j =j_1}^{N} \Tm^{i,j}\mbox{ for $j_1 \in \llbracket k, N \rrbracket$}.$$

Suppose that $(\ell, n) \in \Tp^{i, n}$ for some $n \in \llbracket k, N\rrbracket$. We let $\tilde{n}$ be the largest index in $\llbracket n, N\rrbracket$ such that $\lambda_{n-i+1}^n = \lambda_{\tilde{n}-i+1}^{\tilde{n}}$. Note that by the interlacing condition in the definition of $\XX$ we have that $\lambda_{n-i+1}^n \geq\lambda_{m-i+1}^{m}$ for all $m \in \llbracket n, N \rrbracket$.

With the above choice of $\tilde{n}$ we define the configuration $\tilde{\ell} =(\tilde{\ell}^N, \dots, \tilde{\ell}^k)$ by 
$$\tilde{\ell}^m = \ell^m \mbox{ for $m \not \in \llbracket {n}, \tilde{n} \rrbracket$, and } $$
$$ \tilde{\ell}^m= (\ell^m_1, \dots, \ell^m_{m-i}, \ell^m_{m- i+1} - 1, \ell^m_{m - i + 2} , \dots, \ell^m_{m}) \mbox{ for $m \in \llbracket {n}, \tilde{n} \rrbracket $}.$$

The above two  paragraphs define $(\tilde{\ell}, \tilde{n})$ and we let $\bb(\ell, n) = (\tilde{\ell}, \tilde{n})$ for $(\ell, n) \in  \Tp^{i, n}$, see Figure \ref{S2_2}. We check that $(\tilde{\ell}, \tilde{n}) \in \Tm^{i, \tilde{n}}$, which would complete our work in this step as $\tilde{n} \geq n$ by construction.\\
\begin{figure}[h]
\centering
  \scalebox{0.7}{\includegraphics[width=0.9\linewidth]{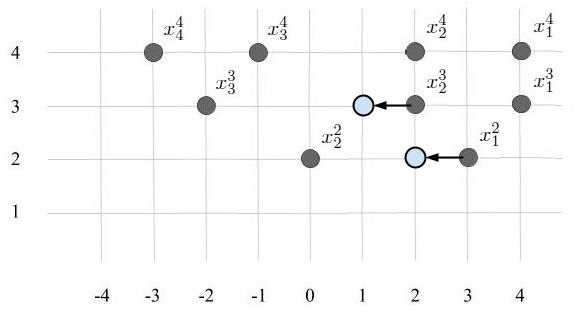}}
\captionsetup{width=.95\linewidth}
  \caption{The figure depicts the action of the map $\b$. To make the action comprehensible we explain how it works in the changed coordinates $x_i^j = \ell_i^j + i \cdot (\theta -1) $ -- this way the $x$'s all lie on the integer lattice and $x_1^j, \dots, x_j^j$ are distinct. The function $\bb$ takes as input $(\ell,n)$ and in the changed coordinates looks at the longest string of particles $x^j_{j -i + 1}$ for $N \geq j \geq n$ that lie on the line of slope $-1$ and above $x^n_{n-i+1}$. This string is denoted by $x^{n}_{n- i + 1}, \dots,x^{\tilde{n}}_{\tilde{n} - i+1}$ and $\bb$ acts by shifting all these particles to the left by one. In the picture we have $N  = 4$, $k = n = 2$, $\tilde{n} = 3$ and $i = 2$. }
  \label{S2_2}
\end{figure}

We need to show that $\tilde{\ell} \in \XX$ and that (\ref{QR2}) if $\tilde{n} = N$ or (\ref{QR1}) if $N-1 \geq \tilde{n} \geq k$ have simple poles at $z = s$. 

To prove that $\tilde{\ell} \in \XX$ we need to show that $\tilde{\ell}$ satisfies the interlacing condition in the definition of $\XX$ and that $M \geq \tilde{\ell}_q^j + q \cdot \theta \geq 0$ for $k \leq j \leq N$ and $1 \leq q \leq j$. From the definition of $\tilde{\ell}$ we see that the only way that the interlacing condition can be violated is if one of the following holds:
\begin{enumerate}
\item for some $m \in \llbracket {n}, \tilde{n} \rrbracket$ we have $\lambda_{m-i + 2}^m \geq s + (N - i + 1)\cdot \theta$;
\item $\tilde{n} \leq N-1$ and $\lambda_{\tilde{n} - i + 2}^{\tilde{n} + 1} \geq s + (N - i + 1) \cdot \theta$;
\item $n \geq k+1$ and $\lambda_{n - i +1}^{n-1} \geq s + (N - i + 1) \cdot \theta$.
\end{enumerate}
Note that if either condition (1) or (3) held we would have by the interlacing condition in the definition of $\XX$ that either $n = k$ and $\lambda_{n - i+2}^n = s + (N - i + 1) \cdot \theta$, in which case (\ref{QR4})  has no pole at $z =s $ due to the factor $z - \ell_{k - i+2}^k + (N-k-1) \cdot \theta$ in the numerator, or $n \geq k+1$ and $\lambda_{n- i + 1}^{n-1} = s + (N - i + 1)\cdot \theta$, in which case (\ref{QR3})  has no pole at $z = s$ due to the factor $z - \ell_{n- i + 1}^{n-1} + (N-n) \cdot \theta$ in the numerator. We conclude that if either condition (1) or (3) held we would obtain a contradiction with the assumption that $(\ell, n) \in \Tp^{i, n}$. We also see that condition (2) cannot hold by the maximality assumption in the definition of $\tilde{n}$. Overall, we conclude that $\tilde{\ell}$ satisfies the interlacing conditions.

We next show that $M \geq \tilde{\ell}_q^j + q \cdot \theta \geq 0$ for $k \leq j \leq N$ and $1 \leq q \leq j$. Since the latter inequalities hold for $\ell$, we see that it suffices to show that $s+ (N - i + 1) \cdot \theta \geq 1$. If $N -1 \geq \tilde{n} \geq k$ the latter is clear since from our earlier discussion we have $ s+ (N - i + 1) \cdot \theta -1 \geq \lambda_{\tilde{n} - i + 2}^{\tilde{n}+1} \geq 0$. If $N = \tilde{n}$ and $i \geq 2$ we have $   s+ (N-i+1) \cdot \theta -1 \geq \lambda_{N - i + 2}^N \geq \lambda_N^N\geq 0$. Finally, if $N = \tilde{n} $ and $i = 1$ we have $s+ (N -i + 1) \cdot \theta \geq 1$ since in this case $s \in \mathbb{Z}_{\geq 0}$ and we assumed $s \neq 0$.\\

Our work in the last two paragraphs shows that $\tilde{\ell} \in \XX$. We next show that (\ref{QR2}) if $\tilde{n} = N$ or (\ref{QR1})  if $N-1 \geq \tilde{n} \geq k$ have simple poles at $z = s$. 

Suppose first that $\tilde{n} = N$. In this case, we have that $\lambda_{m - i +1}^m = s + (N - i + 1) \cdot \theta$ for $m \in \llbracket n, N \rrbracket$ (this is by the definition of $\bb(\ell, n) $). In particular, we see that in (\ref{QR2}) the denominator vanishes at $z = s$ due to the factor $z -\tilde{\ell}_{N - i + 1}^N - 1$. In addition, the numerator in (\ref{QR2})  does not vanish at $z = s$. To see the latter note that if $1 \leq p \leq N - i$ we have
$$s - \tilde{\ell}_p^N + \theta - 1 =s - {\ell}_p^N + \theta - 1 \leq s - \ell_{N-i}^N + \theta - 1 \leq s - \ell_{N - i + 1}^N -1 = -1.$$
Also for $N \geq p \geq  N-i + 1$ we have $s - \tilde{\ell}_p^N + \theta - 1 \geq s - \tilde{\ell}_{N - i + 1}^N + \theta - 1 = \theta$.
So none of the factors in the numerator in (\ref{QR2}) vanish. In particular, we see that (\ref{QR2}) has a simple pole at $z = s$.

Suppose next that $N-1  \geq \tilde{n} \geq k$. In this case we have that $\lambda^m_{m - i + 1} = s + (N-i+1) \cdot \theta$ for $m = n, \dots, \tilde{n}$ and $\lambda^{\tilde{n}+1}_{\tilde{n} -i + 2} \leq s + (N - i + 1) \cdot \theta - 1$. The latter implies that the denominator in (\ref{QR1}) has a simple zero at $z = s$ due to the factor $z - \tilde{\ell}_{\tilde{n} - i + 1}^{\tilde{n}} + (N - \tilde{n}) \cdot \theta - 1 $. One also observes that the numerator in (\ref{QR1}) does not vanish. Indeed, if $1\leq p \leq \tilde{n} - i +1$ we have by the definition of $\bb$
$$s - \tilde{\ell}_p^{\tilde{n} + 1} + (N-\tilde{n}) \cdot \theta - 1  =  s - {\ell}_p^{\tilde{n}+ 1} +(N-\tilde{n}) \cdot \theta - 1 \leq s - {\ell}^{\tilde{n}+1}_{\tilde{n} - i + 1} +(N-\tilde{n}) \cdot \theta - 1 = - 1.$$
If $\tilde{n}+1 \geq p \geq \tilde{n} - i + 2 $ we have by the definition of $\bb $ that
$$s - \tilde{\ell}_p^{\tilde{n} + 1} + (N-\tilde{n}) \cdot \theta - 1  = s - {\ell}_p^{\tilde{n} + 1} + (N-\tilde{n}) \cdot \theta - 1   \geq s - \tilde{\ell}^{\tilde{n}+1}_{ \tilde{n} - i +2} + (N-\tilde{n}) \cdot \theta - 1   \geq \theta,$$
where we used $\lambda^{\tilde{n}+1}_{\tilde{n} - i +2} \leq s + (N - i + 1) \cdot \theta - 1$. The latter observations show that (\ref{QR1}) has a simple pole at $z = s$. \\

{\bf \raggedleft Step 5.} In this step we show that the map $\bb$ from Step 4 is a bijection from $\sqcup_{j = k}^N \Tp^{i,j}$ to $\sqcup_{j = k}^N \Tm^{i,j}$.

Suppose first that $\bb(\ell, n) = \bb(\rho, r) = (\tilde{\ell}, \tilde{n})$ for some $(\ell, n) , (\rho, r)  \in \sqcup_{j = k}^N \Tp^{i,j}$ and assume without loss of generality that $n \geq r$. We wish to show that $\ell = \rho$ and $n = r$. 

Suppose for the sake of contradiction that $n > r$. We have that $\rho_{j-i + 1}^j = s + (N -  j ) \cdot \theta$ for $j \in \llbracket r , \tilde{n} \rrbracket$ and so $\tilde{\ell}_{n - i  }^{n-1} = s + (N - n +1)\cdot \theta-1$ (here we used $\tilde{n} \geq n > r$). On the other hand, $\ell_{n- i}^{n-1} = \tilde{\ell}^{n-1}_{n-i} = s + (N - n +1) \cdot \theta-1$ (from $\bb(\ell, n) = (\tilde{\ell}, \tilde{n})$ ) while $\ell_{n-i+1}^{n} = s + (N-n) \cdot \theta$ (from the assumption $(\ell, n) \in \Tp^{i,n}$), which contradicts the fact that $\ell_{n-i}^{n-1} - \theta \geq \ell_{n-i+1}^n$. The contradiction arose from our assumption $n > r$ and so we conclude that $n = r$.

By definition, we have that $\tilde{n} \geq n \geq i $, and $\ell_p^j = \tilde{\ell}_p^j = \rho_p^j$ provided that $k \leq j \leq N$, $1 \leq p \leq j$ and $p \neq  j - i + 1$. The latter equality still holds if $p =j - i + 1$ and $j \not \in \llbracket {n},  \tilde{n}\rrbracket$ and for $j  \in \llbracket {n},  \tilde{n}\rrbracket$ we have $\ell_{j - i + 1}^j = \tilde{\ell}_{j - i + 1}^j +1 = \rho_{j - i + 1}^j$. We conclude that $\rho = \ell$. This proves $(\ell, n) = (\rho, r) $ and so $\bb$ is injective.\\

Suppose now that $ (\tilde{\ell}, \tilde{n}) \in \sqcup_{j = k}^N \Tm^{i,j}$. We wish to find $(\ell, n) \in  \sqcup_{j = k}^N \Tp^{i,j}$ with $\bb(\ell, n) =  (\tilde{\ell}, \tilde{n})$.

Let $n$ be the smallest index in $\llbracket \max( i, k+1), \tilde{n}\rrbracket$ such that $ \tilde{\ell}^{n-1}_{ n - i } + (n - i ) \cdot \theta > \tilde{\ell}^{\tilde{n}}_{\tilde{n} - i+ 1} + (\tilde{n} - i+ 1) \cdot \theta$. By the interlacing condition in the definition of $\XX$ we have that $ \tilde{\ell}^{m-1}_{ m - i } + (m - i ) \cdot \theta   \geq \tilde{\ell}^{\tilde{n}}_{\tilde{n} - i+ 1} + (\tilde{n} - i+ 1) \cdot \theta$ for all $m \in \llbracket \max( i, k+1), \tilde{n}\rrbracket$ and if we have equality for all $m \in \llbracket \max( i, k+1), \tilde{n}\rrbracket$ we set $n = k$. 

With the above choice of $n$ we define the configuration ${\ell} =({\ell}^N, \dots, {\ell}^k)$ through
$${\ell}^m = \tilde{\ell}^m \mbox{ for $m \not \in \llbracket {n},  \tilde{n}\rrbracket$, and } $$
$$ {\ell}^m= (\tilde{\ell}^m_1, \dots, \tilde{\ell}^m_{m - i}, \tilde{\ell}^m_{m -i  + 1} +1 , \tilde{\ell}^m_{m - i + 2} , \dots, \tilde{\ell}^m_{m}) \mbox{ for $m \in \llbracket {n},  \tilde{n}\rrbracket $}.$$
We next proceed to check that $(\ell,n) \in \Tp^{i, n}$ and $\bb(\ell, n) =  (\tilde{\ell}, \tilde{n})$.

We need to show that ${\ell} \in \XX$ and that (\ref{QR4}) if $n = k$ or (\ref{QR3})  if $n \geq k+1 $ have simple poles at $z = s$. In order to prove that ${\ell} \in \XX$ we need to show that ${\ell}$ satisfies the interlacing condition in the definition of $\XX$ and that $M \geq {\ell}_q^j +  q \cdot \theta \geq 0$ for $k \leq j \leq N$ and $1 \leq q \leq j$. From the definition of ${\ell}$ we see that the only way that the interlacing condition can be violated is if one of the following holds:
\begin{enumerate}
\item for some $m \in \llbracket {n},  \tilde{n}\rrbracket$ we have $\tilde{\ell}_{m - i }^m \leq s  + (N- m + 1) \cdot \theta - 1$;
\item $n \geq k+1$ and $\tilde{\ell}_{n - i  }^{n-1} \leq s + (N- n +1) \cdot \theta - 1$;
\item $\tilde{n} \leq  N - 1$ and $\tilde{\ell}_{\tilde{n} - i +1}^{\tilde{n}+1} \leq s+ (N - \tilde{n}) \cdot \theta  -1$.
\end{enumerate}
Note that if either condition (1) or (3) held we would have by the interlacing condition in the definition of $\XX$ that either $\tilde{n} = N$ and $\tilde{\ell}_{N-i}^N= s + \theta - 1$, in which case (\ref{QR2})  has no pole at $z =s $ due to the factor $z - \tilde{\ell}_{N-i}^N + \theta - 1$ in the numerator, or $N-1\geq \tilde{n} $ and $\tilde{\ell}_{\tilde{n} - i + 1}^{\tilde{n}+1} = s+ (N - \tilde{n}) \cdot \theta  -1$, in which case (\ref{QR1})  has no pole at $z = s$ due to the factor $z - \tilde{\ell}_{\tilde{n} - i + 1}^{\tilde{n}+1} + (N - \tilde{n}) \cdot \theta  -1$ in the numerator. We conclude that if either condition (1) or (3) held we would obtain a contradiction with the assumption that $(\tilde{\ell}, \tilde{n}) \in \Tm^{i, \tilde{n}}$. We also see that condition (2) cannot hold by the minimality assumption in the definition of $n$. Overall, we conclude that ${\ell}$ satisfies the interlacing conditions.

We next show that $M \geq {\ell}_p^j + p \cdot \theta \geq 0$ for $k \leq j \leq N$ and $1 \leq p \leq j$. Since the latter inequalities hold for $\tilde{\ell}$, we see that it suffices to show that $s + (N - i + 1) \cdot \theta \leq M$. If $\tilde{n}- 1 \geq i$ the latter is clear since from our earlier discussion we have $ s + (N - i + 1) \cdot \theta  \leq \tilde{\ell}_{\tilde{n} - i}^{\tilde{n}} - (\tilde{n} - i) \cdot \theta  \leq M$. If $i = \tilde{n}$ and $\tilde{n} \leq N-1$ we have $ s + (N - i + 1) \cdot \theta  \leq \tilde{\ell}_{1}^{\tilde{n}+1}  - \theta \leq M$. Finally, if $i = \tilde{n} = N$ we have $s + (N - i + 1) \cdot \theta  \leq M$ since in this case $s + (N - i + 1) \cdot \theta \in \llbracket 1, M + 1\rrbracket$ and we assumed $s \neq s_M$.\\

Our work in the last two paragraphs shows that ${\ell} \in \XX$. We next show that that (\ref{QR4}) if $n = k$ or (\ref{QR3})  if $n \geq k + 1$ have simple poles at $z = s$. Suppose that $k = n$. Then we have for $1 \leq p \leq k - i + 1$ 
$$s - \ell_p^k + (N - k -1) \cdot \theta \leq s - \ell_{k - i + 1}^k  + (N - k -1) \cdot \theta = -\theta,$$
while for $k \geq p \geq k - i + 2$ we have
$$s - \ell_p^k + (N - k -1) \cdot \theta  \geq s - \ell_{k - i + 2}^k + (N - k -1) \cdot \theta = s - \tilde{\ell}_{k - i + 2}^k + (N - k -1) \cdot \theta \geq $$
$$  s - \tilde{\ell}_{k - i + 1 }^k + (N - k) \cdot \theta = 1.$$
The latter implies that none of the factors in the numerator in (\ref{QR4}) vanish and so the expression in (\ref{QR4}) has a simple pole at $z = s$, coming from the factor $z - \ell_{k - i + 1}^k + (N-k) \cdot \theta$.

Suppose next that $\tilde{n} \geq n \geq k + 1$. Then we have for $1 \leq p \leq n - i $ that 
$$s - \ell_{p}^{n-1} + (N-n) \cdot \theta  = s - \tilde{\ell}_{p}^{n-1} + (N-n) \cdot \theta \leq s - \tilde{\ell}_{n - i}^{n-1} + (N-n) \cdot \theta \leq -\theta,$$
where we used that $\tilde{\ell}_{n - i}^{n-1} -  (N- n + 1) \cdot \theta \geq s$ by the definition of $n$. For $n -1 \geq p \geq  n - i + 1 $
$$s - \ell_{p}^{n-1} + (N-n) \cdot \theta =s - \tilde{\ell}_{p}^{n-1} + (N-n) \cdot \theta  \geq s - \tilde{\ell}_{n - i + 1}^{n-1} + (N-n) \cdot \theta \geq  s - \tilde{\ell}_{n -i + 1}^{n} + (N-n)\cdot \theta = 1.$$
The latter implies that none of the factors in the numerator in (\ref{QR3}) vanish and so the expression in (\ref{QR3}) a simple pole at $z = s$, coming from the factor $z - \ell_{n - i + 1 }^n + (N -n) \cdot \theta$.\\

Our work so far shows that $(\ell, n) \in  \sqcup_{j = k}^N \Tp^{i,j}$. What remains is to show that $\bb(\ell, n) =  (\tilde{\ell}, \tilde{n})$. The last statement is clear by the definitions of $(\ell, n)$ and $\bb$ once we use $\tilde{\ell}^{\tilde{n}+1}_{\tilde{n} - i+2} \leq \tilde{\ell}^{\tilde{n}}_{\tilde{n} - i + 1} = s - 1$. 

Our work in this step shows that $\bb$ is both injective and surjective, hence bijective.\\

{\bf \raggedleft Step 6.} In this final step we show that (\ref{FR1}), (\ref{FR2}), (\ref{FR3}) and (\ref{FR4}) all hold.

We have that the right side of (\ref{FR1}) is equal to 
\begin{equation*}
\begin{split}
& \left[ \frac{\phi_2^{N+1}(s) \P(\tilde\ell) }{ \phi_2^k(s)  \P({\ell})  }   \cdot \prod_{p = 1, p \neq N - i + 1}^N\frac{s- {\ell}^N_p +\theta-1}{s - {\ell}^N_p-1} \cdot \prod_{p = 1, p \neq k - i + 1}^{k}\frac{s - \ell^{k}_p +(N-k) \cdot \theta}{s- \ell^{k}_p+(N-k-1) \cdot \theta }  - 1 \right]   \\
&  \times \phi_2^k(s)  \P({\ell})   \cdot  \prod_{p = 1, p \neq k - i + 1}^{k} \hspace{-2mm}\frac{s- \ell^{k}_p+(N-k-1) \cdot \theta }{s - \ell^{k}_p +(N-k) \cdot \theta}  = \left[ \frac{\phi_2^{N+1}(s) }{ \phi_2^k(s)   } \cdot \prod_{j = k }^{N} \dfrac{w_{j}(s + (N-j) \cdot \theta-1)}{w_{j}(s + (N-j) \cdot \theta)}  - 1 \right]   \\
& \times \phi_2^k(s)  \P({\ell}) \prod_{p = 1, p \neq i}^{k}\frac{s- \ell^{k}_p+(N-k-1) \cdot \theta }{s - \ell^{k}_p +(N-k) \cdot \theta}  = 0,
\end{split}
\end{equation*}
where in the first equality we used (\ref{UR4}) and in the second we used (\ref{eq:phi}). This proves (\ref{FR1}).

We have that the right side of (\ref{FR2}) is equal to 
\begin{equation*}
\begin{split}
& \phi_2^{n}(s) \P({\ell}) \prod_{p = 1, p \neq n - i+ 1}^{n} \frac{s- \ell^{n}_p+(N-n+1) \cdot \theta -1}{s - \ell_p^{n} +(N-n) \cdot \theta}    \prod_{p = 1}^{n-1}  \frac{s-\ell_p^{n-1} +(N-n) \cdot \theta}{s -  \ell_p^{n-1}+(N-n+1) \cdot \theta-1}  \\
& \times \Bigg[ \frac{\phi_2^{N+1}(s) \P(\tilde{\ell})}{ \phi_2^{n}(s) \P({\ell}) } \hspace{-2mm}  \prod_{p = 1, p \neq n - i+ 1}^{n} \frac{s - \ell_p^{n} +(N-n) \cdot \theta}{s- \ell^{n}_p+(N-n+1) \cdot \theta -1}    \prod_{p = 1}^{n-1}  \frac{s -  \ell_p^{n-1}+(N-n+1) \cdot \theta-1}{s-\ell_p^{n-1} +(N-n) \cdot \theta}     \\
& \times \hspace{-2mm} \prod_{p = 1, p \neq N - i + 1}^N\frac{s- {\ell}^N_p +\theta-1}{s - {\ell}^N_p-1}- 1 \Bigg] =  \phi_2^{n}(s) \P({\ell}) \prod_{p = 1, p \neq n - i+ 1}^{n} \frac{s- \ell^{n}_p+(N-n+1) \cdot \theta -1}{s - \ell_p^{n} +(N-n) \cdot \theta}   \\
& \times \prod_{p = 1}^{n-1}  \frac{s-\ell_p^{n-1} +(N-n) \cdot \theta}{s -  \ell_p^{n-1}+(N-n+1) \cdot \theta-1} \cdot \left[\frac{\phi_2^{N+1}(s)}{ \phi_2^{n}(s)} \prod_{j = n }^{N} \dfrac{w_{j}(s + (N-j) \cdot \theta-1)}{w_{j}(s + (N-j)\cdot \theta)}  - 1\right] = 0.
\end{split}
\end{equation*}
where in the first equality we used (\ref{UR2}) and in the second we used (\ref{eq:phi}). This proves (\ref{FR2}).

We have that the right side of (\ref{FR3}) is equal to 
\begin{equation*}
\begin{split}
&\phi_2^{n}(s) \P(\ell)  \cdot  \prod_{p = 1, p \neq n - i+ 1}^{n} \frac{s- \ell^{n}_p+(N-n+1) \cdot\theta -1}{s - \ell_p^{n} +(N-n) \cdot \theta}    \prod_{p = 1}^{n-1}  \frac{s-\ell_p^{n-1} +(N-n) \cdot \theta}{s -  \ell_p^{n-1}+(N-n+1) \cdot \theta-1}  \\
&\times \Bigg[ \frac{\phi_2^{\tilde{n} + 1}(s) \P(\tilde\ell) }{\phi_2^{n}(s) \P(\ell)  } \cdot  \prod_{p = 1}^{\tilde{n}+1} \frac{s- \ell^{\tilde{n}+1}_p+(N-\tilde{n}) \cdot \theta -1}{s - \ell_p^{\tilde{n}+1} +(N-\tilde{n} - 1) \cdot \theta}    \prod_{p = 1, p \neq \tilde{n} -  i + 1}^{\tilde{n}}  \frac{s-\ell_p^{\tilde{n}} +(N-\tilde{n} - 1) \cdot \theta}{s -  \ell_p^{\tilde{n}}+(N-\tilde{n}) \cdot \theta-1}  \\
& \times  \prod_{p = 1, p \neq n - i+ 1}^{n} \frac{s - \ell_p^{n} +(N-n) \cdot \theta}{s- \ell^{n}_p+(N-n+1) \cdot \theta -1}    \prod_{p = 1}^{n-1}  \frac{s -  \ell_p^{n-1}+(N-n+1) \cdot \theta-1}{s-\ell_p^{n-1} +(N-n) \cdot \theta}  - 1 \Bigg]  \\
&= \phi_2^{n}(s) \P(\ell)  \cdot  \prod_{p = 1, p \neq n - i+ 1}^{n} \frac{s- \ell^{n}_p+(N-n+1) \cdot \theta -1}{s - \ell_p^{n} +(N-n) \cdot \theta}    \prod_{p = 1}^{n-1}  \frac{s-\ell_p^{n-1} +(N-n) \cdot \theta}{s -  \ell_p^{n-1}+(N-n+1) \cdot \theta-1}  \\
& \times \left[ \frac{\phi_2^{\tilde{n}+1}(s)}{ \phi_2^{n}(s)} \prod_{j = n }^{\tilde{n}} \dfrac{w_{j}(s + (N-j) \cdot \theta-1)}{w_{j}(s + (N-j) \cdot \theta)}  - 1\right] = 0,
\end{split}
\end{equation*}
where in the first equality we used (\ref{UR1}) and in the second we used (\ref{eq:phi}). This proves (\ref{FR3}).

We have that the right side of (\ref{FR4}) is equal to 
\begin{equation*}
\begin{split}
&   \Bigg[ \frac{\phi_2^{\tilde{n} + 1}(s) \P(\tilde\ell)}{\phi_2^k(s)  \P({\ell}) }  \cdot  \prod_{p = 1}^{\tilde{n}+1} \frac{s- \ell^{\tilde{n}+1}_p+(N-\tilde{n})\cdot \theta -1}{s - \ell_p^{\tilde{n}+1} +(N-\tilde{n} - 1) \cdot \theta}    \prod_{p = 1, p \neq \tilde{n} - i + 1}^{\tilde{n}}  \frac{s-\ell_p^{\tilde{n}} +(N-\tilde{n} - 1) \cdot \theta}{s -  \ell_p^{\tilde{n}}+(N-\tilde{n}) \cdot \theta-1}  \\
& \times \prod_{p = 1, p \neq k -  i + 1}^{k}\frac{s - \ell^{k}_p +(N-k) \cdot \theta}{s- \ell^{k}_p+(N-k-1) \cdot \theta } - 1 \Bigg] =  \phi_2^k(s)  \P({\ell})  \cdot \prod_{p = 1, p \neq k -  i + 1}^{k}\frac{s- \ell^{k}_p+(N-k-1) \cdot \theta }{s - \ell^{k}_p +(N-k) \cdot \theta} \\
& = \left[ \frac{\phi_2^{\tilde{n}+1}(s)}{ \phi_2^{k}(s)} \prod_{j = k }^{\tilde{n}} \dfrac{w_{j}(s + (N-j) \cdot \theta-1)}{w_{j}(s + (N-j) \cdot \theta)}  - 1\right] \cdot \prod_{p = 1, p \neq k -  i + 1}^{k}\frac{s- \ell^{k}_p+(N-k-1) \cdot \theta }{s - \ell^{k}_p +(N-k) \cdot \theta}  = 0,
\end{split}
\end{equation*}
where in the first equality we used (\ref{UR3}) and in the second we used (\ref{eq:phi}). This proves (\ref{FR4}).
\end{proof}

%-------------------------------------------------------------------------------------------------------------------------------------------------------------------------------------------------
%    Section 2.3
%
%-------------------------------------------------------------------------------------------------------------------------------------------------------------------------------------------------
\subsection{Multi-level Nekrasov equations: $\theta = 1$}\label{Section2.3} In this section we present the multi-level Nekrasov equations for the case $\theta = 1$. 

\begin{theorem}\label{MainNek2} Let $\mathbb{P}^{1,M}_{N,k}$ be a complex measure as in (\ref{S2PDef}) for $\theta =1$, $ N \in \mathbb{N}$, $k \in \llbracket 1, N \rrbracket$, $M \in \mathbb{Z}_{\geq 0}$. Let $\mathcal{M} \subseteq \mathbb{C}$ be an open set and $[- N, M] \subseteq \mathcal{M}$. Suppose that there exist functions $\phi_r^{ j}$ for $j=k,\dots, N + 1$, $r=1,2$ that are analytic in $\mathcal{M}$ and such that for any $j \in \llbracket k, N \rrbracket$
\begin{equation}\label{eq:phi_s}
\begin{split}
\frac{\phi^{j+1}_1(z)}{ \phi^{j}_1(z)} &=  \frac{w_j(z-1)}{w_j(z)} \mbox{ for $z \in [1 -j, M - 1]$};\\
 \frac{\phi_2^{j}(z)}{ \phi_2^{j+1}(z)} &=\dfrac{w_{j}(z + N - j - 1)}{w_{j}(z + N - j)} \mbox{ for $z \in [1 - N, M - N +j - 1]$}.
\end{split}
\end{equation}
Then the following functions $R_1(z)$, $R_2(z)$ are analytic in $\mathcal{M}$:
\begin{equation}\label{eq:mN1_simple}
\begin{split}
R_1 (z):= \hspace{2mm}&\phi_1^{N+1}(z) \cdot \mathbb{E} \left[ \prod_{p = 1}^N\frac{z- \ell^N_p -1}{z - \ell^N_p} \right] + \phi_1^k(z) \cdot \mathbb{E} \left[ \prod_{p = 1}^{k}\frac{z- \ell^{k}_p }{z - \ell^{k}_p - 1} \right]   \\
& + \sum\limits_{j=k+1}^{N}\phi^{j}_1(z) \cdot \mathbb{E} \left[  \sum_{p = 1}^{j} \frac{1}{z - \ell_p^{j} - 1} - \sum_{p = 1}^{j-1} \frac{1}{z -  \ell_p^{j-1}} \right]-\R_1(z);
\end{split}
\end{equation}

\begin{equation}\label{eq:mN2_simple}
\begin{split}
&R_2 (z):= \phi_2^{N+1}(z)  \cdot \mathbb{E}\hspace{-1mm} \left[ \prod_{p = 1}^N\frac{z- \ell^N_p}{z - \ell^N_p-1} \right] \hspace{-1mm} + \phi_2^k(z) \cdot \mathbb{E}\hspace{-1mm} \left[ \prod_{p = 1}^{k}\frac{z- \ell^{k}_p+(N-k-1) }{z - \ell^{k}_p +(N-k)} \right] \\
&  + \sum\limits_{j=k+1}^{N} \phi^{ j}_2(z) \cdot \mathbb{E}\left[ - \sum_{p = 1}^{j} \frac{1}{z - \ell_p^{j} +(N-j)}    +\sum_{p = 1}^{j-1}  \frac{1}{z -  \ell_p^{j-1}+(N-j)} \right]-\R_2(z),
\end{split}
\end{equation}
where the functions $\R_1(z)$, $\R_2(z)$ are given by
\begin{equation}\label{NERem1_s}
\begin{split}
&\R_1(z)= \dfrac{(-1) \phi^{N+1}_1(-N)}{z + N }  \cdot \mathbb{E}\left[\prod\limits_{p=1}^{N-1}\dfrac{\ell^N_p+N+1}{\ell^N_p +N } \cdot {\bf 1}\{  \ell^N_N= -N  \}\right]\\
&+ \dfrac{\phi_1^k(M)  }{z-M} \cdot \mathbb{E}\left[\prod\limits_{p=2}^{k}\dfrac{M-\ell^k_p}{M-\ell^k_p-1}\cdot {\bf 1}\{ \ell^k_1=M-1 \} \right] + \sum\limits_{j=k+1}^{N} \hspace{-1mm}\frac{  \phi^{j}_1(M)  }{z - M} \cdot \mathbb{E}\left[ {\bf 1 } \{ \ell^j_1=M-1 \} \right],
\end{split}
\end{equation}
\begin{equation}\label{NERem2_s}
\begin{split}
& \R_2(z)=  \frac{ \phi_2^{N+1}(M)}{z - M} \cdot  \mathbb{E}\left[\prod\limits_{p=2}^{N}\dfrac{M-\ell^N_p}{M-\ell^N_p-1}\cdot {\bf 1 }\{ \ell^N_1=M-1 \}\right]  \\
& + \frac{(-1)\phi_2^k(-N  ) }{z + N} \cdot \mathbb{E}\left[\prod\limits_{p=1}^{k-1}\dfrac{\ell^k_p + k-1}{\ell^k_p + k} \cdot 
      {\bf 1 }\{\ell^k_{k}= -k\} \right] \hspace{-1mm}+ \hspace{-1mm} \sum \limits_{j=k+1}^{N} \frac{(-1)\phi^{j}_2(-N)}{z + N} \cdot \mathbb{E}\left[ {\bf 1}\{ \ell^{j}_{j}= - j \} \right].
    \end{split}
\end{equation}
\end{theorem}
\begin{remark}\label{RemBoundary2}
If we have that $\phi^{N+1}_1(-N)=0$ and $ \phi^{j}_1(M) = 0$ for $j = 1, \dots, N$ then $\R_1(z) = 0$. Analogously, if $\phi_2^{N+1}(M ) = 0$ and $\phi^j_2(-N ) = 0$ for $j = 1, \dots, N$ then $\R_2(z) = 0$.
\end{remark}
\begin{proof} We will deduce the theorem from Theorem \ref{MainNek1} by performing an appropriate $\theta \rightarrow 1$ limit transition. We begin with the analyticity of $R_1$.

From the first line in (\ref{eq:phi_s}) we know that $w_j$ for $j \in \llbracket k, N \rrbracket$ have (unique) analytic continuations to a complex neighborhood $U \subseteq \mathcal{M}$ of $[-N, M]$, which we continue to call $w_j$. Let $\theta$ be sufficiently close to $1$ so that $[-N \cdot \theta, M+ 1 - \theta] \subseteq U$ and write $\P$ for the measure as in (\ref{S2PDef}) with the weights $w_j$ we just introduced. It follows from (\ref{eq:mN1_gen}) that 
\vspace{-1mm}
\begin{equation}\label{S23E1}
\begin{split}
R^{\theta}_1 (z):= \hspace{2mm}&\phi_1^{N+1}(z) \cdot \mathbb{E}^{\theta} \left[ \prod_{p = 1}^N\frac{z- \ell^N_p -\theta}{z - \ell^N_p} \right] + \phi_1^k(z) \cdot \mathbb{E}^{\theta}  \left[ \prod_{p = 1}^{k}\frac{z- \ell^{k}_p + \theta - 1}{z - \ell^{k}_p - 1} \right]   \\
&+ \frac{\theta}{1-\theta}  \sum\limits_{j=k+1}^{N}\phi^{j}_1(z) \cdot \mathbb{E}^{\theta} \left[  \prod_{p = 1}^{j} \frac{z- \ell^{j}_p  -\theta}{z - \ell_p^{j} - 1}  \prod_{p = 1}^{j-1} \frac{z-  \ell_p^{j-1} + \theta - 1}{z -  \ell_p^{j-1}} \right]-\R^{\theta}_1(z),
\end{split}
\end{equation}
is analytic in $U$. In (\ref{S23E1}) we have inserted $\theta$ into the notation to indicate the dependence of the expressions on it. In particular, the expectations above are with respect to $\P$, rather than $\mathbb{P}^{1,M}_{N,k}$. Subtracting $G^{\theta}_1(z) := \frac{\theta}{1-\theta}\sum_{j=k+1}^{N}\phi^{j}_1(z)  $
from both sides of (\ref{S23E1}) and letting $\theta \rightarrow 1$, we see that the right side converges uniformly over compact subsets of $U \setminus [-N, M]$ to the right side of (\ref{eq:mN1_simple}). In particular, we see that for $z \in U \setminus [-N, M]$ we have 
\begin{equation}\label{S23E2}
\lim_{\theta \rightarrow 1} [R_1^{\theta}(z) - G^{\theta}_1(z)]  = R_1(z).
\end{equation}

Let $\epsilon > 0$ be sufficiently small so that $V_{\epsilon} = \{z \in \mathbb{C}: d(z,[-N, M]) < \epsilon \} \subseteq U$ and let $\gamma$ denote a positively oriented contour that encloses $V_{\epsilon}$, and is contained in $U$. From Cauchy's theorem and (\ref{S23E2}) we have for $z \in V_{\epsilon/2} $
$$\lim_{\theta \rightarrow 1} [R_1^{\theta}(z) - G^{\theta}_1(z)] = \lim_{\theta \rightarrow 1} \frac{1}{2\pi \i}\int_{\gamma} \frac{R_1^{\theta}(\zeta) - G^{\theta}_1(\zeta)}{\zeta -z} = \frac{1}{2\pi \i}\int_{\gamma} \frac{R_1(\zeta) }{\zeta -z},$$
and the latter convergence is uniform on $V_{\epsilon/2} $. From Theorem \ref{MainNek1} for $\theta \neq 1$ we know that $R_1^{\theta}$ are analytic and $G^{\theta}_1$ are analytic by our assumption on $\phi_1^j$. This means that 
$$\frac{1}{2\pi \i}\int_{\gamma} \frac{R_1(\zeta) }{\zeta -z}$$
defines an analytic function on $V_{\epsilon/2}$ as the uniform limit of analytic functions, see \cite[Chapter 2, Theorem 5.2]{SS}. In addition, by (\ref{S23E2}) we know that this function agrees with $ R_1(z)$ on $V_{\epsilon/2} \setminus [-N, M].$ The latter shows that $R_1(z)$, which from (\ref{eq:mN1_simple}) is clearly analytic in $\mathcal{M} \setminus [-M, N]$,  has an analytic continuation to $\mathcal{M}$. Since from (\ref{eq:mN1_simple}) we know that $R_1(z)$ is meromorhpic on $\mathcal{M}$, we conclude that it is in fact analytic there as desired.

\vspace{2mm}

For the analyticity of $R_2$ we argue as follows. From the second line in (\ref{eq:phi_s}) we know that $w_j$ for $j \in \llbracket k, N \rrbracket$ have (unique) analytic continuations to a complex neighborhood $U \subseteq \mathcal{M}$ of $[-N, M]$, which we continue to call $w_j$. Let us define $w_j^{\theta}$ through 
$$w_j^{\theta}(z + \theta (N-j)) = w_j(z+ N - j),$$ 
and write $\P$ for the measure as in (\ref{S2PDef}) with the weights $w^{\theta}_j$ we just introduced. We see that $w_j^{\theta}$ satisfy the second set of equalities in (\ref{eq:phi}) in $U$ provided $\theta$ is close enough to $1$ and so from (\ref{eq:mN2_gen}) 
\vspace{-1mm}
\begin{equation}\label{S23E3}
\begin{split}
&R^{\theta}_2 (z):= \phi_2^{N+1}(z)  \mathbb{E}^{\theta}\hspace{-1mm} \left[ \prod_{p = 1}^N\frac{z- \ell^N_p +\theta-1}{z - \ell^N_p-1} \right] \hspace{-1mm} + \phi_2^k(z)  \mathbb{E}^{\theta}\hspace{-1mm} \left[ \prod_{p = 1}^{k}\frac{z- \ell^{k}_p+(N-k-1)\theta }{z - \ell^{k}_p +(N-k)\theta} \right] \hspace{-1mm} + \hspace{-1mm} \frac{\theta}{1-\theta} \\
& \times \hspace{-2mm} \sum\limits_{j=k+1}^{N}\hspace{-2mm} \phi^{ j}_2(z)\mathbb{E}^{\theta}\left[  \prod_{p = 1}^{j} \frac{z- \ell^{j}_p+(N-j+1)\theta -1}{z - \ell_p^{j} +(N-j)\theta}    \prod_{p = 1}^{j-1}  \frac{z-\ell_p^{j-1} +(N-j)\theta}{z -  \ell_p^{j-1}+(N-j+1) \theta-1} \right]-\R^{\theta}_2(z),
\end{split}
\end{equation}
is analytic in $U$. As before we have reflected the dependence on $\theta$ of the expressions above in the notation. We now subtract $G^{\theta}_2(z) := \frac{\theta}{1-\theta}\sum_{j=k+1}^{N}\phi^{j}_2(z)  $ from both sides of (\ref{S23E3}), let $\theta \rightarrow 1$ and see that the right side converges uniformly over compact subsets of $U \setminus [-N, M]$ to (\ref{eq:mN2_simple}). From here we can repeat the argument for analyticity of $R_1$ verbatim to get the analyticity of $R_2$. 
 \end{proof}

%-------------------------------------------------------------------------------------------------------------------------------------------------------------------------------------------------
% Section 3
%
%-------------------------------------------------------------------------------------------------------------------------------------------------------------------------------------------------
\section{Application of Nekrasov equations}\label{Section3} In this section we continue with the notation from Sections \ref{Section1} and \ref{Section2}. In Section \ref{Section3.1} we summarize various basic properties of joint cumulants. In Section \ref{Section3.2} we consider measures on Gelfand-Tsetlin patterns as in (\ref{S1PDef}) and explain how to use the multi-level Nekrasov equations from Sections \ref{Section2.2} and \ref{Section2.3} to obtain equations relating joint observables on several levels, see Lemma \ref{CrudeCumExp}. 

%-------------------------------------------------------------------------------------------------------------------------------------------------------------------------------------------------
% Section 3.1
%
%-------------------------------------------------------------------------------------------------------------------------------------------------------------------------------------------------
\subsection{Joint cumulants}\label{Section3.1} In this section we summarize some notation and results about joint cumulants of random variables. For more background we refer the reader to \cite[Chapter 3]{Taqqu}.

For $n$ bounded complex-valued random variables $X_1, \dots, X_n$ we let $M(X_1, \dots, X_n)$ denote their joint cumulant. Explicitly, we have 
\begin{equation}\label{CumDef}
M(X_1, \dots, X_n) = \frac{\partial^n}{\partial z_1 \cdots \partial z_k} \log \mathbb{E} \left[ \exp \left( \sum_{i = 1}^n z_i \cdot X_i \right) \right] \Bigg{\vert}_{z_1 = \cdots = z_k = 0}.
\end{equation}
If $n = 1$, then $M(X_1) = \mathbb{E}[X_1]$. For every subset $J = \{j_1, \dots, j_k \} \subseteq \llbracket 1, n \rrbracket$ we write 
$$\bX_J = \{ X_{j_1}, \dots, X_{j_k} \} \mbox{ and } \bX^J = X_{j_1} \times\cdots \times X_{j_k},$$
where $\times$ denotes the usual product. If $Y$ is a bounded random variable we also write 
$$M(\bX_J) = M (X_{j_1}, \dots, X_{j_k}) \mbox{, } \mathbb{E}\left[ \bX^J \right] = \mbox{$\mathbb{E} \left[ \prod_{j \in J} X_j \right]$}, \mbox{ and } M(Y;\bX_J) = M(Y, X_{j_1}, \dots, X_{j_k}) $$
where we remark that the definitions make sense as the joint cumulant of a set of random variables $\bX_J$ is invariant with respect to permutations of $J$.

We record the following basic properties of $M(X_1, \dots, X_n)$, see \cite[Section 3.1]{Taqqu}. If $Y$ is a bounded complex-valued random variable and $c_1, \dots, c_n \in \mathbb{C}$ we have 
\begin{equation}\label{S3Linearity}
\begin{split}
&M(X_1 + Y,  X_2, \dots,  X_n )  = M(X_1, \dots, X_n)  +   M(Y, X_2, \dots, X_n) , \\
&M(c_1 X_1, \dots, c_n X_n) = \prod_{i = 1}^n c_i \cdot M(X_1, \dots, X_n) \\
&M( X_1 + c_1, \dots,  X_n + c_n) = M(X_1, \dots, X_n) \mbox{ if $n \geq 2$, while }M(X_1 + c_1) = M(X_1) + c_1.\\
\end{split}
\end{equation}

From \cite[Proposition 3.2.1]{Taqqu} we have the following identities that relate joint moments and joint cumulants of the bounded complex-valued random variables $X_1, \dots, X_n$. 
\begin{lemma}\label{CumToMom} For any $J \subseteq \llbracket 1, n \rrbracket$ we have
\begin{equation}\label{Mal1}
\begin{split}
&\mathbb{E} \left[ \bX^J \right] = \sum_{ \pi = \{ b_1, \dots, b_k \} \in \mathcal{P}(J)} M(\bX_{b_1}) \cdots M(\bX_{b_k})\\
\end{split}
\end{equation}
\begin{equation}\label{Mal2}
\begin{split}
&M(\bX_J) = \sum_{\sigma = \{a_1, \dots, a_r\} \in \mathcal{P}(J)} (-1)^{r-1} (r-1)! \cdot \mathbb{E}\left[ \bX^{a_1} \right] \cdots \mathbb{E} \left[ \bX^{a_r} \right],
\end{split}
\end{equation}
where $\mathcal{P}(J)$ is the set of partitions of $b$, see \cite[Section 2.2]{Taqqu}. If $X,Y$ are bounded complex-valued random variables we also have
\begin{equation}\label{Mal3}
\begin{split}
&M(XY,X_1, \dots, X_n) = M(X,Y,X_1, \dots, X_n) + \sum_{I \subseteq \llbracket 1, n \rrbracket} M(X;\bX_I ) \cdot M(Y;\bX_{I^c} ),
\end{split}
\end{equation}
where $I^c = \llbracket 1, n \rrbracket \setminus I$.
\end{lemma}

%-------------------------------------------------------------------------------------------------------------------------------------------------------------------------------------------------
% Section 3.2
%
%-------------------------------------------------------------------------------------------------------------------------------------------------------------------------------------------------
\subsection{Deformed measures}\label{Section3.2}
Let $\P$ be as in (\ref{S1PDef}), and assume that $(\ell^N, \dots, \ell^k) \in \XX$ is distributed according to $\P$. For $n \in \llbracket k, N \rrbracket$ and $L \in (0,\infty)$ we define
\begin{equation}\label{S3DefG}
G^n_L(z) = \sum_{i = 1}^{n} \frac{1}{z - \ell_i^n/L},
\end{equation}
to be the Stieltjes transform of the random measure $\mu^n_L = \sum_{i = 1}^n \delta( \ell_i^n/L)$. The first goal of this section is to find expressions for the joint cumulants of $G_L^{k}, \dots, G_L^{N}$. We do this in Lemma \ref{LemCum1}. Afterwards we use our multi-level Nekrasov equations, Theorems \ref{MainNek1} and \ref{MainNek2}, to derive integral equations that relate the joint cumulants of $G_L^{k}, \dots, G_L^{N}$ in Lemma \ref{CrudeCumExp}.

Let us fix $m_j \in \mathbb{Z}_{\geq 0}$ for $j \in \llbracket  k, N \rrbracket$. For each $j \in \llbracket  k, N \rrbracket$ we fix parameters $\tb^j = (t^j_1, \dots, t^j_{m_{j}})$ and $\vm^j = (v^j_1, \dots, v^j_{m_j})$ and denote the whole sets of parameters by $\tb = (\tb^{k}, \dots, \tb^{N})$ and $\vm = (\vm^{k}, \dots, \vm^{N})$. We assume that the parameters $\tb ,\vm$ are such that $v^j_i + t^j_i - y \neq 0$ and $v_i^j - y \neq 0$ for all $j \in \llbracket k, N \rrbracket$, $i = 1, \dots, m_j$, and all $y \in [- N \theta/ L, (M  - \theta)/ L]$.

With the above data we define the deformed distribution $\mathbb{P}^{\tb, \vm}$ on $\XX$ through
\begin{equation} \label{eq:distrgen_deformed}
\begin{split}
\mathbb{P}^{\tb, \vm}(\ell)=Z(\tb, \vm)^{-1} \P(\ell) \cdot  \prod_{j = k}^N \prod_{r = 1}^{j} \prod_{i = 1}^{m_j} \left( 1 + \frac{t_i^j}{v_i^j - \ell_r^j/L} \right).
\end{split}
\end{equation}
 If $\sum_{j = k}^N m_j = 0$ we have that $\mathbb{P}^{\tb, \vm} = \P$ is the undeformed measure. In general, $\mathbb{P}^{\tb, \vm}$ may be a complex-valued measure but we always choose the normalization constant $Z(\tb, \vm)$ so that $\sum_{\ell \in \XX} \mathbb{P}^{\tb, \vm}(\ell) = 1$. In addition, we require that the numbers $t^j_i$ are sufficiently close to zero so that $Z(\tb, \vm) \neq 0$. 

Recall from Section \ref{Section3.1} that if $X_1, \dots, X_n$ are bounded complex-valued random variables, we denote their joint cumulant by $M(X_1, \dots, X_n)$, and if $A$ is a set of bounded random variables $A = \{X_1, \dots, X_n\}$ and $X$ is another bounded random variable we write $M(X; A)$ for the joint cumulant $M(X, X_1, \dots, X_n)$.

The definition of the deformed measure $\mathbb{P}^{\tb, \vm}$ is motivated by the following observation. 
\begin{lemma}\label{LemCum1} Let $\xi$ be a bounded random variable. If $\mathbb{P}^{\tb, \vm}$ is as in (\ref{eq:distrgen_deformed}) we have
\begin{equation}\label{eq:derivative_k}
\left(\prod_{j = k}^N \prod_{i = 1}^{m_j} \frac{\partial}{\partial t_i^j} \right) \mathbb E_{\mathbb{P}^{\tb, \vm}}\left[\xi\right]\bigg\rvert_{t^j_i = 0} = M( \xi; \{ G_L^j(v_i^j): j \in \{ k, \dots, N \}, i = 1, \dots, m_j\}),
\end{equation}
where the right side is the joint cumulant of the given random variables with respect to $\P$.
\end{lemma}
\begin{remark}
The above result is analogous to \cite[Lemma 2.4]{BGG}, which in turn is based on earlier related work in random matrix theory. We present a proof below for the sake of completeness.
\end{remark}
\begin{proof}
Recall from (\ref{CumDef}) that the joint cumulant of $\xi ,\xi^{k}_1, \dots, \xi^{k}_{m_k}$, $\xi^{k+1}_1, \dots, \xi^{k+1}_{m_{k+1}}$, $\dots,\xi^{N}_1, \dots, \xi^{N}_{m_N} $ is given by
$$ \left(\prod_{j = k}^N \prod_{i = 1}^{m_j} \frac{\partial}{\partial t_i^j} \right) \frac{\partial}{\partial t_0} \log \left(\mathbb{E} \exp \left(t_0 \xi +  \sum_{j = k}^N \sum_{i = 1}^{m_j} t^j_i \xi^j_i \right) \right) \Bigg\rvert_{t_0= 0, t^j_i=0}. $$
Performing the differentiation with respect to $t_0$ we can rewrite the above as
$$\left(\prod_{j \in k}^N \prod_{i = 1}^{m_j} \frac{\partial}{\partial t_i^j} \right)  \frac{\mathbb{E}  \left[\xi \exp \left(  \sum_{ j = k}^N \sum_{i = 1}^{m_j} t^j_i \xi^j_i \right) \right] }{\mathbb{E}  \left[ \exp \left(  \sum_{ j = k}^N \sum_{i = 1}^{m_j} t^j_i \xi^j_i\right)  \right]} \Bigg \rvert_{t^j_i=0}. $$
Set $\xi^j_i = G^j_L(v^j_i)$ for $j \in \{k, \dots, N \}$ and $i = 1, \dots, m_j$ and observe that 
$$\exp \left( t G^j_L(z)\right) = \prod_{i = 1}^j \left(1+ \frac{t}{z-\ell^j_i/L} \right) + O(t^2).$$
The above statements imply the statement of the lemma.
\end{proof}

In the remainder of this section we utilize our multi-level Nekrasov equations, Theorems \ref{MainNek1} and \ref{MainNek2}, to derive integral equations that relate the joint cumulants of $G_L^{k}, \dots, G_L^{N}$. We proceed to make a simplifying assumption about the measures $\P$, summarized in the following definition.
\begin{definition}\label{DefAss}
We assume that there exists an open set $\mathcal{M}  \subseteq \mathbb{C}$, such that $[- N\theta, M + 1 -\theta] \subseteq \mathcal{M}$. In addition, we require the existence of holomorphic functions $\Phi^+_N, \Phi^-_N$ on $\mathcal{M}$ such that 
\begin{equation}\label{eqPhiN}
\begin{split}
&\frac{w_N(x)}{w_N(x-1)}=\frac{\Phi_N^+(x)}{\Phi_N^-(x)},
\end{split}
\end{equation} 
whenever $x \in [ -N\theta +1, M -\theta]$. We mention that since $w_N(x) > 0 $ by assumption we must have that $\Phi_N^\pm(x)$ are non-vanishing, and also that the choice of $\Phi_N^\pm(x)$, saitsfying (\ref{eqPhiN}), is not unique as we can multiply $\Phi_N^\pm(x)$ by the same non-vanishing analytic function on $\mathcal{M}$ and still satisfy (\ref{eqPhiN}). 
\end{definition}

\begin{lemma}\label{CrudeCumExp} Suppose that $\P$ is as in (\ref{S1PDef}) for $\theta > 0$, and satisfies the assumptions in Definition \ref{DefAss}. Fix $L \in (0,\infty)$ and $m_j \in \mathbb{Z}_{\geq 0}$ for $j \in \llbracket  k, N \rrbracket$. For each $j \in \llbracket  k, N \rrbracket$ we fix parameters $\vm^j = (v^j_1, \dots, v^j_{m_j})$, such that $L \cdot v_i^j \in \mathcal{M} \setminus [- N\theta, M -\theta] $ for $j \in \llbracket k, N \rrbracket$ and $i \in \llbracket 1,  m_j \rrbracket$, where $\mathcal{M}$ is as in Definition \ref{DefAss}. Suppose further that $\Gamma$ is a positively oriented simple contour contained in $L^{-1} \cdot \mathcal{M}$ that encloses $[-\theta N /L, (M-\theta)/L]$. We also assume that $\Gamma$ excludes the points $\{ v_i^j, v_i^j + L^{-1} : j \in \llbracket k, N \rrbracket, \hspace{2mm} i \in \llbracket 1, m_j\rrbracket \}$. Finally, let $v \in \mathbb{C}$ be outside of $\Gamma$, 
$$S_v(Lz) = \frac{L(z-v)}{(Lz + N\theta)(Lz - M -1 + \theta)},$$
and $\Phi_N^{\pm}$ be as in Definition \ref{DefAss}. Then we have the following formula 
\begin{equation}\label{Exp1V3}
\begin{split}
&0=  \sum_{\substack{F_r \subseteq \llbracket 1, m_r \rrbracket \\  r = k, \dots, N }}    \prod_{i = k}^N {\bf 1 } \{F_i = \llbracket 1, m_i \rrbracket\} \cdot  \oint_{\Gamma}\frac{dz \Phi_N^-(Lz)}{2 \pi \i  S_v(Lz)}\cdot  M\left( \prod_{p = 1}^N\frac{Lz- \ell^N_p -\theta}{Lz - \ell^N_p}; F_k, \dots, F_N \right)  \\
& +  \oint_{\Gamma} \frac{dz \Phi_N^+(Lz)}{2 \pi \i  S_v(Lz)}   \prod_{i = k}^N \prod_{f \in F_i^c}\frac{1}{L(v_f^i-z)(v_f^i - z + L^{-1})} \cdot  M\left( \prod_{p = 1}^{k}\frac{Lz- \ell^{k}_p + \theta - 1}{Lz - \ell^{k}_p - 1}; F_k, \dots, F_N \right) \\
&  +   \sum\limits_{j=k+1}^{N}  \prod_{i = k}^{j-1} {\bf 1 } \{F_i = \llbracket 1, m_i \rrbracket\} \cdot \oint_{\Gamma} \frac{dz\Phi_N^+(Lz)}{2\pi \i S_v(Lz)}    \cdot   \prod_{i = j}^N \prod_{f \in F_i^c}\frac{1}{L(v_f^i-z)(v_f^i - z + L^{-1})}  \\
&\times    M \left( \Pi_1^{\theta, j}(Lz); F_k, \dots, F_N \right),
\end{split}
\end{equation}
where 
\begin{equation}\label{S42E2}
\Pi_1^{\theta,j}(z) = \begin{cases} \mathlarger{\frac{\theta}{1- \theta} \cdot \prod_{i = 1}^j \frac{z - \ell_i^j - \theta}{z - \ell_i^j - 1} \cdot \prod_{i = 1}^{j-1} \frac{z - \ell_i^{j-1} + \theta - 1}{z - \ell_i^{j-1} }} &\mbox{ if } \theta \neq 1, \\ \mathlarger{\sum_{i = 1}^j \frac{1}{z - \ell_i^j- 1} - \sum_{i = 1}^{j-1} \frac{1}{z - \ell_i^{j-1}}} &\mbox{ if } \theta = 1. \end{cases}
\end{equation}
In equation (\ref{Exp1V3}) the sum is over $(N-k+1)$-tuples of subsets $F_r \subseteq  \llbracket 1, m_r \rrbracket$ for $r = k, \dots, N$ and we have written $M(\xi; F_k, \dots, F_N)$ in place of the joint cumulant $M(\xi; \{ G_L^j(v_i^j): j \in \llbracket k, N \rrbracket, i \in F_j \})$ (with respect to the measure $\P$), where we recall that this notation was introduced just before Lemma \ref{LemCum1} and $G_L^j$ was defined in (\ref{S3DefG}). Also, we have $F_i^c= \llbracket 1, m_i \rrbracket \setminus F_i$. 
\end{lemma}
\begin{remark}\label{RemOnly1} The formulas in (\ref{Exp1V3}) constitute a large family of equations --  one for each complex $v$ and set of parameters $\{v_i^j\}$ outside of $\Gamma$. We also mention that in deriving (\ref{Exp1V3})  we will utilize (\ref{eq:mN1_gen}) and (\ref{eq:mN1_simple}). One could also use (\ref{eq:mN2_gen}) and (\ref{eq:mN2_simple}) to derive analogous formulas; however, (\ref{Exp1V3}) will suffice for the purposes of the present paper. 
\end{remark}
\begin{remark}\label{SvChoice} The choice of $S_v$ is suitable for our analysis in Section \ref{Section4}; however, for other types of analyses different choices of $S_v$ may be more suitable. The only requirement for $S_v$ to make (\ref{Exp1V3}) hold is that $1/S_v(Lz)$ is analytic in the region enclosed by $\Gamma$, and that (\ref{SvProof}) in the proof of Lemma \ref{CrudeCumExp} below holds. In particular, if $\Phi^-(-N\theta) = 0 = \Phi^+(M + 1 - \theta)$, then (\ref{SvProof})  automatically holds as $\R_1(Lz) = 0$ in view of (\ref{S3Phis}) and Remarks \ref{RemBoundary} and \ref{RemBoundary2}. If $\Phi^-(-N\theta) = 0 = \Phi^+(M + 1 - \theta)$, then (\ref{Exp1V3}) holds with $S_v(Lz) = L(z-v) H(z)$ for any analytic function $H(z)$ that does not vanish in the region enclosed by $\Gamma$.
\end{remark}

\begin{proof} Let $\mathbb{P}^{\tb, \vm}$ be the deformed measures from (\ref{eq:distrgen_deformed}) with parameters $\vm = (\vm^{k}, \dots, \vm^{N})$ as in the statement of the lemma, and $\tb = (\tb^{k}, \dots, \tb^{N})$, $\tb^j = (t^j_1, \dots, t^j_{m_{j}})$. Here we assume that $t_i^j$ are sufficiently small in absolute value so that $v^j_i + t^j_i - y \neq 0$ and $v_i^j - y \neq 0$ for all $j \in \llbracket k, N \rrbracket$, $i \in \llbracket 1,  m_j \rrbracket$, and all $y \in [- N \theta/ L, (M  - \theta)/ L]$. In addition, we require that $t^j_i$ are sufficiently close to zero so that $Z(\tb, \vm) \neq 0$, where we recall that $Z(\tb, \vm)$ is the complex normalization constant in (\ref{eq:distrgen_deformed}). Finally, we require that $t^j_i$ are sufficiently close to zero so that $v_i^j + t_i^j$, $v_i^j + t_i^j +L^{-1}$ lie outside of the region enclosed by $\Gamma$. In view of our assumptions in the statement of the lemma, we can find $\epsilon > 0$ sufficiently small, depending on $\theta, M, L, \Gamma,\vm  = (\vm^{k}, \dots, \vm^{N})$ and $w_N$, so that all of the conditions we listed are satisfied whenever $|t_i^j| < \epsilon$ for all $j \in \llbracket k, N \rrbracket$, $i \in \llbracket 1, m_j \rrbracket$.

Applying (\ref{eq:mN1_gen}), if $\theta \neq 1$, or (\ref{eq:mN1_simple}), if $\theta = 1$, to the deformed measures $\mathbb{P}^{\tb, \vm}$, we obtain that the following function $R_1$ is analytic in $L^{-1} \cdot \mathcal{M}$ 
\begin{equation}\label{Exp1V1}
\begin{split}
R_1(Lz)= \hspace{2mm} & \Phi_N^-(Lz)\cdot \prod_{i = k}^{N} A_i(z) \cdot  \mathbb{E}_{\mathbb{P}^{\tb, \vm}}\left[ \prod_{p = 1}^N\frac{Lz- \ell^N_p -\theta}{Lz - \ell^N_p} \right] \\
& + \Phi_N^+(Lz) \cdot  \prod_{i = k}^N B_i(z) \cdot  \mathbb{E}_{\mathbb{P}^{\tb, \vm}} \left[ \prod_{p = 1}^{k}\frac{Lz- \ell^{k}_p + \theta - 1}{Lz - \ell^{k}_p - 1} \right] +    \sum\limits_{j=k+1}^{N} \Phi_N^+(Lz)     \\
& \times \prod_{i = k}^{j-1} A_i(z)  \cdot \prod_{i = j}^{N} B_i(z)   \cdot \mathbb{E}_{\mathbb{P}^{\tb, \vm}} \left[ \Pi_1^{\theta, j}(Lz) \right] -\R_1(Lz),
\end{split}
\end{equation}
where for $j \in \llbracket k, N \rrbracket$ we have
\begin{equation}
\begin{split}
A_j(z) = \prod_{i = 1}^{m_j} \left( v_i^j  + t_i^j - z + L^{-1} \right) \left(v_i^j - z \right)  \mbox{ and } B_j(z) = \prod_{i = 1}^{m_j} \left( v_i^j  + t_i^j - z  \right)\left( v_i^j - z + L^{-1}\right).
\end{split}
\end{equation}

In deriving (\ref{Exp1V1}) we used that $\mathbb{P}^{\tb, \vm}$ is of the form (\ref{S2PDef}) with $w_N$ given by 
$$w_N(x) \cdot \prod_{i  = 1}^{m_N} \left( 1 + \frac{t_i^N}{v_i^N- x/L} \right),$$
and $w_j$ for $j \in \llbracket k,  N-1 \rrbracket $ given by
$$ \prod_{i  = 1}^{m_j} \left( 1 + \frac{t_i^j}{v_i^j- x/L} \right).$$
In particular, the conditions of Theorem \ref{MainNek1} (if $\theta \neq 1$) and Theorem \ref{MainNek2} (if $\theta = 1$) are satisfied with $\mathcal{M}$ as in the statement of the present lemma and functions 
\begin{equation}\label{S3Phis}
\begin{split}
&\phi^{N+1}_1(z) = \Phi_N^-(z) \cdot \prod_{i = k}^{N} A_i(z/L), \mbox{ and }\phi^j_1(z) = \Phi_N^+(z) \cdot  \prod_{i = k}^{j-1} A_i(z/L)  \cdot \prod_{i = j}^{N} B_i(z/L) \mbox{ for $j \in \llbracket k, N \rrbracket$}.
\end{split}
\end{equation}
In deriving (\ref{Exp1V1}) we also changed variables $z \rightarrow Lz$.\\

We next proceed to divide both sides of (\ref{Exp1V1}) by $2 \pi \i \cdot S_v(Lz) \cdot \prod_{i = k}^N A_i(z)$ and integrate over $\Gamma$. Note that the left side of (\ref{Exp1V1}) evaluates to $0$ by Cauchy's theorem, as we are integrating an analytic function over a closed contour. Here we used that by assumption $v$ lies outside of the region enclosed by $\Gamma$ and the same is true for the zeros of the function $\prod_{i = k}^N A_i(z)$ by our assumptions on $\vm, \tb$. Furthermore, we have that 
\begin{equation}\label{SvProof}
\int_{\Gamma} \frac{\R_1(Lz)}{2 \pi \i \cdot S_v(Lz) \cdot \prod_{i = k}^N A_i(z)} dz  = 0,
\end{equation}
by Cauchy's theorem, since the simple poles of $\R_1(Lz)$ at $z = -N\theta/L$ and $z = (M+1-\theta)/L$ are canceled by the zeros of $1/S_v(Lz)$ at these points, implying that the integrand above is analytic in the region enclosed by $\Gamma$. Thus we obtain
\begin{equation}\label{Exp1V2}
\begin{split}
&0=    \oint_{\Gamma}\frac{dz \Phi_N^-(Lz)}{2 \pi \i S_v(Lz)}\cdot  \mathbb{E}_{\mathbb{P}^{\tb, \vm}}\left[ \prod_{p = 1}^N\frac{Lz- \ell^N_p -\theta}{Lz - \ell^N_p} \right]+   \oint_{\Gamma} \frac{dz \Phi_N^+(Lz)}{2 \pi \i  S_v(Lz)}  \prod_{i = k}^N \frac{B_i(z)}{A_i(z)}  \\
&   \times \mathbb{E}_{\mathbb{P}^{\tb, \vm}} \left[ \prod_{p = 1}^{k}\frac{Lz- \ell^{k}_p + \theta - 1}{Lz - \ell^{k}_p - 1} \right] +  \sum\limits_{j=k+1}^{N} \oint_{\Gamma} \frac{dz\Phi_N^+(Lz)}{2 \pi \i  S_v(Lz)} \prod_{i = j}^{N} \frac{B_i(z)}{A_i(z)}   \cdot \mathbb{E}_{\mathbb{P}^{\tb, \vm}} \left[  \Pi_1^{\theta, j}(Lz)  \right].
\end{split}
\end{equation}

We next apply the operator $\mathcal{D} = \prod_{j = k}^N \prod_{i = 1}^{m_j} \partial_{t_i^j}$ to both sides of (\ref{Exp1V2}) and set $t_i^j = 0$ for $j \in \llbracket k, N \rrbracket$, $i \in \llbracket 1, m_j \rrbracket$. Notice that when we perform the differentiation to the right side some of the derivatives could land on the products $\prod_i \frac{B_i(z)}{A_i(z)} $ and some on the expectations $\mathbb{E}_{\mathbb{P}^{\tb, \vm}}$. We  split the result of the differentiation based on subsets $F_r$ for $r \in \llbracket k, N \rrbracket$. The set $F_r$ consists of the indices $f \in \llbracket 1, m_r \rrbracket$ such that $\partial_{t_f^r}$ differentiates the expectation. The result of applying $\mathcal{D}$ to (\ref{Exp1V2}) and setting $t_i^r = 0$  is then precisely (\ref{Exp1V3}),  once we use Lemma \ref{LemCum1} and
$$\frac{( v  + t - z  )( v - z + L^{-1})}{( v  + t - z + L^{-1} ) (v - z )} \bigg\rvert_{t = 0} \hspace{-2mm} = 1 \mbox{ and } \left( \partial_t \frac{( v  + t - z  )( v - z + L^{-1})}{( v  + t - z + L^{-1} ) (v - z )}  \right)  \bigg\rvert_{t = 0}  = \frac{1}{L(v- z)(v- z + L^{-1})}.$$
\end{proof}

%-------------------------------------------------------------------------------------------------------------------------------------------------------------------------------------------------
% Section 4
%
%-------------------------------------------------------------------------------------------------------------------------------------------------------------------------------------------------
\section{Continuous limits}\label{Section4} The purpose of this section is to prove Theorem \ref{ManyLevelLoopThm_intro}. In Section \ref{Section4.1} we recall the continuous $\beta$-corners processes from Section \ref{Section1.1} and derive a few of their properties. In Section \ref{Section4.2} we summarize some useful notation for Jack symmetric functions and explain how the latter relate to the measures in (\ref{S1PDef}). In Section \ref{Section4.3} we derive the continuous measures from Section \ref{Section4.1} as diffuse limits of the measures in (\ref{S1PDef}). Finally, in Section \ref{Section4.4} we derive our continuous multi-level loop equations by combining our weak convergence result, Proposition \ref{prop_cont_limit} from Section \ref{Section4.3}, and Lemma \ref{CrudeCumExp}.

%-------------------------------------------------------------------------------------------------------------------------------------------------------------------------------------------------
% Section 4.1
%
%-------------------------------------------------------------------------------------------------------------------------------------------------------------------------------------------------
\subsection{Continuous $\beta$-corners processes}\label{Section4.1} 

Let us fix $N \in \mathbb{N}$, $a_-,a_+ \in \mathbb{R}$ with $a_- < a_+$ and $\theta > 0$. In addition, we let $V: [a_-,a_+] \rightarrow \mathbb{R}$ be continuous. With this data we define the following probability density function
\begin{equation}\label{S4Corners1}
\begin{split}
&f(y) =  (Z_N^c)^{-1} \cdot {\bf 1}_{\mathcal{G}_N}  {\bf 1}\{ a_- < y^N_1, y^N_N < a_+\} \cdot  g_N(y) \mbox{, where } y = (y^N, \dots, y^1) \in \mathbb{R}^N \times \cdots \times \mathbb{R},\\
&g_N(y) = \prod_{1 \leq i<j \leq N} \hspace{-2mm} (y_j^N-y_i^N)  \hspace{-1mm}  \prod_{j = 1}^{N-1}\left(\prod_{1\leq a<b \leq j}
(y_b^j-y_a^j)^{2-2\theta} \prod_{a=1}^j\prod_{b=1}^{j+1}|y_a^j-y_b^{j+1}|^{\theta-1}\right)  \prod_{i = 1}^N e^{-N \theta V(y^N_i)},
\end{split}
\end{equation}
\begin{equation}\label{S4Corners2}
\begin{split}
&Z^c_N = \prod_{j = 1}^N \frac{\Gamma(\theta)^j}{\Gamma(j \theta)}  \cdot \int_{\mathcal{W}^{\circ}_N} {\bf 1}\{ a_- < y_1, y_N < a_+\} \cdot \prod_{1 \leq i<j \leq N} (y_j - y_i)^{2\theta} \prod_{i = 1}^N e^{-N \theta V(y_i)}dy_i,
\end{split}
\end{equation}
\begin{equation}\label{S4Corners3}
\begin{split}
&\mathcal{W}^{\circ}_N = \{ \vec{y} \in \mathbb{R}^N: y_1 < y_2 < \cdots < y_N\}, \hspace{2mm} \mathcal{G}_N = \{ y = (y^N, \dots, y^1) \in \mathbb{R}^N \times \cdots \times \mathbb{R}:  \\
&y_{i}^{j}<y_i^{j-1}<y_{i+1}^{j}, \mbox{$j = 1,\dots, N$, $i = 1, \dots, j$} \}.
\end{split}
\end{equation}
Note that (\ref{S4Corners1}) is precisely the continuous $\beta$-corners process from (\ref{S1Corners1}) when $\theta = \beta/2$.

Let us briefly explain why $f(y)$ as in (\ref{S4Corners1}) is a density function. The fact that $f(y) \geq 0$ is immediate and so it suffices to show that the integral of $f(y)$ over $\mathbb{R}^N \times \cdots \times \mathbb{R}$ is equal to one. To prove the latter we use the following version of the Dixon-Anderson identity \cite{Anderson, Dixon} (see \cite[Equation (2.2)]{FW08}), which states that for $x_1< \cdots < x_N$ we have
\begin{equation}\label{Dix}
\int_{x_1}^{x_2}dy_1 \cdots \int _{x_{N-1}}^{x_N} \hspace{-3mm} dy_{N-1}\hspace{-3mm}  \prod_{1 \leq i < j \leq N-1} \hspace{-3mm} (y_j - y_i) \prod_{i = 1}^{N-1} \prod_{j = 1}^N |y_i - x_j|^{\theta - 1}   = \frac{\Gamma(\theta)^{N}}{\Gamma(N \theta)} \prod_{1 \leq i < j \leq N}(x_j - x_i)^{2\theta - 1}.
\end{equation}
Using (\ref{Dix}) $k-1$ times, where $k \in \llbracket 1, N \rrbracket$, we see that 
\begin{equation}\label{S4CornersProj}
\begin{split}
&g_{N,k}(y^N, \dots, y^k) := \int_{\mathbb{R}^{k-1} \times \cdots \times \mathbb{R}} g_N(y) dy^{k-1} \cdots dy^1 =  \prod_{j = 1}^{k} \frac{\Gamma(\theta)^j}{\Gamma(j\theta)} \cdot  \prod_{1 \leq i<j \leq N}  (y_j^N-y_i^N)    \\
&\times \prod_{1 \leq i<j \leq k}  (y_j^k-y_i^k)^{2\theta - 1} \prod_{j = k}^{N-1}\left(\prod_{1\leq a<b \leq j}(y_b^j-y_a^j)^{2-2\theta} \prod_{a=1}^j\prod_{b=1}^{j+1}|y_a^j-y_b^{j+1}|^{\theta-1}\right)  \prod_{i = 1}^N e^{-N \theta V(y^N_i)},
\end{split}
\end{equation}
where $dy^j = dy_1^j \cdots dy_j^j$ is the Lebesgue measure on $\mathbb{R}^j$. We see from (\ref{S4CornersProj}) applied to $k = N$ that
\begin{equation*}
\begin{split}
&\int_{\mathbb{R}^N \times \cdots \times \mathbb{R}} f(y) dy^N \cdots dy^1 = (Z_N^c)^{-1} \cdot \prod_{j = 1}^N \frac{\Gamma(\theta)^j}{\Gamma(j \theta)} \\
& \times \int_{\mathcal{W}^{\circ}_N} {\bf 1}\{ a_- < y^N_1, y^N_N < a_+\} \cdot \prod_{1 \leq i<j \leq N} (y^N_j - y^N_i)^{2\theta} \prod_{i = 1}^N e^{-N \theta V(y^N_i)}dy^N_i = 1,
\end{split}
\end{equation*}
where in the last equality we used the definition of $Z^c_N$ from (\ref{S4Corners2}). This proves $f(y)$ is a density.\\

Using (\ref{S4CornersProj}) we see that if $(Y^N, \dots, Y^1)$ is a random vector taking value in $\mathbb{R}^N \times \cdots \times \mathbb{R}^{1}$, whose distribution has density $f(y)$ as in (\ref{S4Corners1}) then $(Y^N, \dots, Y^k)$  is a random vector taking value in $\mathbb{R}^N \times \cdots \times \mathbb{R}^{k}$ with density
\begin{equation}\label{S4CornersProj2}
\begin{split}
&f_{N,k}(y^N, \dots, y^k) :=  (Z_N^c)^{-1} \cdot {\bf 1}_{\mathcal{G}_{N,k}} \cdot  {\bf 1}\{ a_- < y^N_1, y^N_N < a_+\} \cdot g_{N,k}(y^N, \dots, y^k) , \mbox{ where } \\
& \mathcal{G}_{N,k} = \{ y = (y^N, \dots, y^k) \in \mathbb{R}^N \times \cdots \times \mathbb{R}^k: y_{i}^{j}<y_i^{j-1}<y_{i+1}^{j}, \mbox{$j = k,\dots, N$, $i = 1, \dots, j$} \},
\end{split}
\end{equation}
and $g_{N,k}$ is as in (\ref{S4CornersProj}). When $k = N$ we have that (\ref{S4CornersProj2}) becomes 
\begin{equation}\label{S4CornersProjTop}
\begin{split}
f_{N, N}(y^N) =  \hspace{2mm} & (Z_{N}^c)^{-1} \cdot \prod_{j = 1}^{N} \frac{\Gamma(\theta)^j}{\Gamma(j\theta)}  \cdot  {\bf 1}\{ a_- < y^N_1 < y^N_2 < \cdots < y^N_N < a_+\}  \\
& \times  \prod_{1 \leq i<j \leq N} (y_j^N-y_i^N)^{2 \theta}  \prod_{i = 1}^N e^{-N \theta V(y^N_i)},
\end{split}
\end{equation}
which is precisely the $\beta$-log gas measures studied in \cite{BoGu} once one sets $\theta = \beta/2$. The fact that the projection of $f(y)$ from (\ref{S4Corners1}) to the top level $y^N$ is given by (\ref{S4CornersProjTop}) is the reason we view (\ref{S4Corners1}) and (more generally) (\ref{S4CornersProj2}) as natural multi-level generalizations of the measures in \cite{BoGu}.

%-------------------------------------------------------------------------------------------------------------------------------------------------------------------------------------------------
% Section 4.2
%
%-------------------------------------------------------------------------------------------------------------------------------------------------------------------------------------------------
\subsection{Jack symmetric functions}\label{Section4.2} In this section we summarize some basic facts about Jack symmetric functions. Our discussion will follow \cite[Section 6]{DK2020} and \cite[Section VI.10]{Mac}, and we refer to the latter for more details.

We fix $\theta > 0$, $N \in \mathbb{N}$ and set
$$\Lambda^{\infty}_N = \{  (\lambda_1, \dots, \lambda_N) \in \mathbb{Z}^N : \lambda_1\geq  \lambda_2 \geq \cdots \geq \lambda_N \geq 0 \}.$$
Viewing an element $\lambda = (\lambda_1, \dots, \lambda_N) \in \Lambda^{\infty}_N$ as a {\em partition} or equivalently as a {\em Young diagram} with $\lambda_1$ left justified boxes in the top row, $\lambda_2$ in the second row and so on we denote the {\em Jack polynomial} with $n$ variables corresponding to $\lambda$ by $J_{\lambda}(x_1, \dots, x_n)$. When $x_1 = x_2 = \cdots = x_n = 1$ we denote $J_{\lambda}(x_1, \dots, x_n)$ by $J_\lambda(1^n)$ and from \cite[Equation (6.10)]{DK2020} we have 
\begin{equation}\label{jackpoly2}
J_\lambda(1^{N}) = \prod_{i = 1}^{N} \frac{\Gamma(\theta)}{\Gamma( i \theta)} \times \prod_{1 \leq i < j \leq N} \frac{\Gamma(\ell_i - \ell_j + \theta)}{\Gamma(\ell_i - \ell_j)},
\end{equation}
where as usual $\ell_i = \lambda_i - i \cdot \theta$.

If $\mu \in \Lambda^{\infty}_{N-1}$ we denote by $J_{\lambda/\mu}$ the {\em skew Jack polynomial} and from \cite[Equation (6.10)]{DK2020} 
\begin{equation}\label{SkewJackPoly}
\begin{split}
J_{\lambda/ \mu}(1) = {\bf 1} \{ \lambda \succeq \mu\} \cdot & \prod_{1 \leq i < j \leq N} \frac{\Gamma(\ell_i - \ell_j + 1 - \theta)}{\Gamma(\ell_i - \ell_j) } \cdot \prod_{1 \leq i < j \leq N-1} \frac{\Gamma(m_i - m_j + 1)}{\Gamma(m_i - m_j + \theta)} \\
&\times \prod_{1 \leq i < j \leq N} \frac{\Gamma(m_i - \ell_j)}{ \Gamma(m_i - \ell_j + 1 - \theta)}  \cdot \prod_{1 \leq i \leq j \leq N-1} \frac{\Gamma(\ell_i - m_j + \theta)}{\Gamma(\ell_i - m_j + 1)},  
\end{split}
\end{equation}
where $\ell_i = \lambda_i  - i  \cdot \theta$, $m_i = \mu_i - i \cdot \theta$ and $\lambda \succeq \mu$ means $\lambda_1 \geq \mu_1 \geq \lambda_2 \geq \mu_2 \cdots \geq \mu_{N-1} \geq \lambda_N$. The skew Jack polynomials satisfy the following identity, see \cite[Equation (6.12)]{DK2020}, 
\begin{equation}\label{S7Branchv2}
J_\lambda(1^N)= \sum_{\lambda^1 \preceq \lambda^2 \preceq \cdots \preceq \lambda^{N-1}\preceq \lambda} J_{\lambda / \lambda^{N-1}}(1) \cdot J_{\lambda^{N-1}/ \lambda^{N-2}}(1) \cdot J_{\lambda^{N-2}/ \lambda^{N-3}}(1) \cdots J_{\lambda^{2}/ \lambda^{1}}(1) ,
\end{equation}
which is a special case of the branching relations for Jack symmetric functions.

With the above notation we can prove the projection statement we made in Remark \ref{S1ProjR}.
\begin{lemma}\label{ProjLemma} Fix $N \in \mathbb{N}$, $M \in \mathbb{Z}_{\geq 0}$ and $\theta > 0$. Let $(\Omega, \mathcal{F},\mathbb{P})$ be a probability space, and let $(\ell^N, \dots, \ell^1)$ be a random vector defined on this space whose distribution under $\mathbb{P}$ is given by (\ref{S1PDef}) with $k = 1$. Then for any $m \in \llbracket 1, N \rrbracket$ distribution of $(\ell^N, \dots, \ell^m)$ is given by (\ref{S1PDef}) with $k = m$. 
\end{lemma}
\begin{proof} Writing as usual $ \ell_i^j = \lambda_i^j - i \cdot \theta$ we have
\begin{equation*}
\begin{split}
&\mathbb{P}(\ell^N, \dots, \ell^m) =  Z_N^{-1} \sum_{\substack{\lambda^1, \dots, \lambda^{m-1}: \\
                                                   \lambda^1 \preceq \cdots \preceq \lambda^{m-1} \preceq \lambda^{m}} }  H^t(\ell^N) \cdot \prod_{j = 1}^{N-1}  \I(\ell^{j+1}, \ell^j)  \\ 
& = Z_N^{-1} \sum_{\substack{\lambda^1, \dots, \lambda^{m-1}: \\
                                                   \lambda^1 \preceq \cdots \preceq \lambda^{m-1} \preceq \lambda^{m}} }   H^t(\ell^N) \cdot \prod_{j = m}^{N-1}  \I(\ell^{j+1}, \ell^j) \prod_{i = 1}^{m-1}J_{\lambda^i/ \lambda^{i-1}}(1)  \\
& = Z_N^{-1}H^t(\ell^N) \cdot \prod_{j = m}^{N-1}  \I(\ell^{j+1}, \ell^j) J_{\lambda^m}(1^{m}) = Z_N^{-1} \cdot \prod_{i = 1}^m \frac{\Gamma(\theta)}{\Gamma(i\theta)} \cdot H^t(\ell^N) \cdot \prod_{j = m}^{N-1}  \I(\ell^{j+1}, \ell^j)  \cdot H^b(\ell^m),
\end{split}
\end{equation*}
where in the second equality we used (\ref{SkewJackPoly}), in the third equality we used (\ref{S7Branchv2}) and in the last equality we used (\ref{jackpoly2}). The above equation proves the statement of the lemma. 
\end{proof}

We denote the {\em conjugate} of a Young diagram $\lambda$ (i.e. the Young diagram obtained by transposing the diagram $\lambda$) by $\lambda'$ and then we have $\lambda'_i=|\{j\in \mathbb{Z}_{>0}\mid \lambda_j\ge i\}|$. For a box $\square=(i,j)$ of a Young diagram $\lambda$ we let $a(\square), \ell(\square)$ denote the {\em arm} and {\em leg lengths} respectively, i.e. 
$$a(\square)=\lambda_i-j, \quad \ell(\square)=\lambda'_j-i.$$
Below we make use of the {\em dual Jack polynomials} $\tilde{J}_{\lambda}$, defined as
\begin{equation}\label{dualJack}
\tilde{J}_{\lambda}={J}_{\lambda}\cdot \prod_{\square\in\lambda}\frac{a(\square)+\theta\ell(\square)+\theta}{a(\square)+\theta \ell(\square)+1}.
\end{equation}
Using (\ref{jackpoly2}) and (\ref{dualJack}) we obtain
\begin{equation}\label{Djackpoly2}
\begin{split}
&\tilde{J}_\lambda(1^{N}) = \prod_{i = 1}^{N} \frac{\Gamma(\theta)}{\Gamma( i \theta)} \times \prod_{1 \leq i < j \leq N} \frac{\Gamma(\ell_i - \ell_j + \theta)}{\Gamma(\ell_i - \ell_j)} \cdot  \prod_{i = 1}^{N} \prod_{j =1}^{\lambda_i} \frac{\lambda_i - j + \theta(\lambda_j' - i) + \theta}{\lambda_i - j + \theta(\lambda_j' - i) + 1 } \\
& = \prod_{i = 1}^{N} \frac{1}{\Gamma( i \theta)} \times \prod_{1 \leq i < j \leq N} \frac{\Gamma(\ell_i - \ell_j + 1)}{\Gamma(\ell_i - \ell_j + 1 - \theta)} \prod_{i = 1}^{N} \frac{\Gamma(\ell_i + \theta N + \theta )}{\Gamma(\ell_i + \theta N + 1 )}.
\end{split}
\end{equation}
Finally, we record the following formula, which is a special case of the {\em Cauchy identity} for Jack symmetric functions (see \cite[Section VI.10]{Mac}),
\begin{equation}\label{CauchyId}
\begin{split}
&\sum_{\lambda \in \Lambda^{\infty}_N} \prod_{i = 1}^Nq^{ \lambda_i} \cdot J_{\lambda}(1^N) \tilde{J}_\lambda(1^{N}) = (1 - q)^{-\theta N^2},
\end{split}
\end{equation}
provided that $q \in [0, 1)$.

%-------------------------------------------------------------------------------------------------------------------------------------------------------------------------------------------------
% Section 4.3
%
%-------------------------------------------------------------------------------------------------------------------------------------------------------------------------------------------------
\subsection{Diffuse limits}\label{Section4.3} 
In this section we derive the measures in (\ref{S4CornersProj2}) as diffuse limits of the measures from (\ref{S1PDef}). We summarize some notation in the following definition.

\begin{definition}\label{S4LDef}
Fix $\theta > 0$, $N \in \mathbb{N}$, $k \in \llbracket 1, N \rrbracket$ and $a_{\pm} \in \mathbb{R}$, $V$ as in Section \ref{Section4.1}. With this data we let for $L \in \mathbb{N}$, $\mathbb{P}_L$ be as (\ref{S1PDef}) with $\theta, k, N$ as just introduced, $M = M_L := \lfloor (a_+ - a_-) L \rfloor$ and $w_N(x) = w_{N,L}(x):= \exp \left( - N \theta V(a_- + xL^{-1}) \right).$
\end{definition}

We turn to the main result of the section.
\begin{proposition} \label{prop_cont_limit} Let $\left( \ell^{L,N}, \dots, \ell^{L,k} \right)$ be a sequence of random vectors, whose probability distribution is $\mathbb{P}_L$ as in Definition \ref{S4LDef}. Then the sequence of random vectors $(X^{L,N}, \dots, X^{L,k}) \in \mathbb{R}^N \times \cdots \times \mathbb{R}^k$, defined by 
$$X_i^{L,j} = a_- + L^{-1} \cdot \ell_{j-i+1}^{L,j} \mbox{ for $i \in \llbracket 1, j \rrbracket$ and $j \in \llbracket k , N \rrbracket$ }$$
 converges weakly to the random vector $(Y^N, \dots, Y^k)$, whose density is given by (\ref{S4CornersProj2}).
\end{proposition}
\begin{remark} We mention that Proposition \ref{prop_cont_limit} can be proved using the arguments in \cite[Section 2]{GS} for $\theta$-Gibbs measures. The proof we present below is more direct and contains more details.
\end{remark}
\begin{proof} The proof we present here is similar to the one for \cite[Proposition 7.3]{DK2020}. To ease the notation we suppress the dependence on $L$ of various quantities, writing for example $\left( \ell^{N}, \dots, \ell^{k} \right)$ in place of $\left( \ell^{L,N}, \dots, \ell^{L,k} \right)$. Furthermore, by performing a simple horizontal shift, we see that it suffices to prove the proposition when $a_- = 0$, which we assume in the sequel. Finally, since the measure in (\ref{S1PDef}) for $k = m$ with $m \in \llbracket 1, N \rrbracket$ is precisely the projection of the measure in (\ref{S1PDef}) for $k = 1$ to $(\ell^N, \dots, \ell^m)$ (this is Lemma \ref{ProjLemma}), and the measure in (\ref{S4CornersProj2}) is precisely the projection of (\ref{S4Corners1}) to $(y^N, \dots, y^k)$, we see that it suffices to prove the proposition when $k = 1$.

 For clarity we split the proof into four steps.\\

{\bf \raggedleft Step 1.} We begin by introducing a bit of useful notation. Throughout the proof we write as usual $ \ell_i^j  = \lambda_i^j  - i \cdot \theta$. For $(\ell^N, \dots, \ell^1) \in \XXX$ we define
\begin{equation}\label{XOP1}
\begin{split}
&H(\ell^N, \dots, \ell^1) = H^t(\ell^N) \cdot \prod_{j = 1}^{N-1}  \I(\ell^{j+1}, \ell^j).
\end{split}
\end{equation}
We define the function $f_L: \mathbb{R}^N \times \cdots \times \mathbb{R} \rightarrow [0, \infty)$ by
\begin{equation}\label{XOP2}
\begin{split}
&f_L(y^N,\dots, y^1) =  L^{N(N-1)/2  - \theta \cdot N(N-1)} \cdot H(\ell^N, \dots, \ell^1) \cdot \prod_{i = 1}^N e^{-N \theta V(\ell^N_i/L)} \\
& \mbox{ if $\lambda^j_i \leq Ly^{j}_{j- i +1} < \lambda^j_i + 1$ for $j = 1 , \dots, N$ and $i = 1, \dots, j$  } \\
&\mbox{ with $\lambda^i \in \Lambda_i^M$ for $i = 1, \dots, N$ and $ \lambda^1 \preceq \cdots \preceq \lambda^N$,  and $f_L(x,y) = 0 $ otherwise}.
\end{split}
\end{equation}
If $y = (y^N, \dots, y^1) \in \mathbb{R}^N \times \cdots \times \mathbb{R}$, we let $ \phi_L(y) $ denote the vector $(z^N, \dots, z^1) \in \mathbb{R}^N \times \cdots \times \mathbb{R}$ such that $z_i^j = L^{-1} \cdot \lfloor y_{i}^j \cdot L \rfloor$. We also write $dy^j$ for the Lebesgue measure on $\mathbb{R}^j$ and $dy$ for the Lebesgue measure on $\mathbb{R}^N \times \cdots \times \mathbb{R}$.

We claim that if $h: \mathbb{R}^N \times \cdots \times \mathbb{R} \rightarrow \mathbb{R}$ is any bounded continuous function then
\begin{equation}\label{XOP3}
\begin{split}
&\lim_{L \rightarrow \infty} \int_{\mathbb{R}^N \times \cdots \times \mathbb{R}} f_L(y) h \left( \phi_L(y) \right) dy = \int_{\mathbb{R}^N \times \cdots \times \mathbb{R}} g_N(y) h \left(y \right) \cdot {\bf 1}_{\mathcal{G}_N}  {\bf 1}\{ a_- < y^N_1, y^N_N < a_+\} dy,
\end{split}
\end{equation}
where $g_N$ is as in (\ref{S4Corners1}). We will prove (\ref{XOP3}) in the steps below. Here we assume its validity and conclude the proof of the proposition.\\

We observe by definition that if $h_1(x) \equiv 1$ we have for each $L \in \mathbb{N}$ and bounded continuous $h: \mathbb{R}^N \times \cdots \times \mathbb{R} \rightarrow \mathbb{R}$ that
\begin{equation*}
\begin{split}
&\mathbb{E} \left[ h(L^{-1} \lambda_{j - i + 1}^j: j = 1,\dots, N, i = 1, \dots, j) \right] = \frac{ \int_{\mathbb{R}^N \times \cdots \times \mathbb{R}} f_L(y) h \left( \phi_L(y) \right) dy}{ \int_{\mathbb{R}^N \times \cdots \times \mathbb{R}} f_L(y) h_1 \left( \phi_L(y) \right) dy}.
\end{split}
\end{equation*}
Taking $L \rightarrow \infty$ in the last identity and using (\ref{XOP3}), and $h_1(x) \equiv 1$ we get 
\begin{equation*}
\begin{split}
&\lim_{L \rightarrow \infty} \mathbb{E} \left[ h(L^{-1} \lambda_{j - i + 1}^j: j = 1,\dots, N, i = 1, \dots, j) \right]  \\
&= \frac{ \int_{\mathbb{R}^N \times \cdots \times \mathbb{R}} g_N(y) h \left(y \right) {\bf 1}_{\mathcal{G}_N}  {\bf 1}\{ a_- < y^N_1, y^N_N < a_+\} dy}{ \int_{\mathbb{R}^N \times \cdots \times \mathbb{R}} g_N(y) {\bf 1}_{\mathcal{G}_N}  {\bf 1}\{ a_- < y^N_1, y^N_N < a_+\} dy} =  \int_{\mathbb{R}^N \times \cdots \times \mathbb{R}} f(y) h \left(y \right) dy,
\end{split}
\end{equation*}
where in the last equality we used (\ref{S4Corners1}). The last equation implies that $(L^{-1} \lambda_{j - i + 1}^j: j = 1,\dots, N, i = 1, \dots, j)$ converge weakly to $(Y_{i}^j: j = 1,\dots, N, i = 1, \dots, j)$, which implies the statement of the proposition.\\

{\bf \raggedleft Step 2.} We claim that there are functions $h_L: \mathbb{R}^N \times \cdots \times \mathbb{R} \rightarrow [0,\infty)$ for $L \in \mathbb{N} \cup \{\infty \}$ such that
\begin{enumerate}
\item $f_L(y) \leq h_L(y)$ for all $L \in \mathbb{N}$ and $y \in \mathbb{R}^N \times \cdots \times \mathbb{R}$;
\item $\lim_{L \rightarrow \infty} h_L(y) = h_{\infty}(y)$ for Lebesgue a.e. $y \in \mathbb{R}^N \times \cdots \times \mathbb{R}$;
\item $\lim_{L \rightarrow \infty} \int_{\mathbb{R}^N \times \cdots \times \mathbb{R}} h_L(y)dy = \int_{\mathbb{R}^N \times \cdots \times \mathbb{R}} h_\infty(y)dy \in [0, \infty)$.
\end{enumerate}
We prove the above claim in the steps below. Here we assume its validity and prove (\ref{XOP3}).\\

From \cite[Equation (7.14)]{DK2020} we have for any $x , \theta > 0$ that
\begin{equation}\label{Sandwich}
\frac{\Gamma(x+ \theta)}{\Gamma(x)} = x^{\theta} \cdot \exp(O(x^{-1})),
\end{equation}
where the constant in the big $O$ notation can be taken to be $\max( \theta, \theta^2)$.  Using $\Gamma(z+1) = z \Gamma(z)$ and (\ref{Sandwich}) we conclude that for $j \in \llbracket 1, N  -1\rrbracket$ and $1 \leq p < q \leq j+1$
\begin{equation}\label{mixed1}
\begin{split}
\frac{\Gamma(\ell^{j}_p -  \ell^{j+1}_q)}{ \Gamma(\ell^{j}_p - \ell^{j+1}_q + 1 - \theta)}  &=    \frac{1}{\ell^{j}_p -  \ell^{j+1}_q} \cdot  \frac{\Gamma(\ell^{j}_p -  \ell^{j+1}_q + 1)}{ \Gamma(\ell^{j}_p - \ell^{j+1}_q + 1 - \theta)}   \\
& = \frac{(\ell^{j}_p -  \ell^{j+1}_q + 1 - \theta)^{\theta}}{\ell^{j}_p -  \ell^{j+1}_q}  \exp (O(|\ell^{j}_p -  \ell^{j+1}_q + 1 - \theta|^{-1} )) .
\end{split}
\end{equation}
Analogous considerations show that if $j \in \llbracket 1, N-1\rrbracket$ and $1 \leq p \leq q\leq j$ we have
\begin{equation}\label{mixed2}
\begin{split}
\frac{\Gamma(\ell^{j+1}_p - \ell^{j}_q + \theta)}{\Gamma(\ell^{j+1}_p - \ell^{j}_q + 1)} & =   \frac{1}{\ell^{j+1}_p - \ell^{j}_q + \theta} \cdot    \frac{\Gamma(\ell^{j+1}_p - \ell^{j}_q + \theta + 1)}{\Gamma(\ell^{j+1}_p - \ell^{j}_q + 1)}  \\
& = \frac{(\ell^{j+1}_p - \ell^{j}_q+ 1)^{\theta }}{\ell^{j+1}_p - \ell^{j}_q+ \theta} \cdot \exp (O(|\ell^{j+1}_p - \ell^{j}_q  + 1|^{-1} )),
\end{split}
\end{equation}
and also for $j \in \llbracket 2, N - 1\rrbracket$ and $1 \leq p < q \leq j$ we have
\begin{equation}\label{mixed3}
\begin{split}
&\frac{\Gamma(\ell^{j}_p - \ell^{j}_q + 1 - \theta)}{\Gamma(\ell^{j}_p - \ell^{j}_q + \theta)} =  \frac{\Gamma(\ell^{j}_p - \ell^{j}_q + 1 - \theta)}{\Gamma(\ell^{j}_p -  \ell^{j}_q + 1)} \cdot \frac{\Gamma(\ell^{j}_p -  \ell^{j}_q + 1)}{\Gamma(\ell^{j}_p - \ell^{j}_q + \theta)}    \\
& = (\ell^{j}_p -  \ell^{j}_q + 1 - \theta)^{-\theta} \cdot (\ell_p^j - \ell_q^j)^{1- \theta} \cdot \exp ( O(|\ell^{j}_p -  \ell^{j}_q + 1 - \theta|^{-1} ) + O( |\ell^{j}_p -  \ell^{j}_q|^{-1} )).
\end{split}
\end{equation}

In view of (\ref{mixed1}), (\ref{mixed2}), (\ref{mixed3}) and the continuity of $V$ and $h$ we have for every $y \in E_N:= \mathcal{G}_N \cap \{ a_- < y^N_1, y^N_N < a_+\}$ that 
$$\lim_{L \rightarrow \infty} f_L(y) h \left( \phi_L(y) \right) = g_N(y) h(y).$$
Also from the definition of $f_L(y)$ we have for every $y \not \in \bar{E}_N$ (the closure of $E_N$) 
$$\lim_{L \rightarrow \infty} f_L(y) h \left( \phi_L(y) \right) = 0.$$
Since the Lebesgue measure of $\partial E_N$ is zero we conclude that for Lebesgue a.e. $y \in \mathbb{R}^N \times \cdots \times \mathbb{R}$ 
\begin{equation*}
\lim_{L \rightarrow \infty} f_L(y) h \left( \phi_L(y) \right) = g_N(y) h \left(y \right) \cdot {\bf 1}_{\mathcal{G}_N}  {\bf 1}\{ a_- < y^N_1, y^N_N < a_+\} dy.
\end{equation*}
The last statement and the three conditions in the beginning of the step imply (\ref{XOP3}) by the Generalized dominated convergence theorem (see \cite[Theorem 4.17]{Royden}) with dominating functions $h_L(y) \cdot \| h \|_{\infty}$. \\

{\bf \raggedleft Step 3.}  In this and the next step we construct functions $h_L$ that satisfy the three conditions in the beginning of Step 2. Our choice for $h_L$ depends on whether $\theta \in (0,1]$ or $\theta \in [1, \infty)$ and in this step we focus on the case when $\theta \in [1, \infty)$. When $\theta \in [1, \infty)$ we will show that $h_L$ can be taken to be the same constant multiple of an indicator function of a compact set.

We observe that by rearranging the terms in the product on the right side of (\ref{XOP1}) and using $\Gamma(z+1) =  z\Gamma(z)$ we have
\begin{equation}\label{XOP5}
\begin{split}
&H(\ell^N, \dots, \ell^1) = I_1 \cdot I_2 \cdot I_3, \mbox{ where }I_1 = \hspace{-2mm} \prod_{1 \leq p < q \leq N} \frac{(\ell_p^N - \ell_q^N) \Gamma(\ell_p^{N-1} - \ell_q^N) \Gamma(\ell_p^{N} - \ell_{q-1}^{N-1} + \theta)}{\Gamma(\ell_p^{N-1} - \ell_q^N + 1 - \theta) \Gamma(\ell_p^{N} - \ell_{q-1}^{N-1} + 1)} ,\\
&I_2 = \prod_{j = 2}^{N-1} \left[ \prod_{1 \leq p < q \leq j} \frac{\Gamma(\ell_p^j - \ell_q^j + 1- \theta)\Gamma(\ell_p^{j-1} - \ell_q^j )}{\Gamma(\ell_p^j - \ell_q^j )\Gamma(\ell_p^{j-1} - \ell_q^j + 1- \theta)} \right], \\
&I_3 = \prod_{j = 2}^{N-1} \left[ \prod_{1 \leq p < q \leq j} \frac{\Gamma(\ell_p^j - \ell_q^j + 1)\Gamma(\ell_p^{j} - \ell_{q-1}^{j-1} + \theta )}{\Gamma(\ell_p^j - \ell_q^j + \theta)\Gamma(\ell_p^{j} - \ell_{q-1}^{j-1} + 1)} \right].
\end{split}
\end{equation}

From (\ref{mixed1}), (\ref{mixed2}) and the fact that $\theta \geq 1$ we have for some $C_1 > 0$, depending on $\theta$ and $N$,
\begin{equation}\label{XOP6}
I_1 \leq C_1 \prod_{1 \leq p < q \leq N}(\ell_p^N - \ell_q^N) (\ell_p^{N-1} - \ell_q^N)^{\theta -1} (\ell_p^{N} - \ell_{q-1}^{N-1} + \theta)^{\theta - 1}.
\end{equation}
We next prove that 
\begin{equation}\label{XOP7}
I_2 \leq 1 \mbox{ and } I_3 \leq 1.
\end{equation}
To prove (\ref{XOP7}) we show that each term in the products defining $I_2$, $I_3$ in (\ref{XOP5}) is at most $1$.\\

We will use below that $\log \Gamma(z)$ is a convex function on $(0, \infty)$, since $\frac{d^2}{dz^2} \log \Gamma(z) = \sum_{n = 0}^{\infty} \frac{1}{(n+z)^2}$, and so for any $y \geq x > 0$ and $z \geq 0$
\begin{equation}\label{convexFun}
\log\Gamma (x + z) - \log\Gamma(x) \leq \log\Gamma (y+z) -\log\Gamma (y),
\end{equation}
which can be deduced from \cite[Exercise 4.23]{Rudin}. 

Let us fix $j \in \llbracket 2, N-2 \rrbracket$ and $1 \leq p < q \leq j$. We also write $a = \ell_p^j - \ell_p^{j-1}$  and $\Delta = \ell_p^{j-1} - \ell_q^j$  (note $a \geq 0$ and $\Delta \in[\theta, \infty)$ by interlacing). We then have
$$\frac{\Gamma(\ell_p^j - \ell_q^j + 1- \theta)\Gamma(\ell_p^{j-1} - \ell_q^j )}{\Gamma(\ell_p^j - \ell_q^j )\Gamma(\ell_p^{j-1} - \ell_q^j + 1- \theta)} =  \frac{\Gamma(\Delta + a + 1- \theta)\Gamma(\Delta )}{\Gamma(\Delta + a )\Gamma(\Delta + 1- \theta)} \leq 1,$$
where the last inequality used (\ref{convexFun}) with $x = \Delta + 1 - \theta$, $y = \Delta + 1 - \theta + a$ and $z = \theta - 1$. The last inequality being true for any $j \in \llbracket 2, N-2 \rrbracket$ and $1 \leq p < q \leq j$ gives the first inequality in (\ref{XOP7}).

Similarly, we write $a = \ell_{q-1}^{j-1} - \ell_q^j - \theta$, $\Delta = \ell_p^j - \ell_q^j$  (note $\Delta - \theta \geq a \geq 0$ and $\Delta \in[\theta, \infty)$ by interlacing) and then we have
$$ \frac{\Gamma(\ell_p^j - \ell_q^j + 1)\Gamma(\ell_p^{j} - \ell_{q-1}^{j-1} + \theta )}{\Gamma(\ell_p^j - \ell_q^j + \theta)\Gamma(\ell_p^{j} - \ell_{q-1}^{j-1} + 1)} \leq  \frac{\Gamma(\Delta+ 1)\Gamma( \Delta - a )}{\Gamma(\Delta+ \theta)\Gamma(\Delta - a- \theta + 1)} \leq 1$$
where the last inequality used (\ref{convexFun}) with $x = \Delta - a - \theta + 1$, $y = \Delta + 1$ and $z = \theta - 1$. The last inequality being true for any $j \in \llbracket 2, N-2 \rrbracket$ and $1 \leq p < q \leq j$ gives the second inequality in (\ref{XOP7}).\\

Combining (\ref{XOP6}) and (\ref{XOP7}) with the fact that $\ell_i^j + i \theta \in [0,  La_+]$ we conclude that there is a constant $C_2$, depending on $\theta, N$ and $a_+$, such that for $(\ell^N, \dots, \ell^1) \in \XXX$ we have
$$L^{N(N-1)/2  - \theta \cdot N(N-1)} \cdot H(\ell^N, \dots, \ell^1) \leq C_2.$$
The fact that $V$ is continuous implies that we can find $C_3 > 0$, depending on $\theta, N$ and $V$, such that for $(\ell^N, \dots, \ell^1) \in \XXX$ we have
\begin{equation}\label{S4C3}
\prod_{i = 1}^N e^{-N \theta V(\ell^N_i/L)} \leq C_3.
\end{equation}
The last two statements imply that
$$h_L(y) = C_2 \cdot C_3 \cdot {\bf 1}\{ y_i^j \in [-1, a_+ + 1]: j = 1, \dots, N, \hspace{2mm} i = 1, \dots, j\}$$
satisfy the three conditions in Step 2. We mention that since the functions $h_L$ are compactly supported and constant, the Generalized dominated convergence theorem in Step 2 is in fact nothing but the bounded convergence theorem. This will change in the $\theta \in (0,1]$ case we consider next.\\

{\bf \raggedleft Step 4.} In this last step we construct functions $h_L$ that satisfy the three conditions as in the beginning of Step 2 when $\theta \in (0, 1]$.

For $\lambda^i \in \Lambda^{\infty}_i$ (recall this was defined in Section \ref{Section4.2}) such that $\lambda^N \succeq \cdots \succeq \lambda^1$ we define
\begin{equation}
\begin{split}
&F_L(\lambda^N, \dots, \lambda^1) = \prod_{i = 2}^N J_{\lambda^i / \lambda^{i-1}}(1) \cdot \tilde{J}_{\lambda^N} (1^N) \cdot \prod_{i = 1}^N e^{-\lambda_i^N/L}.
\end{split}
\end{equation}
It follows from (\ref{SkewJackPoly}) and (\ref{Djackpoly2}) that for $(\ell^N, \dots, \ell^1) \in \XXX$ we have
\begin{equation*}
\begin{split}
&\frac{H(\ell^N, \dots, \ell^1) }{F_L(\lambda^N, \dots, \lambda^1)} = \prod_{i = 1}^N \frac{\Gamma(i\theta) e^{\lambda_i^N/L} \Gamma(\ell_i^N + \theta N + 1)}{\Gamma(\ell_i^N + \theta N + \theta)} = \prod_{i = 1}^N  \frac{\Gamma(i\theta) e^{\lambda_i^N/L}(\ell_i^N + \theta N + \theta)  \Gamma(\ell_i^N + \theta N + 1) }{\Gamma(\ell_i^N + \theta N + \theta + 1)}\\
&  = \prod_{i = 1}^N  \Gamma(i\theta) e^{\lambda_i^N/L}(\ell_i^N + \theta N + \theta) (\ell_i^N + \theta N + 1)^{-\theta} \exp \left( O(| \ell_i^N + \theta N + 1|^{-1} \right),
\end{split}
\end{equation*}
where in the last line we used (\ref{Sandwich}). In particular, we see that there is a constant $A_1 > 0$ sufficiently large, depending on $N, \theta, a_+$, such that for all $L \in \mathbb{N}$ and $(\ell^N, \dots, \ell^1) \in \XXX$ we have
 \begin{equation}\label{XOP8}
\begin{split}
L^{N(N-1)/2  - \theta \cdot N(N-1)} \cdot  H(\ell^N, \dots, \ell^1) \leq \hspace{2mm} &A_1 \cdot L^{N(N+1)/2  - \theta \cdot N^2} \cdot  F_L(\lambda^N, \dots, \lambda^1).
\end{split}
\end{equation}
We mention that in the last inequality we used $\theta \in (0, 1]$ and $\ell_i^N + i \theta \in [0, La_+ ]$. We also note that by the continuity of $V$ there exists $C_3 > 0$, depending on $V, \theta$ and $N$, such that (\ref{S4C3}) holds.

We define the function $h_L: \mathbb{R}^N \times \cdots \times \mathbb{R} \rightarrow [0, \infty)$ by
\begin{equation}\label{XOP9}
\begin{split}
&h_L(y^N,\dots, y^1) =  L^{N(N-1)/2  - \theta \cdot N(N-1)} \cdot A_1 \cdot C_3 \cdot L^{N(N+1)/2  - \theta \cdot N^2} \cdot  F_L(\lambda^N, \dots, \lambda^1), \\
& \mbox{ if $\lambda^j_i \leq Ly^{j}_{j- i +1} < \lambda^j_i + 1$ for $j = 1 , \dots, N$ and $i = 1, \dots, j$  } \\
&\mbox{ with $\lambda^i \in \Lambda_i^\infty$ for $i = 1, \dots, N$ and $ \lambda^1 \preceq \cdots \preceq \lambda^N$,  and $h_L(x,y) = 0 $ otherwise}.
\end{split}
\end{equation}
This specifies our choice of $h_L$ and we proceed to show that it satisfies the three conditions in Step 2.\\

Condition (1) is trivially satisfied in view of equations (\ref{S4C3}) and (\ref{XOP8}). We next note from (\ref{Djackpoly2}), (\ref{Sandwich}) and (\ref{mixed3}) that if $y \in \mathcal{G}$ is such that $y_1^N > 0$ and $\lambda_i = \lfloor L y^N_{N-i+1} \rfloor $ for $i = 1, \dots, N$ 
\begin{equation*}
\begin{split}
&\lim_{L \rightarrow \infty} L^{-\theta N(N+1)/2 +  N} \tilde{J}_\lambda(1^{N}) = \prod_{i = 1}^{N} \frac{1}{\Gamma( i \theta)} \prod_{1 \leq i < j \leq N} (y_j^N - y_i^N)^\theta \prod_{i = 1}^{N} (y_i^N)^{\theta - 1}.
\end{split}
\end{equation*}

Combining the latter with (\ref{SkewJackPoly}), (\ref{mixed1}), (\ref{mixed2}) and (\ref{mixed3}) we conclude that if $y \in E_N:= \{ y \in \mathcal{G}: y_1^N > 0 \}$ we have
\begin{equation*}
\begin{split}
&\lim_{ L \rightarrow \infty} h_L(y) =  A_1 \cdot C_3  \cdot \prod_{i = 1}^{N} \frac{1}{\Gamma( i \theta)} \cdot g_N(y) \prod_{i = 1}^{N} (y_i^N)^{\theta - 1} e^{-y_i^N}.
\end{split}
\end{equation*}
On the other hand, for $y \not \in \bar{E}_N$ (the closure of $E_N$) we have
\begin{equation*}
\begin{split}
&\lim_{ L \rightarrow \infty} h_L(y) = 0.
\end{split}
\end{equation*}
Since the Lebesgue measure of $\partial E_N$ is zero we conclude that for Lebesgue a.e. $y \in \mathbb{R}^N \times \cdots \times \mathbb{R}$
\begin{equation}\label{XOP10}
\begin{split}
&\lim_{ L \rightarrow \infty} h_L(y) = {\bf 1}_{E_N} \cdot A_1 \cdot C_3  \cdot \prod_{i = 1}^{N} \frac{1}{\Gamma( i \theta)} \cdot g_N(y) \prod_{i = 1}^{N} (y_i^N)^{\theta - 1} e^{-y_i^N}.
\end{split}
\end{equation}
We denote the right side of (\ref{XOP10}) by $h_{\infty}(y)$ and then (\ref{XOP10}) proves condition (2) in Step 2.\\

We are left with proving condition (3). From (\ref{S7Branchv2}) and (\ref{CauchyId}) we have for each $L \in \mathbb{N}$ that 
$$\int_{\mathbb{R}^N \times \cdots \times \mathbb{R}} h_L(y)dy = A_1 \cdot C_3 \cdot L^{- \theta N^2}   (1 - e^{-1/L})^{-\theta N^2},$$
which implies 
\begin{equation*}
\begin{split}
&\lim_{ L \rightarrow \infty} \int_{\mathbb{R}^N \times \cdots \times \mathbb{R}} h_L(y)dy = A_1 \cdot C_3.
\end{split}
\end{equation*}
On the other hand, we have from (\ref{S4CornersProj}) that 
\begin{equation*}
\begin{split}
&\int_{\mathbb{R}^N \times \cdots \times \mathbb{R}} \hspace{-4mm} h_{\infty}(y) dy = \frac{A_1 C_3 \Gamma(\theta)^N}{ \prod_{i = 1}^N\Gamma(i\theta)^2 } \cdot \int_{\mathbb{R}^N} {\bf 1 }\{ 0 < y_1 < \cdots < y_N \}\prod_{1 \leq i < j \leq N} (y_j - y_i)^{2\theta}\prod_{i = 1}^Ny_i^{\theta - 1} e^{-y_i}  \\
& = \frac{A_1 C_3 \Gamma(\theta)^N}{ N! \cdot \prod_{i = 1}^N\Gamma(i\theta)^2 }  \int_{[0, \infty)^N} \prod_{1 \leq i < j \leq N} |y_j - y_i|^{2\theta}\prod_{i = 1}^Ny_i^{\theta - 1} e^{-y_i} = A_1 \cdot C_3,
\end{split}
\end{equation*}
where in the last equality we used \cite[Equation (2.5.10)]{agz} with $a= c = \theta$. The last two equations imply the third condition in Step 2, which concludes the proof of the proposition.
\end{proof}

%-------------------------------------------------------------------------------------------------------------------------------------------------------------------------------------------------
% Section 4.4
%
%-------------------------------------------------------------------------------------------------------------------------------------------------------------------------------------------------
\subsection{Proof of Theorem \ref{ManyLevelLoopThm_intro}}\label{Section4.4} We continue with the same notation as in the statement of the theorem and Section \ref{Section4.3} above. For clarity we split the proof into five steps.\\

{\bf \raggedleft Step 1.} In this step we utilize Lemma \ref{CrudeCumExp} to obtain a sequence of equations, indexed by $L$, whose $L \rightarrow \infty$ limit will give the statement of the theorem. 

Let $\left( \ell^{L,N}, \dots, \ell^{L,k} \right)$ be a sequence of random vectors, whose probability distribution is $\mathbb{P}_L$ as in Definition \ref{S4LDef}. As in Section \ref{Section4.3} above we will drop the dependence on $L$ from the notation and simply write $\left( \ell^{N}, \dots, \ell^{k} \right)$. We assume that $K$ is a compact set inside $U$ (recall $U$ was given in the statement of the theorem), which lies outside of the contour $\Gamma$, $v \in K$ and $L$ is sufficiently large so that $v_i^j, v_{i}^j + L^{-1} \in K$ for all $j \in \llbracket k, N \rrbracket$ and $i \in \llbracket 1 , m_j \rrbracket$. We then note that $\mathbb{P}_L$ satisfies the conditions of Lemma \ref{CrudeCumExp} with $\mathcal{M} = L\cdot (U - a_-)$ (the open set obtained by shifting $U$ by $-a_-$ and rescaling by $L$), $\Phi_{L,N}^-(x) = \exp \left( N \theta [  V(a_- + xL^{-1}) - V(a_- + xL^{-1} - L^{-1})] \right) \mbox{ and } \Phi_{L,N}^+(x)  =1.$

It follows from (\ref{Exp1V3}) that when $\theta \neq 1$ we have
\begin{equation}\label{ZOP1}
\begin{split}
&0= I_L^1 + I_L^2  + I_L^3 , \mbox{ where }
\end{split}
\end{equation}
\begin{equation}\label{ZOP2}
\begin{split}
&I_L^1 = \oint_{\Gamma}\frac{dz(Lz + N\theta - La_-)(Lz - M_L -1 + \theta - La_-) e^{N\theta[ V(z) - V(z - L^{-1})]}  }{2 \pi \i  \theta (z - v)}  \\
&  \times M\left( \prod_{p = 1}^N\frac{Lz- \tilde{\ell}^N_p -\theta }{Lz - \tilde{\ell}^N_p}; \llbracket 1, m_k \rrbracket, \dots, \llbracket 1, m_N \rrbracket \right),
\end{split}
\end{equation}
\begin{equation}\label{ZOP3}
\begin{split}
&I_L^2 =  \sum_{\substack{F_r \subseteq \llbracket 1, m_r \rrbracket \\  r = k, \dots, N }} \oint_{\Gamma} \frac{dz(Lz + N\theta - La_-)(Lz - M_L -1 + \theta - La_-)  }{2 \pi \i \theta (z - v)}   \\
& \times \prod_{i = k}^N \prod_{f \in F_i^c}\frac{1}{L(v_f^i-z)(v_f^i - z + L^{-1})}  \cdot  M\left( \prod_{p = 1}^{k}\frac{Lz- \tilde{\ell}^{k}_p + \theta - 1}{Lz - \tilde{\ell}^{k}_p - 1 }; F_k, \dots, F_N \right),
\end{split}
\end{equation}
\begin{equation}\label{ZOP4}
\begin{split}
& I_L^3  =  \sum_{\substack{F_r \subseteq \llbracket 1, m_r \rrbracket \\  r = k, \dots, N }}   \sum\limits_{j=k+1}^{N}  \prod_{i = k}^{j-1} {\bf 1 } \{F_i = \llbracket 1, m_i \rrbracket\} \oint_{\Gamma} dz  \frac{(Lz + N\theta - La_-)(Lz - M_L -1 + \theta - La_-)  }{2 \pi \i  (z - v)}  \\
&\times  \prod_{i = j}^N \prod_{f \in F_i^c}\frac{1}{L(v_f^i-z)(v_f^i - z + L^{-1})} \cdot  M \left( \Pi^{\theta, j}_1(Lz - La_-); F_k, \dots, F_N \right),
\end{split}
\end{equation}
where we recall that $\Pi^{\theta, j}_1$ is as in (\ref{S42E2}). In equations (\ref{ZOP2}), (\ref{ZOP3}) and (\ref{ZOP4}) we have written $M(\xi; F_k, \dots, F_N)$ in place of the joint cumulant $M(\xi; \{ G_L^j(v_i^j - a_-): j \in \{k, \dots, N\}, i \in F_j \})$ (with respect to the measure $\mathbb{P}_L$), where we recall that this notation was introduced just before Lemma \ref{LemCum1} and $G_L^j$ was defined in (\ref{S3DefG}). Also, we have $F_i^c= \llbracket 1, m_i \rrbracket \setminus F_i$. We mention that when we applied (\ref{Exp1V3}) we multiplied both sides by $L \theta^{-1}$, performed a change of variables $z \rightarrow z + a_-$ and set $\tilde{\ell}_i^j = \ell_i^j - L a_-$ to ease the notation.\\

{\bf \raggedleft Step 2.} In this step we summarize the large $L$ expansion of the various expressions in equations (\ref{ZOP2}), (\ref{ZOP3}) and (\ref{ZOP4}). In the remainder of the proof all constants, including those in big $O$ notations, will depend on $N, k, \theta, V, a_-, a_+, K$ and $\Gamma$. We will not mention this further. Below we use $\xi_L(z)$ to denote a generic random analytic function, which is $\mathbb{P}_L$-almost surely $O(1)$, and whose meaning will change from line to line.

We first observe that for each $x,y \in [-\theta - 1, \theta + 1]$, $z \in \Gamma$ and $j \in \llbracket k, N \rrbracket$ we have
\begin{equation*}
\begin{split}
& \prod_{ p = 1}^j   \frac{Lz - \tilde{\ell}_p^j  + x }{Lz - \tilde{\ell}_p^j + y} = \exp  \left(\sum_{p = 1}^j \log \left(  1 +  \frac{L^{-1} (x - y)}{z - \tilde{\ell}^j_p/L + y/L } \right) \right)  = \\
& \exp \left(  \frac{1}{L} \sum_{p = 1}^j \frac{x - y }{z - \tilde{\ell}^j_p/L + y/L }    - \frac{1}{2L^2}  \sum_{p = 1}^j \frac{(x - y)^2}{(z - \tilde{\ell}^j_p/L + y/L)^2}  +  \frac{\xi_L(z)}{L^3}  \right).
\end{split}
\end{equation*}
Recalling the definition of $G_L^j(z)$ from (\ref{S3DefG}) and setting $\tilde{G}_L^j(z) = {G}_L^j(z - a_-)$, we have
$$\tilde{G}_L^j(z) = \sum_{i = 1}^{j} \frac{1}{z - \tilde{\ell}_i^j/L}, \hspace{2mm} \partial_z \tilde{G}_L^j(z) = - \sum_{i = 1}^{j} \frac{1}{(z - \tilde{\ell}_i^N/L)^2}.$$
Combining the last two equations we get
\begin{equation*}
\begin{split}
&\prod_{ p = 1}^j   \frac{Lz - \tilde{\ell}_p^j  + x }{Lz - \tilde{\ell}_p^j + y} = \exp \left( \frac{x- y}{L}  \tilde{G}^j_L(z) + \frac{x^2 - y^2}{2L^2}  \partial_z \tilde{G}_L^j(z) + \frac{\xi_L(z)}{L^3}  \right).
\end{split}
\end{equation*}
Taylor expanding the above exponential function we obtain
\begin{equation}\label{YOP1}
\begin{split}
&\prod_{ p = 1}^j   \frac{Lz - \tilde{\ell}_p^j  + x }{Lz - \tilde{\ell}_p^j + y} =  1 + \frac{x- y}{L}  \tilde{G}^j_L(z) + \frac{x^2 - y^2}{2L^2}  \partial_z \tilde{G}_L^j(z) + \frac{(x-y)^2}{2L^2}[ \tilde{G}^j_L(z) ]^2 + \frac{\xi_L(z)}{L^3}.
\end{split}
\end{equation}

We also record the following expansions
\begin{equation}\label{YOP2}
e^{N\theta[V(z) - V(z - L^{-1}) ]} = 1 +  \frac{N \theta \partial_z V(z)}{L}  + O(L^{-2}),
\end{equation}
\begin{equation}\label{YOP3}
\frac{(Lz + N\theta - La_-)(Lz - M_L -1 + \theta - La_-) }{L^2} = (z-a_-)(z- a_+) + O(L^{-1}).
\end{equation}

{\bf \raggedleft Step 3.} We claim that 
\begin{equation}\label{ZOP5}
\begin{split}
&I^1_L =   I_L^{1,1} + I_L^{1,2} + O(L^{-1}) \mbox{, where } 
\end{split}
\end{equation}
\begin{equation*}
\begin{split}
I_L^{1,1} = &\hspace{2mm}   -   \oint_{\Gamma}\frac{dz(Lz + N\theta - La_-)(Lz - M_L -1 + \theta - La_-)e^{N\theta[ V(z) - V(z - L^{-1})]}   }{2 \pi \i  L (z - v)} \\
& \times M\left(   \tilde{G}^N_L(z)  ; \llbracket 1, m_k \rrbracket, \dots, \llbracket 1, m_N \rrbracket \right)
\end{split}
\end{equation*}
\begin{equation*}
\begin{split}
& I_L^{1,2} = \theta \oint_{\Gamma}\frac{dz(z-a_-)(z- a_+)  M\left(  \partial_z \tilde{G}_L^N(z) + [\tilde{G}^N_L(z) ]^2  ; \llbracket 1, m_k \rrbracket, \dots, \llbracket 1, m_N \rrbracket \right)   }{4 \pi \i  (z - v)};
\end{split}
\end{equation*}
\begin{equation}\label{ZOP6}
\begin{split}
& I^2_L =    I_L^{2,1} + I_L^{2,2} +  I_L^{2,3} + O(L^{-1})\mbox{, where }
\end{split}
\end{equation}
\begin{equation*}
\begin{split}
 I_L^{2,1} = \hspace{2mm} & \oint_{\Gamma} \frac{dz(Lz + N\theta - La_-)(Lz - M_L -1 + \theta - La_-) }{2 \pi \i  L (z - v)  }   M\left(  \tilde{G}^k_L(z); \llbracket 1, m_k \rrbracket, \dots, \llbracket 1, m_N \rrbracket   \right), 
\end{split}
\end{equation*}
\begin{equation*}
\begin{split}
&  I_L^{2,2} =  \oint_{\Gamma} \frac{dz(z-a_-)(z- a_+) M\left(   (\theta -2 ) \partial_z \tilde{G}_L^k(z) + \theta[ \tilde{G}^k_L(z) ]^2 ; \llbracket 1, m_k \rrbracket, \dots, \llbracket 1, m_N \rrbracket   \right) }{4 \pi \i  (z - v)  } 
\end{split}
\end{equation*}
\begin{equation*}
\begin{split}
 I_L^{2,3} = \hspace{2mm} &   \sum_{r = k}^N \sum_{f \in \llbracket 1, m_r \rrbracket }\oint_{\Gamma} \frac{dz(z-a_-)(z- a_+)  }{2 \pi \i  (z - v) (v_f^r-z)^2 }    \\
& \times M\left(   \tilde{G}^k_L(z); \llbracket 1, m_k \rrbracket, \dots, \llbracket 1, m_{r-1} \rrbracket, \llbracket 1, m_{r} \rrbracket \setminus \{f \}, \llbracket 1, m_{r+1} \rrbracket, \dots , \llbracket 1, m_N \rrbracket   \right);
\end{split}
\end{equation*}
\begin{equation}\label{ZOP7}
 I^3_L =  I_L^{3,1} + I_L^{3,2} +  I_L^{3,3} + O(L^{-1})\mbox{, where } .
\end{equation}
\begin{equation*}
\begin{split}
 I_L^{3,1} = \hspace{2mm} &    \sum\limits_{j=k+1}^{N} \oint_{\Gamma} \frac{dz(Lz + N\theta - La_-)(Lz - M_L -1 + \theta - La_-) }{2 \pi \i L (z - v)}  \\
&  \times  M \left( \tilde{G}^j_L(z) -   \tilde{G}^{j-1}_L(z) ; \llbracket 1, m_k \rrbracket, \dots, \llbracket 1, m_N \rrbracket \right) ,
\end{split}
\end{equation*}
\begin{equation*}
\begin{split}
&  I_L^{3,2} =     \sum\limits_{j=k+1}^{N}  \oint_{\Gamma}  \frac{dz(z-a_-)(z- a_+)  }{4 \pi \i  (z - v)}  \\
&\times M \left( -(\theta + 1)  \partial_z \tilde{G}_L^j(z) +   (1-\theta)[ \tilde{G}^j_L(z) - \tilde{G}^{j-1}_L(z) ]^2   +  (1 - \theta)  \partial_z \tilde{G}_L^{j-1}(z) ;  \llbracket 1, m_k \rrbracket, \dots, \llbracket 1, m_N \rrbracket \right),
\end{split}
\end{equation*}
\begin{equation*}
\begin{split}
 I_L^{3,3} = \hspace{2mm}&   \sum_{r = k}^N \sum_{f \in \llbracket 1, m_r\rrbracket }   \sum\limits_{j=k + 1}^{r}   \oint_{\Gamma}\frac{ dz(z-a_-)(z- a_+)  }{2 \pi \i  (z - v)(v_f^r-z)^2 }   \\
& \times   M \left(   \tilde{G}^j_L(z) - \tilde{G}^{j-1}_L(z) ; \llbracket 1, m_k \rrbracket, \dots, \llbracket 1, m_{r-1} \rrbracket, \llbracket 1, m_{r} \rrbracket \setminus \{f \}, \llbracket 1, m_{r+1} \rrbracket, \dots , \llbracket 1, m_N \rrbracket  \right).
\end{split}
\end{equation*}
We establish (\ref{ZOP5}), (\ref{ZOP6}) and (\ref{ZOP7}) in the steps below. Here we assume their validity and conclude the proof of the lemma.\\

We first note that by telescoping we have 
\begin{equation}\label{YOP4}
\begin{split}
&I_L^{1,1} + I_L^{2,1} + I_L^{3,1} =   \oint_{\Gamma} \frac{dz(Lz + N\theta - La_-)(Lz - M_L -1 + \theta - La_-) }{2 \pi \i L (z - v)}   \\
&  \times [ 1  -e^{N\theta[ V(z) - V(z - L^{-1})]}]  \cdot  M \left(  \tilde{G}^N_L(z) ; \llbracket 1, m_k \rrbracket, \dots, \llbracket 1, m_N \rrbracket \right)   \\
& = \oint_{\Gamma} dz \frac{ N \theta (z-a_-)(z- a_+) [- \partial_z V(z)] }{2 \pi \i (z - v)} \cdot   M \left(  \tilde{G}^N_L(z) ; \llbracket 1, m_k \rrbracket, \dots, \llbracket 1, m_N \rrbracket \right) + O(L^{-1}),
\end{split}
\end{equation}
where in the last equality we used (\ref{YOP2}) and (\ref{YOP3}).

We also have
\begin{equation}\label{YOP5}
\begin{split}
I_L^{1,2}  +  I_L^{2,2}   +  I_L^{3,2} = \hspace{2mm} &\oint_{\Gamma}   \frac{dz(z-a_-)(z- a_+)  }{4 \pi \i  (z - v)}  M\Bigg( \theta \partial_z \tilde{G}_L^N(z) + \theta [\tilde{G}^N_L(z) ]^2 +  (\theta -2 ) \partial_z \tilde{G}_L^k(z)  \\
& +  \theta [ \tilde{G}^k_L(z) ]^2 +    \sum\limits_{j=k+1}^{N}-(\theta + 1)  \partial_z \tilde{G}_L^j(z) +    (1-\theta)[ \tilde{G}^j_L(z) \\
& - \tilde{G}^{j-1}_L(z) ]^2   +  (1 - \theta)  \partial_z \tilde{G}_L^{j-1}(z)  ; \llbracket 1, m_k \rrbracket, \dots, \llbracket 1, m_N \rrbracket \Bigg) .
\end{split}
\end{equation}

Finally, we have by telescoping
\begin{equation}\label{YOP6}
\begin{split}
&I_L^{2,3} + I_L^{3,3} =  \sum_{r = k}^N \sum_{f \in \llbracket 1, m_r\rrbracket }   \oint_{\Gamma}\frac{ dz(z-a_-)(z- a_+)  }{2 \pi \i  (z - v)(v_f^r-z)^2 }  \\
& \times   M \left(   \tilde{G}^r_L(z) ; \llbracket 1, m_k \rrbracket, \dots, \llbracket 1, m_{r-1} \rrbracket, \llbracket 1, m_{r} \rrbracket \setminus \{f\}, \llbracket 1, m_{r+1} \rrbracket, \dots , \llbracket 1, m_N \rrbracket  \right).
\end{split}
\end{equation}

We now study the $L \rightarrow \infty$ limit of (\ref{YOP4}), (\ref{YOP5}) and (\ref{YOP6}). From Proposition \ref{prop_cont_limit} we know that $\left( L^{-1} \tilde{\ell}_{j-i+1}^j: j \in \llbracket k, N \rrbracket, i \in \llbracket 1, j \rrbracket \right)$ converge weakly as $L \rightarrow \infty$ to $\left( Y_i^j: j \in \llbracket k, N \rrbracket, i \in \llbracket 1, j \rrbracket \right)$ as in the statement of the lemma. In particular, any joint moment of $\tilde{G}^j_L(z_i)$'s and $\partial_{z}  \tilde{G}^j(z_i)$'s converge to the corresponding joint moment of $\mathcal{G}^j(z_i)$'s and $\partial_{z}  \mathcal{G}^j(z_i)$'s. From (\ref{Mal2}) we conclude the same is true for the joint cumulants. Taking the $L \rightarrow \infty$  limits (\ref{YOP4}), (\ref{YOP5}) and (\ref{YOP6}), and using (\ref{ZOP1}) we obtain (\ref{TLRankmn_intro}), which completes the proof of the theorem.\\

{\bf \raggedleft Step 4.} In this step we prove (\ref{ZOP5}) and (\ref{ZOP6}). Using (\ref{ZOP2}), (\ref{YOP1}) with $x = -\theta, y = 0$ we get 
\begin{equation}\label{ZOP8}
\begin{split}
&I_L^1 = O(L^{-1}) +  \oint_{\Gamma}\frac{dz(Lz + N\theta - La_-)(Lz - M_L -1 + \theta - La_-) e^{N\theta[V(z) - V(z - L^{-1}) ]} }{2 \pi \i \theta  (z - v)} \\
& \times M\left( 1+ \frac{(-\theta)}{L}  \tilde{G}^N_L(z) + \frac{\theta^2 }{2L^2}  \partial_z \tilde{G}_L^N(z) + \frac{\theta^2}{2L^2}[ \tilde{G}^N_L(z) ]^2  ; \llbracket 1, m_k \rrbracket, \dots, \llbracket 1, m_N \rrbracket \right).
\end{split}
\end{equation}
We can remove $1$ from the above cumulant provided that $m_k + \cdots + m_N \geq 1$ in view of (\ref{S3Linearity}), and if $m_k + \cdots + m_N = 0$, we have that we may again remove the $1$ as its contribution integrates to $0$ by Cauchy's theorem (here we used that $v$ and $v_f^i$ all lie outside of $\Gamma$). Removing $1$ from (\ref{ZOP8}), using  (\ref{YOP2}), (\ref{YOP3}) and the linearity of cumulants, see (\ref{S3Linearity}), we obtain (\ref{ZOP5}). In the remainder of this step we focus on establishing (\ref{ZOP6}).\\

Using (\ref{ZOP3}), (\ref{YOP1}) with $x = \theta - 1, y = -1$ we get 
\begin{equation}\label{ZOP9}
\begin{split}
&I_L^2 =  \sum_{\substack{F_r \subseteq \llbracket 1, m_r \rrbracket \\  r = k, \dots, N }} \oint_{\Gamma} \frac{dz(Lz + N\theta - La_-)(Lz - M_L -1 + \theta - La_-)  }{2 \pi \i \theta (z - v)}      \\
& \times \frac{M\left( 1 + \frac{\theta}{L}  \tilde{G}^k_L(z) + \frac{\theta (\theta -2 )}{2L^2}  \partial_z \tilde{G}_L^k(z) + \frac{\theta^2}{2L^2}[ \tilde{G}^k_L(z) ]^2 ; F_k, \dots, F_N \right)}{ \prod_{i = k}^N \prod_{f \in F_i^c} L(v_f^i-z)(v_f^i - z + L^{-1})} + O(L^{-1}).
\end{split}
\end{equation}

As before we can remove $1$ from the above cumulant provided that $|F_k| + \cdots + |F_N| \geq 1$, and if $|F_k| + \cdots + |F_N| = 0$, we have that we may again remove the $1$ as its contribution integrates to $0$ by Cauchy's theorem (here we used that $v$ and $v_f^i$ all lie outside of $\Gamma$).

We next note that if $|F_k| + \cdots + |F_N| < m_1 + \cdots + m_N$ (i.e. $F_j \neq \llbracket 1, m_j \rrbracket$ for some $j \in \llbracket k, N \rrbracket$)
$$\frac{M\left(  \frac{\theta (\theta -2 )}{2L^2}  \partial_z \tilde{G}_L^k(z) + \frac{\theta^2}{2L^2}[ \tilde{G}^k_L(z) ]^2 ; F_k, \dots, F_N \right)}{ \prod_{i = k}^N \prod_{f \in F_i^c} L(v_f^i-z)(v_f^i - z + L^{-1})} = O(L^{-3}),$$
and also if $|F_k| + \cdots + |F_N| < m_1 + \cdots + m_N - 1$
$$\frac{M\left(   \frac{\theta}{L}  \tilde{G}^k_L(z)   ; F_k, \dots, F_N \right)}{ \prod_{i = k}^N \prod_{f \in F_i^c} L(v_f^i-z)(v_f^i - z + L^{-1})} = O(L^{-3}).$$

Combining the observations from the last two paragraphs with (\ref{ZOP9}) gives
\begin{equation*}
\begin{split}
&I_L^2 =   O(L^{-1}) + \oint_{\Gamma} \frac{dz(Lz + N\theta - La_-)(Lz - M_L -1 + \theta - La_-)  }{2 \pi \i \theta (z - v) L^2 }      \\
&\times M\left(  \frac{\theta (\theta -2 )}{2}  \partial_z \tilde{G}_L^k(z) + \frac{\theta^2}{2}[ \tilde{G}^k_L(z) ]^2 ; \llbracket 1, m_k \rrbracket, \dots, \llbracket 1, m_N \rrbracket   \right)\\
& + \sum_{r = k}^N \sum_{f \in \llbracket 1, m_r \rrbracket }\oint_{\Gamma} \frac{dz(Lz + N\theta - La_-)(Lz - M_L -1 + \theta - La_-)  }{2 \pi \i  \theta (z - v) (v_f^r-z)(v_f^r - z + L^{-1}) L^2 }     \\
& \times M\left(  \theta \tilde{G}^k_L(z); \llbracket 1, m_k \rrbracket, \dots, \llbracket 1, m_{r-1} \rrbracket, \llbracket 1, m_{r} \rrbracket \setminus \{i\}, \llbracket 1, m_{r+1} \rrbracket, \dots , \llbracket 1, m_N \rrbracket   \right)\\
& + \oint_{\Gamma} \frac{dz(Lz + N\theta - La_-)(Lz - M_L -1 + \theta - La_-)  }{2 \pi \i \theta (z - v)  L}   M\left( \theta \tilde{G}^k_L(z); \llbracket 1, m_k \rrbracket, \dots, \llbracket 1, m_N \rrbracket   \right) .
\end{split}
\end{equation*}
Combining the last equation with (\ref{YOP3}) we get (\ref{ZOP6}).\\

{\bf \raggedleft Step 5.} In this step we prove (\ref{ZOP7}). We first observe that 
\begin{equation}\label{JK1}
\begin{split}
&\Pi_1^{\theta,j}(Lz - La_-) = \frac{\xi_L(z)}{L^3} + \frac{{\bf 1}\{ \theta \neq 1\}}{1 - \theta} +   \frac{\theta}{L}  \tilde{G}^j_L(z) - \frac{\theta (\theta + 1)}{2L^2}  \partial_z \tilde{G}_L^j(z) + \frac{\theta (1-\theta)}{2L^2}[ \tilde{G}^j_L(z) ]^2    \\
&- \frac{\theta}{L}  \tilde{G}^{j-1}_L(z) + \frac{\theta (1 - \theta)}{2L^2}  \partial_z \tilde{G}_L^{j-1}(z) + \frac{\theta (1-\theta)}{2L^2}[ \tilde{G}^{j-1}_L(z) ]^2 - \frac{\theta (1 - \theta)}{L^2}\tilde{G}^j_L(z)\tilde{G}^{j-1}_L(z) .
\end{split}
\end{equation}
When $\theta \neq 1$ the latter follows from applying (\ref{YOP1}) twice with $x = -\theta, y = -1$ and $x = \theta -1, y = 0$ to the two products in the definition of $\Pi_1^{\theta,j}$ in (\ref{S42E2}). When $\theta = 1$ we have from (\ref{S42E2})
\begin{equation*}
\begin{split}
\Pi_1^{1,j}(Lz - La_-)  = \sum_{i = 1}^j \frac{1}{Lz - \tilde{\ell}_i^j- 1} - \sum_{i = 1}^{j-1} \frac{1}{Lz - \tilde{\ell}_i^{j-1}} = \frac{1}{L} \tilde{G}^j_L(z) - \frac{1}{L^2} \partial_z \tilde{G}^j_L(z) - \frac{1}{L} \tilde{G}^{j-1}_L(z) + \frac{\xi_L(z)}{L^3},
\end{split}
\end{equation*} 
which agrees with (\ref{JK1}).

Substituting (\ref{JK1}) into (\ref{ZOP4}) we get
\begin{equation}\label{ZOP10}
\begin{split}
& I_L^3  =  O(L^{-1}) +    \sum_{\substack{F_r \subseteq \llbracket 1, m_r \rrbracket \\  r = k, \dots, N }}   \sum\limits_{j=k+1}^{N}  \prod_{i = k}^{j-1} {\bf 1 } \{F_i = \llbracket 1, m_i \rrbracket\}    \\
& \times \oint_{\Gamma} dz \frac{(Lz + N\theta - La_-)(Lz - M_L -1 + \theta - La_-)  }{2 \pi \i \theta (z - v)}   \prod_{i = j}^N \prod_{f \in F_i^c}\frac{1}{L(v_f^i-z)(v_f^i - z + L^{-1})} \\
&\times   M \Bigg{(} \frac{{\bf 1}\{ \theta \neq 1\}}{1 - \theta} +   \frac{\theta}{L}  \tilde{G}^j_L(z) - \frac{\theta (\theta + 1)}{2L^2}  \partial_z \tilde{G}_L^j(z) + \frac{\theta (1-\theta)}{2L^2}[ \tilde{G}^j_L(z) ]^2 - \frac{\theta}{L}  \tilde{G}^{j-1}_L(z)   \\
& + \frac{\theta (1 - \theta)}{2L^2}  \partial_z \tilde{G}_L^{j-1}(z) + \frac{\theta (1-\theta)}{2L^2}[ \tilde{G}^{j-1}_L(z) ]^2 - \frac{\theta (1 - \theta)}{L^2}\tilde{G}^j_L(z)\tilde{G}^{j-1}_L(z)  ; F_k, \dots, F_N \Bigg{)}.
\end{split}
\end{equation}

The same argument that allowed us to remove $1$ from (\ref{ZOP8}) in Step 4 applies here and allows us to remove $\frac{{\bf 1}\{ \theta \neq 1\}}{1 - \theta}$ from (\ref{ZOP10}). We next note that if $|F_k| + \cdots + |F_N| < m_1 + \cdots + m_N$ (i.e. $F_j \neq \llbracket 1, m_j \rrbracket$ for some $j \in \llbracket k, N \rrbracket$) 
\begin{equation*}
\begin{split}
&   \prod_{i = k}^{j-1} {\bf 1 } \{F_i = \llbracket 1, m_i \rrbracket\} \prod_{i = j}^N \prod_{f \in F_i^c}\frac{1}{L(v_f^i-z)(v_f^i - z + L^{-1})} \times   M \Bigg{(}  - \frac{\theta (\theta + 1)}{2L^2}  \partial_z \tilde{G}_L^j(z) + \frac{\theta (1-\theta)}{2L^2}[ \tilde{G}^j_L(z) ]^2   \\
& + \frac{\theta (1 - \theta)}{2L^2}  \partial_z \tilde{G}_L^{j-1}(z) + \frac{\theta (1-\theta)}{2L^2}[ \tilde{G}^{j-1}_L(z) ]^2 - \frac{\theta (1 - \theta)}{L^2}\tilde{G}^j_L(z)\tilde{G}^{j-1}_L(z)   ; F_k, \dots, F_N \Bigg{)} = O(L^{-3}).
\end{split}
\end{equation*}
Also if $|F_k| + \cdots + |F_N| < m_1 + \cdots + m_N - 1$
\begin{equation*}
\begin{split}
&O(L^{-3}) =   \prod_{i = k}^{j-1} {\bf 1 } \{F_i = \llbracket 1, m_i \rrbracket\} \prod_{i = j}^N \prod_{f \in F_i^c}\frac{1}{L(v_f^i-z)(v_f^i - z + L^{-1})}  \\
&  \times M \left(\frac{\theta}{L}  \tilde{G}^j_L(z)  -  \frac{\theta}{L}  \tilde{G}^{j-1}_L(z); F_k, \dots, F_N  \right).
\end{split}
\end{equation*}

Combining the observations from the last paragraph with (\ref{ZOP10}) gives
\begin{equation*}
\begin{split}
& I_L^3  =  O(L^{-1}) + \tilde{I}_L^{3,1}+ \tilde{I}_L^{3,2} + \tilde{I}_L^{3,3}, \mbox{ where }
\end{split}
\end{equation*}
\begin{equation*}
\begin{split}
&\tilde{I}_L^{3,1} =      \sum\limits_{j=k+1}^{N}  \oint_{\Gamma} \frac{dz(Lz + N\theta - La_-)(Lz - M_L -1 + \theta - La_-)  }{2 \pi \i L (z - v)}  \\
&  \times  M \left( \tilde{G}^j_L(z) -   \tilde{G}^{j-1}_L(z) ; \llbracket 1, m_k \rrbracket, \dots, \llbracket 1, m_N \rrbracket \right).
\end{split}
\end{equation*}
\begin{equation*}
\begin{split}
&\tilde{I}_L^{3,2} =      \sum\limits_{j=k+1}^{N}  \oint_{\Gamma}  \frac{dz(Lz + N\theta - La_-)(Lz - M_L -1 + \theta - La_-) }{4 \pi \i L^2 (z - v)} \\
&\times M \left( -(\theta + 1)  \partial_z \tilde{G}_L^j(z) +   (1-\theta)[ \tilde{G}^j_L(z) - \tilde{G}^{j-1}_L(z) ]^2   +  (1 - \theta)  \partial_z \tilde{G}_L^{j-1}(z) ;  \llbracket 1, m_k \rrbracket, \dots, \llbracket 1, m_N \rrbracket \right) ,
\end{split}
\end{equation*}
\begin{equation*}
\begin{split}
&  \tilde{I}_L^{3,3} =  \sum_{r = k}^N \sum_{f \in \llbracket 1, m_r\rrbracket }   \sum\limits_{j=k + 1}^{r}   \oint_{\Gamma}\frac{ dz (Lz + N\theta - La_-)(Lz - M_L -1 + \theta - La_-)  }{2 \pi \i  L^2 (z - v)(v_f^r-z)(v_f^r- z + L^{-1}) }  \\
& \times  M \left(   \tilde{G}^j_L(z) - \tilde{G}^{j-1}_L(z) ; \llbracket 1, m_k \rrbracket, \dots, \llbracket 1, m_{r-1} \rrbracket, \llbracket 1, m_{r} \rrbracket \setminus \{f\}, \llbracket 1, m_{r+1} \rrbracket, \dots , \llbracket 1, m_N \rrbracket  \right).
\end{split}
\end{equation*}
Combining the last four equations with (\ref{YOP3}) gives (\ref{ZOP7}).

\bibliographystyle{alpha}
\bibliography{PD}

\end{document}